\documentclass[12pt,oneside,reqno]{amsart}
\usepackage{amsfonts}
\usepackage {amssymb}
\usepackage {amsmath}
\usepackage {amsmath}
\usepackage {amsthm}
\usepackage{graphicx}
\usepackage{multirow}
\usepackage{xypic}
\usepackage {amscd}
\usepackage{mathrsfs}
\usepackage[colorlinks, linkcolor=blue, anchorcolor=black, citecolor=red]{hyperref}
\usepackage{enumerate}

\usepackage[letterpaper]{geometry}
\newcommand{\Fut}{{\rm Fut}}

\newcommand{\sddb}{{\sqrt{-1}\partial\bar{\partial}}}
\newcommand{\ddb}{{\partial\bar{\partial}}}

\newcommand{\mcY}{{\mathcal{Y}}}
\newcommand{\mcL}{{\mathcal{L}}}
\newcommand{\mcX}{{\mathcal{X}}}
\newcommand{\mcD}{{\mathcal{D}}}

\newcommand{\vol}{{\rm vol}}
\newcommand{\ord}{{\rm ord}}
\newcommand{\lct}{{\rm lct}}
\newcommand{\vphi}{\varphi}

\newcommand{\bC}{{\mathbb{C}}}

\newcommand{\FS}{{\rm FS}}
\newcommand{\bP}{{\mathbb{P}}}

\newcommand{\sslash}{{/\!/}}
\newcommand{\bG}{\mathbb{G}}

\newcommand{\bT}{\mathbb{T}}

\newcommand{\NA}{{\rm NA}}

\newcommand{\mcJ}{{\mathcal{J}}}

\newcommand{\bQ}{{\mathbb{Q}}}

\newcommand{\mcB}{{\mathcal{B}}}
\newcommand{\bR}{{\mathbb{R}}}

\newcommand{\mcH}{{\mathcal{H}}}
\newcommand{\mcO}{{\mathcal{O}}}
\newcommand{\mcF}{{\mathcal{F}}}
\newcommand{\bZ}{{\mathbb{Z}}}

\newcommand{\bN}{{\mathbb{N}}}

\newcommand{\la}{\langle}
\newcommand{\ra}{\rangle}
\newcommand{\mcE}{\mathcal{E}}
\newcommand{\bfL}{{\bf L}}
\newcommand{\mcI}{\mathcal{I}}

\newcommand{\DHM}{{\rm DH}}
\newcommand{\mcR}{\mathcal{R}}

\newcommand{\wt}{{\rm wt}}

\newcommand{\bfD}{{\bf D}}

\newcommand{\bfM}{{\bf M}}
\newcommand{\bfH}{{\bf H}}
\newcommand{\bfE}{{\bf E}}
\newcommand{\bfJ}{{\bf J}}

\newcommand{\bfF}{{\bf F}}

\newcommand{\Aut}{{\rm Aut}}
\newcommand{\PSH}{{\rm PSH}}

\newcommand{\Lam}{{\bf \Lambda}}
\newcommand{\cF}{\mathcal{F}}
\newcommand{\bfI}{{\bf I}}

\newcommand{\bB}{\mathbb{B}}

\newcommand{\cO}{\mathcal{O}}

\newcommand{\cE}{\mathcal{E}}

\newcommand{\cJ}{\mathcal{J}}

\newcommand{\bfm}{{\bf m}}
\newcommand{\bm}{{\bf m}}
\newcommand{\bV}{{\bf V}}
\newcommand{\bS}{\mathbb{S}}

\newcommand{\aut}{{\rm aut}}

\newtheorem{thm}{Theorem}[section]
\newtheorem{prop}[thm]{Proposition}
\newtheorem{defn}[thm]{Definition}

\newtheorem{cor}[thm]{Corollary}
\newtheorem{rem}[thm]{Remark}

\newtheorem{exmp}[thm]{Example}
\newtheorem{lem}[thm]{Lemma}

\newtheorem{defn-prop}[thm]{Definition-Proposition}

\makeatletter
\renewcommand{\@secnumfont}{\bfseries}
\makeatother

\begin{document}

\title[Weighted K\"{a}hler-Ricci solitons and Ricci-flat K\"{a}hler cone]{Notes on weighted K\"{a}hler-Ricci solitons and application to Ricci-flat K\"{a}hler cone metrics}
\author{Chi Li}
\thanks{The author is partially supported by an NSF grant (DMS-1810867) and a Sloan fellowship. He would like to thank J. Han for joint work on studying $g$-soliton equations, V. Apostolov and G. Tian for their interests and helpful comments. He also thanks C. Xu for notifying us of an ongoing related work of Kai Huang. }
\maketitle{}

\begin{abstract}
This is largely an exposition article that expands the author's talk at the Xiamen International Conference on Geometric Analysis in June 2021. We first survey the author's joint work with Jiyuan Han on the Yau-Tian-Donaldson (YTD) conjecture for $g$-weighted K\"{a}hler-Ricci solitons (or $g$-solitons). We then review recent works of Apostolov-Canderbank-Jubert-Lahdili which establish a connection between a particular $g$-soliton equation with Ricci-flat K\"{a}hler cone metrics (or equivalently Sasaki-Einstein metrics). The main interest in this connection is the transformation of a possibly irregular Sasaki-Einstein metric to a particular $g$-soliton equation on any quasi-regular quotient. We will revisit this transformation by understanding how the corresponding transversal complex Monge-Amp\`{e}re equations are transformed under the deformation of Reeb vector fields. Finally we explain how this PDE/pluripotential point of view allows one to combine the version of YTD  conjecture for $g$-solitons on log Fano pairs proved in \cite{HL20}, the algebraic results from \cite{BLXZ, LXZ21} and the discovery in \cite{AC21, AJL21} to prove the YTD conjecture for general Fano cones. 
\end{abstract}

\tableofcontents

\vskip 1cm

\section{Yau-Tian-Donaldson conjecture for $g$-soliton equations}

In the first part, we first review the pluripotential theory for $g$-soliton equations as developed in \cite{BW14, HL20}. Our main contribution in this aspect is to use a fibration construction to study and extend the previous well-established theory to the $g$-soliton setting. We then sketch the proof of the Yau-Tian-Donaldson conjecture for $g$-solitons on general log Fano pairs. 
\subsection{$g$-soliton equations}

Let $X$ be an $n$-dimensional Fano manifold. In other words, $X$ is a projective manifold satisfying that $-K_X:=\wedge^n T^{hol}X$ is an ample line bundle. 
In particular, $X$ is a K\"{a}hler manifold. 
For simplicity of notation, we also denote $-K_X$ by $L$. Fix a reference Hermitian metric $h_0$ on $L$ such that its Chern curvature $-\sddb \log h_0=\omega_0$ is a K\"{a}hler form. Any other Hermitian metric metric on $L$ is of the form $h_\vphi:=h_0 e^{-\vphi}$ for some function $\vphi$ on $X$. Its Chern curvature is equal to $-\sddb \log h_\vphi=\omega_0+\sddb \vphi=:\omega_\vphi$. 
$\omega_{\vphi_1}=\omega_{\vphi_2}$ iff $\sddb (\vphi_2-\vphi_1)=0$ iff $\vphi_2-\vphi_1$ is a real constant. 

A function $\vphi$ is called $\omega_0$-plurisubharmonic if $\vphi$ is upper semicontinuous and $\vphi+\psi$ is plurisubharmonic 
for any local potential $\psi$ of $\omega_0$. The set of all $\omega_0$-psh functions will be denoted by $\PSH(\omega_0)$. 
The space of smooth strictly $\omega_0$-psh functions (or called K\"{a}hler potentials) will be denoted by
\begin{equation}
\mcH:=\mcH(L)=\mcH(X, L)=\left\{\vphi\in C^\infty(X);  \omega_0+\sddb\vphi>0\right\}.
\end{equation}

Any Hermitian metric $h_\vphi$ can be considered as a volume form as follows. Choose any local holomorphic coordinate chart $\{z_1, \dots, z_n\}$. Set $s=dz_1\wedge \cdots \wedge dz_n$, $s^*=\partial_{z_1}\wedge \cdots \wedge \partial_{z_n}\in -K_X$. It is easy to verify the following volume does not depend on the choice of holomorphic coordinates:
  \begin{equation}
\Omega_\vphi:=|s^*|^2_{h_\vphi} (\sqrt{-1})^{n^2} s\wedge \bar{s}=\Omega_0 e^{-\vphi}.
 \end{equation}

For any holomorphic vector field $v$ on $X$, denote by $\mathfrak{L}_v$ the Lie derivative with respect to $v$ and set
\begin{equation}
\theta_v(\vphi):=-\frac{\mathfrak{L}_{{v}} \Omega_\vphi}{\Omega_\vphi}=-\frac{\mathfrak{L}_v\Omega_0}{\Omega_0}+v(\vphi)=:\theta_{v,0}+v(\vphi)
\end{equation}
where $\theta_{v,0}=\theta_v(0)$. 
Then $\theta_v=\theta_v(\vphi)$ satisfies the identity 
\begin{equation}\label{eq-thetav}
\sqrt{-1}\bar{\partial}\theta_v(\vphi)=\iota_v (\omega_0+\sddb\vphi)=\iota_v\omega_\vphi
\end{equation} 
where $\iota_v$ is the contraction with respect to $v$. 
In local coordinates this is saying that $v^i (\omega_\vphi)_{i\bar{j}}=\partial_{\bar{z}_j} \theta_v$. 
If $Im(v)$ is furthermore a Killing vector field with respect to $\omega=\omega_\vphi$, then $\theta_v$ is a real valued function and the identity \eqref{eq-thetav} is equivalent to the identity 
\begin{equation}
\iota_{2 Im(v)}\omega=-\iota_{\xi}\omega=d \theta_v
\end{equation} 
where $\xi=-2 Im(v)$. Sometimes we will also write $\theta_v$ as $\theta_\xi$.  

Let $T$ denote the $r$-dimensional real torus $(S^1)^r$ and let $\bT\cong (\bC^*)^r$ be its complexification. 
Assume that $\bT\cong (\bC^*)^r$ acts on $X$ biholomorphically. There is a canonical action of $\bT$ on $-K_X$ induced by the pushforward of tangent vectors. 
Denote by $\mcH(-K_X)^T$ the set of $T$-invariant metrics in $\mcH(-K_X)$. Fix $h_0e^{-\vphi}\in \mcH(-K_X)^T$. 

Denote by $N_\bR\cong \bR^r$ the Lie algebra of $(S^1)^r$, and by $M_\bR\cong \bR^r$ the dual of $N_\bR$. 
For any $\kappa\in \{1, \dots, r\}$, let $\xi_\kappa\in N_\bR$ be the standard generator of the $\kappa$-th factor of $(S^1)^r$. Then $\xi_\kappa$ is a Killing vector field. Denote $v_\kappa=\frac{1}{2}(-J\xi_\kappa-\sqrt{-1}\xi_\kappa)$. By the above discussion, there exists real valued function $\theta_\kappa$ satisfying $\iota_{\xi_\kappa}\omega_\vphi=d\theta_\kappa(\vphi)$. 
If we set $\theta_\kappa=\theta_{v_\kappa,0}$, then 
\begin{equation}
\theta_\kappa(\vphi)=\theta_{\kappa,0}+v_\kappa(\vphi)=\theta_{\kappa, 0}-\frac{1}{2}d\vphi(J\xi_k).
\end{equation} 
The action of $(S^1)^r$ on $(X, \omega_\vphi)$ becomes Hamiltonian and the associated moment map is given by:
\begin{eqnarray*}
\bfm_\vphi: X&\longrightarrow& M_\bR\cong \bR^r \\
\bfm_\vphi(z)&=&(\theta_1(\vphi), \cdots, \theta_r(\vphi))=\left(\theta_\kappa(\vphi)\right)_{\kappa\in \{1,\dots, r\}}.
\end{eqnarray*}
We will denote by $P$ the image of $\bm_\vphi$. By Atiyah-Guillemin-Sternberg theorem, $P$ is a convex polytope which does not depend on the choice of $\omega_\vphi=\omega_0+\sddb \vphi>0$. Moreover, the pushforward measure $(\bm_\vphi)_*\omega_\vphi^n$ is a Radon measure that does not depend on the choice of $\vphi\in \mcH(L)$. We will call it the Duistermaat-Heckmann measure and denote it by $\DHM_T(L)$. 


Let $g: P\rightarrow \bR$ be a smooth {\it positive} function which we can assume to the restriction of a smooth function on $M_\bR$. We will be interested in the following equation for $\vphi\in \mcH(-K_X)^T$ which we call the $g$-weighted K\"{a}hler-Ricci (KR) soliton \footnote{It was called generalized K\"{a}hler-Ricci soliton in \cite{HL20}}, or just $g$-soliton equation:
\begin{equation}\label{eq-gsoliton}
g(\bfm_\vphi) (\omega_0+\sddb\vphi)^n= e^{-\vphi}\Omega_0.
\end{equation}
The integral of the left-hand-side does not depend on $\vphi\in \mcH(L)$:
\begin{equation}\label{eq-bVg}
\bV_g:=\int_X g(\bm_\vphi) (\omega_0+\sddb\vphi)^n=\int_P g(x) \DHM_T(L).
\end{equation}

For simplicity of notation, we will write $g_\vphi=g(\bfm_\vphi)=g(\theta_1(\vphi), \cdots, \theta_r(\vphi))$. It is also convenient to introduce 
\begin{equation}
f_\vphi=\log g_\vphi, \quad 
H_\vphi=\log \frac{\Omega_\vphi}{\omega_\vphi^n}.
\end{equation}
Then \eqref{eq-gsoliton} can be re-written as the equation:
\begin{equation}\label{eq-gsoliton2}
\log \frac{\Omega_\vphi}{(\omega_0+\sddb\vphi)^n}-\log g_\vphi=\text{ constant}, \quad \text{i.e.} \quad H_\vphi-f_\vphi=\text{ constant }.
\end{equation}
It will be useful to consider the following equivalent form of the equation \eqref{eq-gsoliton2}:
\begin{equation}\label{eq-gsoliton3} 
(\Delta_\vphi+v_{f, \vphi}) (H_\vphi-f_\vphi)=0
\end{equation}
where the $\Delta_\vphi=\omega_\vphi^{i\bar{j}}\partial_i\partial_{\bar{j}}$ is the Laplace operator of $\omega_\vphi$ and $v_{f,\vphi}$ is given by:
\begin{eqnarray}\label{eq-vfvphi}
v_{f, \vphi}&=&\frac{\partial f_\vphi}{\partial \bar{z}_j} \omega_\vphi^{i\bar{j}}\partial_i=\sum_\kappa \frac{\partial f}{\partial \theta_\kappa} \frac{\partial \theta_\kappa(\vphi)}{\partial \bar{z}_j}\omega_\vphi^{i\bar{j}}\partial_i=\sum_\kappa f_\kappa(\bfm_\vphi) v_\kappa.
\end{eqnarray}

Note also that the equation \eqref{eq-gsoliton} is equivalent to the tensorial equation:
\begin{equation}
Ric(\omega_\vphi)-\omega_\vphi=\sddb \log g_\vphi.
\end{equation}
 
The $g$-soliton equation, or more generally, the weighted extremal metrics in an equivariant setting, seems to first appear in Tian's work in \cite{Tia05} generalizing the K\"{a}hler-Ricci soliton equation. There are many following works in related equations, see for example \cite{Nak11, JL19}). Our work, which deals with the  general $g$-soliton equation, is partly inspired by the work of Berman-Witt-Nystr\"{o}m in \cite{BW14} which studies the K\"{a}hler-Ricci soliton equation from the variational point of view. 

\begin{rem}
The study of $g$-soliton equations can be put in a more general framework of weighted extremal metrics introduced by Lahdili (\cite{Lahdili19, Lahdili20}) (see also \cite{Ino19b, Ino20}).
However, the complete existence theory in particular the Yau-Tian-Donaldson conjecture seems to be established only for general $g$-solitons. See also \cite{AJL21}.
\end{rem}

 The following examples of $g$-soliton equations show that the general equation \eqref{eq-gsoliton} include several interesting classes of (canonical) K\"{a}hler metrics on Fano manifolds.  Our discussion in this article will mainly be around the Yau-Tian-Donaldson conjecture for these metrics.

\begin{enumerate}
\item ({\bf K\"{a}hler-Einstein metrics}) $T=\{e\}$ and $g=1$. This is the K\"{a}hler-Einstein case and is well-understood thanks to the works of Tian, Berman, Chen-Donaldson-Sun and others (see \cite{Tia97, Berm15, CDS15, Tia15}). 
\item ({\bf K\"{a}hler-Ricci soliton}) 
$\vphi\in \mcH(-K_X)$ is a (shrinking) K\"{a}hler-Ricci soliton on $(X, \bT)$ if there exists a holomorphic vector field $v=v_\xi$ with $\xi\in N_\bR$ such that $\omega_\vphi$ satisfies 
\begin{eqnarray*}
Ric(\omega_\vphi)-\omega_\vphi=\mathfrak{L}_v \omega_\vphi.
\end{eqnarray*}
It is well-known that KR solitons give rise to self-similar solutions to the K\"{a}hler-Ricci flow. 
By using $\mathfrak{L}_{v}\omega_\vphi=d\iota_v \omega_\vphi=\sddb\theta_v(\vphi)$, the above equation is equivalent to 
 the following equation
\begin{equation}  
(\omega_0+\sddb \vphi)^n=e^{\theta_v(\vphi)-\vphi}\Omega_0. 
\end{equation}
This is a special case of \eqref{eq-gsoliton} when $g(x)=e^{\la x, \xi\ra}=e^{\ell_{\xi}(x)}$ where $v=v_\xi$ and for simplicity of notation, we set:
\begin{equation}\label{eq-ellxi}
\ell_\xi(x)=\la x, \xi\ra. 
\end{equation} 
The K\"{a}hler-Ricci soliton equation was studied extensively in works of Tian-Zhu and others (see \cite{TZ99, TZ02, CTZ05}).  

\item ({\bf Mabuchi soliton})
$\vphi$ is called a Mabuchi soliton if there exists a holomorphic vector field $v=v_\xi$ with $\xi\in N_\bR$ such that
\begin{eqnarray}
&&Ric(\omega_\vphi)-\omega_\vphi=\sddb \log (1+\theta_v(\vphi)-\underline{\theta}_v)\nonumber \\
&&\hskip 3cm \Longleftrightarrow \left(1+\theta_v(\vphi)-\underline{\theta}_v\right)(\omega_0+\sddb\vphi)^n=e^{-\vphi}\Omega_0.\label{eq-Mabsoliton}
\end{eqnarray}
where $\underline{\theta}_v=\frac{1}{\int_X \omega^n}\int_X \theta_v\omega^n$. It is known that solution $\vphi$ to this equation is the critical point to the Ricci-Calabi functional (see \cite{Yao17})
\begin{equation*}
\vphi \mapsto \int_X (e^{\bar{H}_\vphi}-1)^2\omega_\vphi^n
\end{equation*}
where $\bar{H}_\vphi=H_\vphi-\log (\frac{1}{\bV_1}\int_X e^{H_\vphi}\omega_\vphi^n)$. The equation \eqref{eq-Mabsoliton} is a special case of \eqref{eq-gsoliton} when 
$g(x)=1+\la x-\underline{x}, \xi\ra$ where $\underline{x}=\frac{1}{\bV_1}\int_P x\, \DHM_T(-K_X)$ is the barycenter of the Duistermaat-Heckman measure $\DHM_T(L)$ over $P$. 

The Mabuchi soliton equation was introduced in \cite{Mab01, Mab03} and has been recently further studied in \cite{Yao17} for the toric case which motivates many other works (see \cite{His19, LZ19, Nak19, Yao19}). 
\vskip 1mm
\item ({\bf Ricci-flat K\"{a}hler cone metric}) This quite un-expected example is due to the works in \cite{AC21, AJL21}. 
Because $-K_X$ is ample, there is an affine cone over $X$ given by $Y:={\rm Spec}\left(\bigoplus_m H^0(X, -mK_X)\right)$. We will explain in section \ref{sec-orbMA} the existence of a Ricci-flat K\"{a}hler cone metric on the cone $Y$ with the Reeb vector field induced by $\xi\in N_\bR$ is equivalent to the solvability to the following equation:
\begin{eqnarray*}
&&Ric(\omega_\vphi)-\omega_\vphi=-(n+2)\sddb \log (n+1+\theta_v(\vphi))\\
&&\hskip 4cm \Longleftrightarrow \frac{1}{(n+1+\theta_v(\vphi))^{n+2}} (\omega_0+\sddb\vphi)^n=e^{-\vphi}\Omega_0
\end{eqnarray*}
where $v=v_{\xi}=\frac{1}{2}(-J \xi-\sqrt{-1}\xi)$ with $\xi\in N_\bR$. 
This is \eqref{eq-gsoliton} when 
\begin{equation}
g(x)=\frac{1}{(n+1+\la x, \xi\ra)^{n+2}}=\frac{1}{(n+1+\ell_\xi(x))^{n+2}}.
\end{equation}  
The constant $n+1$ comes from a natural normalization, and, by rescaling (which corresponds to adding a constant to $\vphi$), it can be changed to other positive constant without affecting the solvability of the equation. 

\end{enumerate}

\subsection{Generalized Futaki invariant and Matsushima type result}

Fox any $\vphi\in \mcH(-K_X)^T$, set $H_\vphi=\log \frac{\Omega_\vphi}{(\omega_0+\sddb\vphi)^n}$ which satisfies the identity $Ric(\omega_\vphi)-\omega_\vphi=\sddb H_\vphi$. For any holomorphic vector field $v$, define the $g$-weighted Futaki invariant as:
\begin{eqnarray}
\Fut_g(v)&=&\int_X v(H_\vphi-\log g_\vphi) g_\vphi \omega_\vphi^n.
\end{eqnarray}
Note that in the case of K\"{a}hler-Ricci solitons this is nothing but Tian-Zhu's modified Futaki invariant.
The basic result is then:
\begin{thm}
$\Fut_g$ does not depend on the choice of $\vphi\in \mcH(-K_X)^T$. Moreover, if there is a $g$-soliton, then $\Fut_g(v)=0$ for any holomorphic vector field $v$. 
\end{thm}

By appropriate integration by parts, Futaki invariant can written into two useful forms:\\
{\bf (1):} Recall that $\mathfrak{L}_ve^{-\vphi}$ is the Lie derivative of the volume form $e^{-\vphi}$ with respect to $v$. We have:
\begin{eqnarray}
\Fut_g(v)&=&\int_X v\left(\log \frac{\Omega_\vphi}{\omega_\vphi^n}-\log g_\vphi\right)g_\vphi \omega_\vphi^n\nonumber \\
&=&\int_X \frac{\mathfrak{L}_{{v}}\Omega_\vphi}{\Omega_\vphi}g_\vphi \omega_\vphi^n-\mathfrak{L}_{{v}}(g_\vphi \omega_\vphi^n)=- \int_X \theta_v(\vphi) g_\vphi \omega_\vphi^n. \label{eq-Futg2}
\end{eqnarray}
It is straightforward to verify that the last integral does not depend on the choice of $\vphi\in \mcH^T$. 

Assume that $v=v_\zeta=\frac{1}{2}(-J\zeta-\sqrt{-1}\zeta)$ such that $[\zeta, \xi]=0$ for any $\xi\in T$. Then $\zeta$ and $T$ together generates a possibly bigger torus $T'$. We can choose a $T'$-invariant metric $\vphi\in \mcH(L)$ and get a moment map from $X$ to $\mathfrak{t}'^*$ (the dual of the Lie algebra of $T'$). Then the Futaki invariant can then be expressed as an integral with respect to the associated Duistermaat-Heckman measure on $\mathfrak{t}'^*$. In particular, if we choose $v=v_\kappa$, then we get the necessary vanishing condition for the existence of $g$-solitons:
\begin{eqnarray}\label{eq-xicond}
\int_X \theta_v(\vphi) g_\vphi \omega_\vphi^n=\int_{P} x_{\kappa} g(x) \DHM_T(-K_X)=0.
\end{eqnarray}
A special class of $g$-soliton satisfies $g(x)=b(\ell_\xi)$ where $b$ is a smooth function over $\bR$ and $\ell_\xi=\la x, \xi\ra$. In this situation, let $a$ be a primitive function of $b$, i.e. $a'=b$ and define a function on $N_\bR$
\begin{equation}
\xi \mapsto \int_P a(x)\DHM_T(-K_X)=:\mathfrak{V}(\xi)
\end{equation}
where $a(x)=a(\ell_\xi(x))$.
The vanishing condition \eqref{eq-xicond} becomes the condition that $\xi$ is the critical point of $\mathfrak{V}(\xi)$. If we assume that $a$ is moreover strictly convex (resp. strictly concave), or equivalently that $b$ is strictly increasing (resp. strictly decreasing), then 
$\mathfrak{V}(\xi)$ is also strictly convex (resp. strictly concave) over $N_\bR$. So the critical point of $\mathfrak{V}(\xi)$ is unique if it exists. 

{\bf (2):}  The following form is useful later for integration to get $g$-Mabuchi functional. Set $f_\vphi=\log g_\vphi$ and
\begin{eqnarray}\label{eq-Futform2}
\Fut_g(v)=-\int_X \theta_v (\Delta_\vphi+f^i\partial_i) (H_\vphi-f_\vphi) g_\vphi \omega_\vphi^n. 
\end{eqnarray}
We have the following Matsushima type result (which can also be proved by using the second uniqueness statement of Theorem \ref{thm-gpluripotential}).
We should point out that there is such type of result in \cite{Lahdili19} in the more general setting of weighted extremal metrics. 
\begin{thm}
If we set
\begin{equation}\label{eq-AutXT}
\Aut(X, \bT)=\{\sigma\in \Aut(X); \sigma\cdot t=t\cdot \sigma \text{ for any } t\in \bT\}, 
\end{equation}
then its identity component $\Aut_0(X, \bT)$ is reductive.
\end{thm}
\begin{proof}
Let $\aut(X, \bT)$ be the Lie algebra of $\Aut(X, \bT)$. 
Assume that $\omega=\omega_\vphi$ is a $g$-soliton. We just need to prove the following identity: under the correspondence $v\mapsto \theta_v=-\frac{\mathfrak{L}_v \Omega_\vphi}{\Omega_\vphi}$, 
\begin{equation}\label{eq-aut2kernel}
\aut(X, \bT)\cong \{\theta\in C^\infty(X, \bC)^T; (\Delta+v_{f}+1)\theta=0 \}
\end{equation}
where $v_{f}=v_{f,\vphi}=\sum_\kappa f_\kappa v_\kappa$ (see \eqref{eq-vfvphi}). Indeed, note that when $\theta$ is $T$-invariant, $v_{f}(\theta)=Re(v_f)(\theta)$ and hence $\Delta+v_{f}+1$ is a real 
operator over $C^\infty(X, \bC)^T$, whose kernel is the complexification of the real subspace which corresponds to Killing vector fields that commute with $\bT$-action. To see the identity \eqref{eq-aut2kernel} we can first calculate by using the $g$-soliton equation $R_{k\bar{i}}=\omega_{k\bar{i}}+f_{k\bar{i}}$ to get:
\begin{eqnarray*}
((\Delta+v_f+1)\theta)_{\bar{i}}&=&(\theta+\sum_j(\theta_{j\bar{j}}+f_j\theta_{\bar{j}}))_{\bar{i}}=\theta_{\bar{i}}+\sum_j (\theta_{\bar{j}\bar{i},j}-\theta_{\bar{j}}R_{j\bar{i}}+f_{j\bar{i}}\theta_{\bar{j}}+f_j \theta_{\bar{j}\bar{i}})\\
&=&\sum_j (\theta_{\bar{j}\bar{i},j}+f_{j}\theta_{\bar{j}\bar{i}}).
\end{eqnarray*}
Now multiplying both sides by $\bar{\theta}_i$ and integrating by parts with respect to the measure $e^f \omega^n$, we get:
\begin{eqnarray*}
\int_X \sum_i ((\Delta+v_f+1)\theta)_{\bar{i}}\bar{\theta}_i e^f\omega^n &=&\int_X \sum_{i,j} \theta_{\bar{j}\bar{i}}\bar{\theta}_{ij} e^f \omega^n.
\end{eqnarray*}
We get $\theta_{\bar{j}\bar{i}}=0$ if and only if $(\Delta+v_f+1)(\theta)$ is a constant. This happens if and only if $(\Delta+v_f+1)(\theta)=0$ because:
\begin{equation*}
\int_X (\Delta+v_f+1)\theta e^f \omega^n=\int_X \theta e^f \omega^n=0
\end{equation*}
by the vanishing of $g$-weighted Futaki invariant.
\end{proof}

\subsection{Energy functionals}\label{sec-Archfunc}

We first define functionals on $\mcH(\omega_0)^T$. For any $\vphi\in \mcH(\omega_0)^T$, define:
\begin{eqnarray*}
\bfE_g(\vphi)=\frac{1}{\bV_g}\int_0^1 dt \int_X \dot{\vphi} g_\vphi \omega_\vphi^n=\frac{1}{\bV_g}\int_0^1 dt \int_X \vphi g_{t\vphi} \omega_{t\vphi}^n
\end{eqnarray*}
where $\vphi(t)$ is a smooth path connecting $0$ and $\vphi$ in $\mcH(\omega_0)^T$, hence the second identity. Note that $\bfE_g$ satisfies the monotonicity:
\begin{equation}
\vphi_1\le \vphi_2\quad \Longrightarrow\quad \bfE_g(\vphi_1)\le \bfE_g(\vphi_2). 
\end{equation}
One can show that $\bfE_g$ is well defined: it does not depend on the path. 
Moreover we define the following functionals:
\begin{eqnarray*}
\bfI_g(\vphi)&=&\frac{1}{\bV_g}\int_X \vphi(g_0 \omega_0^n-g_\vphi\omega_\vphi^n)\\
\Lam_g(\vphi)&=&\frac{1}{\bV_g}\int_X \vphi g_0 \omega_0^n\\
\bfJ_g(\vphi)&=&\Lam_g(\vphi)-\bfE_g(\vphi)\\
(\bfI_g-\bfJ_g)(\vphi)&=&\bfE_g(\vphi)-\frac{1}{\bV_g}\int_X \vphi g_\vphi\omega_\vphi^n.
\end{eqnarray*}
We will denote by $\bfE, \bfI, \bfJ$ the above functional when $g=1$. 
\begin{lem}\label{lem-IJ}
\begin{enumerate}
\item[(i)]
There exists $C_1=C_1(n, g)>0$ such that $C_1^{-1} \bfF\le \bfF_g \le C_1 \bfF$ for any $\bfF\in \{\bfI, \bfJ, \bfI-\bfJ\}$. 
\item[(ii)]
$\bfI_g(\vphi)\ge 0$. Moreover $\bfI_g(\vphi)=0$ if and only if $\vphi-\vphi_0$ is a constant.
\item[(iii)] There exists $C_2=C_2(n, g)>0$ such that $C_2^{-1} \bfJ_g\le \bfI_g-\bfJ_g \le C_2\bfJ_g$. 
\item[(iv)]
For any $t\in [0, 1]$, we have the inequality:
\begin{equation}\label{eq-Dingineq}
\bfJ_g(t u)\le t^{1+C_2^{-1}} \bfJ_g(u).
\end{equation}
\item[(v)]
There exists $C=C(X, L)$ such that 
\begin{equation}
\Lam_g(\vphi) \le \sup\vphi\le \Lam_g(\vphi)+C.
\end{equation}
\end{enumerate}
\end{lem}
One can prove (i) first and then deduce (i) $\Rightarrow$ (ii) by using the well-known property for $\bfI, \bfJ$. With (ii) proved,  (iii) follows by integrating the differential inequality:
\begin{eqnarray*}
\frac{d}{dt}\bfJ_g(tu)=\frac{1}{\bV_g}\int_X u( \omega_0^n-\omega_{tu}^n)=\frac{1}{t}\bfI_g(tu)\ge \frac{1}{C_2 t} \bfJ_g(tu).
\end{eqnarray*}
The first inequality in (v) is obviously true. The second one can be proved using Hartogs' theorem for $\omega_0$-plurisubharmonic functions (see \cite[Lemma 13]{HL20}). 

The $g$-soliton are critical points of two important functionals $\bfD_g$ and $\bfM_g$ over $\mcH(L)$. 
First define $\bfD_g$:
\begin{eqnarray*}
\bfL(\vphi)&=&-\log \left(\frac{1}{\bV_g}\int_X e^{-\vphi}\Omega_0\right)\\
\bfD_g(\vphi)&=&-\bfE_g(\vphi)+\bfL(\vphi).
\end{eqnarray*}
Next define $\bfM_g$.
Recall that the fixed $T$-invariant reference metric $h_0$ on $-K_X$ is considered as a volume form $\Omega_0$ on $X$. Set:
\begin{eqnarray*}
\bfH(d\nu)&:=&\int_X \log \frac{d\nu}{\Omega_0}\frac{d\nu}{\bV_g}\\
\bfH_g(\vphi)&:=& \bfH(g_\vphi \omega_\vphi^n)\\
\bfM_g(\vphi)&:=&\bfH_{g}(\vphi)-(\bfI_g-\bfJ_g)(\vphi)\\
&=& \bfH(g_\vphi \omega_\vphi^n)+\frac{1}{\bV_g}\int_X \vphi g_\vphi \omega_\vphi^n-\bfE_g(\vphi). 
\end{eqnarray*}
In the above formula, if one replace $\Omega_0$ by another smooth volume form $\Omega$, then one gets new functionals which differ from the above ones by uniformly bounded quantities.

The following calculation shows that, similar to the original definition of Mabuchi functional, $\bfM_g$ can be viewed as an integration of the $g$-weighted Futaki invariant (in the form given in \eqref{eq-Futform2})
\begin{eqnarray*}
&&-\int_X \dot{\vphi} (\Delta_\vphi+v_{f,\vphi})(H_\vphi-f_\vphi)  g_\vphi \omega_\vphi^{n}\\
&=&-\int_X (H_\vphi-f_\vphi) ((\Delta_\vphi+v_{f,\vphi})\dot{\vphi})g_\vphi \omega_\vphi^n\\
&=&-\frac{d}{dt}\int_X (H_\vphi-f_\vphi)g_\vphi \omega_\vphi^n+\int_X \left(\frac{d}{dt}(H_\vphi-f_\vphi)\right) g_\vphi \omega_\vphi^n\\
&=&\frac{d}{dt}\int_X \log \frac{g_\vphi \omega_\vphi^n}{\Omega_\vphi}g_\vphi \omega_\vphi^n-\int_X \dot{\vphi} g_\vphi \omega_\vphi^n-\frac{d}{dt}\int_X g_\vphi \omega_\vphi^n\\
&=&\frac{d}{dt}\int_X \log\frac{g_\vphi \omega_\vphi^n}{\Omega_0}g_\vphi\omega_\vphi^n +\frac{d}{dt} \int_X \vphi g_\vphi \omega_\vphi^n-\bV_g\cdot \frac{d}{dt}\bfE_g(\vphi)\\
&=&\bV_g\cdot \frac{d}{dt}\left(\bfH_{g}-(\bfI_g-\bfJ_g)\right)=\bV_g\cdot \frac{d}{dt}\bfM_{g}(\vphi).
\end{eqnarray*}
For the first identity, $v_{f,\vphi}$ is the vector field defined in \eqref{eq-vfvphi} and we used the fact that $\Delta_\vphi+v_{f,\vphi}$ is self-adjoint with respect to the volume form $g_\vphi \omega_\vphi^n=e^{f_\vphi}\omega_\vphi^n$. 
For the 3rd equality, we used the identities:
\begin{eqnarray*}
\left(\frac{d}{dt}(H_\vphi-f_\vphi)\right)g_\vphi \omega_\vphi^n&=&\left(\frac{d}{dt}\log \frac{\Omega_\vphi}{g_\vphi \omega_\vphi^n}\right)g_\vphi \omega_\vphi^n
=-\dot{\vphi}g_\vphi \omega_\vphi^n-\frac{d}{dt}(g_\vphi \omega_\vphi^n).
\end{eqnarray*}
For the fourth equality, we used that fact that $\int_X g_\vphi\omega_\vphi^n=\bV_g$ does not depend on $\vphi\in \mcH$.
\vskip 2mm
For each of $\bfF\in \{\bfI, \bfJ, \bfI-\bfJ, \bfD, \bfM\}$, $\bfF_g(\vphi+c)=\bfF_g(\vphi)$ for any constant $c\in \bR$. So we can write $\bfF(\omega_\vphi)$ for $\bfF(\vphi)$.

\subsection{Analytic criterion}

Let $p_1: X\times [0, 1]\times S^1\rightarrow X$ be the projection. 
For any $\vphi_0, \vphi_1\in \mcH(\omega_0)^T$, the geodesic between them is a bounded $p_1^*\omega_0$-psh function $\Phi$ on $X\times [0,1]\times S^1$ that is the unique solution to the following degenerate complex Monge-Amp\`{e}re equation. 
\begin{equation}
(p_1^*\omega_0+\sddb\Phi)^{n+1}=0, \quad \left.\Phi\right|_{\{i\}\times S^1}=\vphi_i, i=0 ,1. 
\end{equation}  
Note that $\Phi$ is automatically $S^1\times T$-invariant, 
It is now known that $\Phi\in C^{1,1}(X\times [0,1]\times S^1)$ (\cite{Che00, CTW18}). 
A basic property we need is:
\begin{lem}
$\bfE_g$ is affine along any geodesic segment.
\end{lem}
If we set $f(s)=\bfE_g(\vphi(s))$ for $s\in [0, 1]\times S^1$, then the above statement follows from the formula:
\begin{equation*}
\sddb f(s)=(p_1)_*\left( g_\Phi (p_1^*\omega_0+\sddb\Phi)^{n+1}\right)
\end{equation*}
where $g_\Phi(z, s)=g_{\vphi(s)}$.

To use the variational approach, we need to adapt the pluripotential theory to the setting of $g$-Monge-Amp\`{e}re measure.
Recall that $\vphi\in \PSH(\omega_0)$ means that $\vphi$ is $\omega_0$-psh: $\vphi+\psi$ is a plurisubharmonic function for any local potential function $\psi$ of $\omega_0$. 
For any $\vphi \in \PSH(\omega_0)$, we also say that $h_\vphi=h_0 e^{-\vphi}$ 
is a psh Hermitian metric. 
By Bedford-Taylor theory, the Monge-Amp\`{e}re measure $(\omega_0+\sddb\vphi)^n$ is well-defined for any bounded function $\vphi\in \PSH(\omega_0)$. For general $\vphi\in \PSH(\omega_0)$ 
one can define the non-pluripolar Monge-Amp\`{e}re measure:
\begin{equation}\label{eq-nonpluri}
(\omega_0+\sddb\vphi)^n:=\lim_{j\rightarrow+\infty}{\bf 1}_{\{\vphi>-j\}} \left(\omega_0+\sddb\max\{\vphi, -j\}\right)^n. 
\end{equation}
Define the space $\cE(\omega_0)$ of potentials with full mass:
\begin{eqnarray*}
\cE=\cE(\omega_0)&=&\left\{\vphi\in \PSH(\omega_0); \int_X (\omega_0+\sddb\vphi)^n=\int_X \omega_0^n\right\}\\
\cE^1=\cE^1(\omega_0)&=&\left\{\vphi\in \cE(\omega_0); \int_X |\vphi|(\omega_0+\sddb\vphi)^n<+\infty\right\}.
\end{eqnarray*}

We will also denote $\cE(L)=\{h_\vphi=h_0 e^{-\vphi}; \vphi\in \cE(\omega_0)\}$ and $\cE^1(L)=\{h_{\vphi}; \vphi\in \cE^1(\omega_0)\}$. 
In our present setting with $\bT$-action, we denote by $\cE^1(\omega_0)^T$ (resp. $\cE^1(L)^T$) the space of $T$-invariant functions from $\cE^1(\omega_0)$ (resp. $T$-invariant Hermitian metrics $h_\vphi\in \cE^1(L)$). 

We will first discuss a fibration construction which is well-known in the framework of equivariant cohomology (Appendix \ref{App-equiv} and also \cite{Don05}). This construction will be used to define $g$-Monge-Amp\`{e}re measure 
for singular psh potentials. It is also used to prove the slope formula and the monotonicity along MMP in the work \cite{HL20}. There is a further application in the study of weighted-extremal metrics in \cite{AJL21}.

{\bf Fibration construction:} 
Let $S^{2k+1}\rightarrow \bP^k$ be the Hopf fibration. For $\vec{d}=(k_1, \dots, k_r)\in \bN^r$, set 
 $\bP^{[\vec{d}]}=\bP^{k_1}\times\cdots \bP^{k_r}$ and $\bS^{[\vec{d}]}=S^{2k_1+1}\times\cdots S^{2k_r+1}$. 
Let $\bS^{[\vec{d}]}\rightarrow \bP^{[\vec{d}]}$ be the $(S^1)^r$-principal bundle and $(X^{[\vec{d}]}, L^{[\vec{d}]})=(X, L) \times_{(S^1)^r} \bS^{[\vec{d}]}\rightarrow \bP^{[\vec{d}]}$ be the associated fibre bundle with the natural projection $\pi: X^{[\vec{d}]}\rightarrow \bP^{[\vec{d}]}$.  
Any $T$-invariant function lifts to become a $T$-invariant function on $X^{[\vec{d}]}$. Moreover for any $T$-invariant Hermitian metric $h=h_0e^{-\vphi}$ on $L$, there is a Hermitian metric $h^{[\vec{d}]}=h_0^{[\vec{d}]} e^{-\vphi^{[\vec{d}]}}$ on $L^{[\vec{d}]}$. 
Note that there is a canonical ample line bundle $H^{[\vec{d}]}$ over $\bP^{[\vec{d}]}$ which is the product of hyperplane bundles of each $\bP^{d_\kappa}$ and is equipped with the canonical Fubini-Study metric denoted by $h_{\FS}^{\vec{d}}$. Choosing $c\gg 1$ and pulling back the Fubini-Study metric $h^{\vec{[d]}}_\FS$, $h^{[\vec{d}]}\otimes (h_{\FS}^{[\vec{d}]})^{\otimes c}$ is a psh metric on $L^{[\vec{d}]}\otimes \pi^*(H^{[\vec{d}]})^{\otimes c}$. If $h$ is smooth psh metric, then by the formula \eqref{eq-equicurv} the Hermitian metric $h^{[\vec{d}]}\otimes (h^{[\vec{d}]}_\FS)^{\otimes c}$ is also smooth psh Hermitian metric $\omega^{[\vec{d}],c}_0$. By taking decreasing approximating sequence, we know this psh preserving property is still true for singular psh Hermitian metrics.  

For any $\vec{d}\in \bN^r$ and $g=g_{\vec{d}}=\prod_\kappa (\theta_\kappa+c)^{d_\kappa}$, we can define 
$g_\vphi\omega_\vphi^n=g(\bm_\vphi) (\omega_0+\sddb\vphi)^n$: for any test smooth function $f$, set $f^T=\int_{T} f(\sigma\cdot z)d\sigma$ where $d\sigma$ is the Haar measure on $T$ and define  
\begin{equation}
\int_X f \cdot g_\vphi\omega_\vphi^n:=\frac{n!}{(n+d)!}\int_{X^{[\vec{d}]}} (f^T)^{[\vec{d}]}\cdot (\omega^{[\vec{d}],c}+\sddb \vphi^{[\vec{d}]})^{n+d},
\end{equation}
where the right-hand-side is defined using he non-pluripolar product as in \eqref{eq-nonpluri}. 
This allows us to define $g_\vphi \omega_\vphi^n$ as a signed measure when $g$ is polynomial. If $g$ is continuous function on $P$, by Stone-Weierstrass theorem, we can find a sequence of polynomials $g_j$ that converges to $g$ uniformly on $P$. We then define:
\begin{equation}
g_\vphi \omega_\vphi^n=\lim_{j\rightarrow+\infty} (g_j)_\vphi \omega_\vphi^n. 
\end{equation}
In other words, for any $f\in C^0(X)$, define:
\begin{equation}
\int_X f g_\vphi \omega_\vphi^n=\lim_{j\rightarrow+\infty} \int_X f\cdot (g_j)_\vphi \omega_\vphi^n. 
\end{equation}
One can verify that the limit of right-hand-side indeed exists and is bounded by $C \|f\|_{C^0}$. The Riesz representation theorem defines $g_\vphi\omega_\vphi^n$ uniquely. It is useful to note that:
\begin{equation}\label{eq-gjuniform}
\left|\int_X f\cdot (g_j)_\vphi \omega_\vphi^n-\int_X f g_\vphi \omega_\vphi^n\right|\le C \|g_j-g\|_{C^0}\cdot \|f\|_{C^0}.
\end{equation}
This is clear for smooth $\vphi$. For general $\vphi$, it follows by using smooth approximation of $\vphi$ using functions from $\mcH(\omega_0)$. 
\begin{rem}
Another way to define $g_\vphi \omega_\vphi^n$ was given in \cite{BW14} by using a more complicated process. Our definition seems more natural and more adapted to the pluripotential analysis in both the Archimedean and non-Archimedean settings.
\end{rem}
With the above definition, we can define the $g$-weighted version of finite energy space of Guedj-Zeriahi: 
\begin{eqnarray*}
\cE_g=\cE_g(\omega_0)&=&\{\vphi\in \PSH(\omega_0)^T; \int_X g_\vphi \omega_\vphi^n=\bV_g\}\\
\cE^1_g=\cE^1_g(\omega_0)&=&\{\vphi\in \cE_g(\omega_0); \int_X |\vphi| \omega_\vphi^n<+\infty\}. 
\end{eqnarray*}
Because both $g$ and $g^{-1}$ are bounded, it is easy to see that $g_\vphi \omega_\vphi^n$ and $\omega_\vphi^n$ is absolutely continuous with respect to each other.
As a consequence, $\cE_g=(\cE)^T$ and $\cE^1_g=(\cE^1)^T$. Note that when $g=g_{\vec{d}}$ as before, then $\vphi\in \cE^1_g$ if and only if $\vphi^{[\vec{d}]}\in \cE^1(\omega_0^{[\vec{d}],c})$. 

One key property of the $g$-Monge-Amp\`{e}re measure is the continuity under decreasing sequences: if $\vphi_k$ is a decreasing sequences of $\omega_0$-psh functions that converges pointwise to $\vphi\in (\cE^1)^T$, then 
\begin{equation}
\lim_{k\rightarrow+\infty} g_{\vphi_k} \omega_{\vphi_k}^n=g_{\vphi} \omega_\vphi^n. 
\end{equation}
If $g=g_{\vec{d}}$, this is the corresponding property in $\cE^1(\omega_0^{[\vec{d}],c})$ proved in \cite{BBEGZ}. For general $g$, this can be proved using approximation by polynomials and using \eqref{eq-gjuniform}.


For two $\vphi_i, i=0, 1\in \cE^1$, there exists a unique geodesic segment $\Phi$ connecting $\vphi_1$ and $\vphi_2$, which can be constructed using the following approximation process. First by Demailly's regularization result, there exists $\vphi_{i, m}\in \mcH$ such that $\vphi_{i,m}$ decreases to $\vphi_i$. Let $\Phi_{i,m}$ be the geodesic segment connecting $\vphi_{i,m}, i=0,1$. Then $\Phi_{i,m}$ decreases to $\Phi$. If $\vphi_i, i=0,1$ are $T$-invariant, then $\Phi$ is also $T$-invariant. Moreover along $\Phi=\{\vphi(s)\}$, $\bfE_g(s)$ is also affine with respect to $s$. 

\begin{defn}
Define the strong topology on $(\cE^1)^T$: $\vphi_j$ converges to $\vphi$ strongly if $\vphi_j$ converges to $\vphi$ weakly and $\bfE_g(\vphi_j)$ converges to $\bfE_g(\vphi)$. 
\end{defn}
The same argument as used in \cite{BBEGZ} (where $g=1$) shows:
\begin{prop}
$\vphi_j$ converges to $\vphi$ strongly if and only if $\int_X \vphi_j\omega_0^n \rightarrow \int_X \vphi\omega_0^n$ and $\bfI_g(\vphi_j, \vphi)\rightarrow 0$. 
Moreover in this case, $\omega_{\vphi_j}$ converges to $\omega_\vphi$ weakly. 
\end{prop}
We can replace the first condition by $\sup \vphi_j\rightarrow \sup \vphi$ and the second condition by $\bfI(\vphi_j, \vphi)\rightarrow 0$. So we see that the above strong topology coincides with the original strong topology defined by using $\bfE_g$ (or $\bfI_g$) studied in \cite{BBEGZ}. 
Similar discussion also shows that the functional $\bfF_g$ for $\bfF\in \{\bfE, \Lam, \bfI, \bfJ\}$ is continuous under the strong topology of $\cE^1_g$. We also have an important compactness result:
\begin{thm}[\cite{BBEGZ}]\label{thm-E1compact}
The set of metrics $\{\vphi\in \cE^1_g; \bfH_g(\vphi)\le C<\infty, \sup\vphi=0 \}$ is compact in strong topology. 
\end{thm}
The following results summarize the variational characterization of $g$-solitons. The proof uses pluripotential techniques and is a direct generalization of the works in \cite{BBEGZ, Bern15}. 
\begin{thm}[\cite{BBEGZ, BW14, HL20}]\label{thm-gpluripotential}
\begin{enumerate}
\item
$\inf_{\vphi\in \cE^1_g}\bfD_g(\vphi)=\inf_{\vphi\in \cE^1_g}\bfM_g(\vphi)$. 
\item 
$\bfD_g$ and $\bfM_g$ is convex along geodesic segments in $\cE^1_g$. Moreover, $\bfD_g$ is affine along the geodesic segment if and only if the geodesic segment is generated by a one parameter subgroup in $\Aut(X, \bT)$.
\item
The following conditions are equivalent:
\begin{enumerate}
\item[(i)]
$\vphi\in \cE^1_g$ is a smooth $g$-soliton metric.
\item[(ii)]
$\vphi$ obtains the minimum of $\bfD_g$.
\item[(iii)]
$\vphi$ obtains the minimum of $\bfM_g$. 
\end{enumerate}
\item $g$-soliton forms $\omega_\vphi$ are unique up to the action of $\Aut_0(X, \bT)$ (see \eqref{eq-AutXT}). Moreover if $\tilde{\bT}$ is a maximal torus of $\Aut_0(X, \bT)$ and $\tilde{T}$ is a maximal compact torus of $\tilde{T}$, then 
$\tilde{T}$-invariant $g$-soliton forms $\omega_\vphi$ are unique up to the action by $\tilde{\bT}$. 
\end{enumerate}
\end{thm}
We briefly explain the first statement. By using Jensen's inequality, one gets easily
\begin{eqnarray*}
\bfH_g(\vphi)+\frac{1}{\bV_g}\int_X \vphi g_\vphi \omega_\vphi^n\ge \bfL(\vphi)
\end{eqnarray*}
which implies $\bfM_g(\vphi)\ge \bfD_g(\vphi)$. Conversely, for any $\vphi\in \cE^1_g$, there exists $\psi$ satisfying 
$g_{\psi}(\sddb\psi)^n=\Omega_0 e^{-\vphi}$. Then by setting $d\nu=\bV_g\cdot \frac{\Omega_0}{\int_X\Omega_0}$, we get
$\bfL(\vphi)=\bfH_g(\psi)$, $\bfE_g(\vphi)=(\bfI_g-\bfJ_g)(\psi)$, so $\bfD_g(\vphi)=\bfM_g(\psi)\ge \inf \bfD_g(\vphi)$. 

For the second statement, the characterization of affine-ness of $\bfL$ was proved in \cite{BBEGZ} based on Berndtsson's uniqueness result in \cite{Bern15}. 
The extra information of commutativity with $\bT$ follows from the following fact in symplectic geometry. If $(X, \omega)$ is a symplectic manifold and $\xi_1, \xi_2$ are two Hamiltonian vector fields with Hamiltonian functions $\theta_1$ and $\theta_2$ respectively. Then $[\xi_1, \xi_2]=0$ if and only if $\xi_1(\theta_2)=0$. 

The special role of maximal torus in the last statement was first observed in \cite{Li19} and also in \cite{Lahdili20}. 
\begin{defn}
Assume $\bfF\in \{\bfD, \bfM\}$. 
$\bfF_g$ is reduced coercive, if there exist $\gamma>0, C>0$ such that for any $\vphi\in \cE^1_g$ we have:
\begin{equation}\label{eq-redcoer}
\bfD_g(\vphi)\ge \gamma\cdot \bfJ_{g,\tilde{\bT}}(\vphi)-C
\end{equation}
where 
\begin{equation}
\bfJ_{g,\tilde{\bT}}(\vphi)=\inf_{\sigma\in \tilde{\bT} }\bfJ_g(\sigma^*\omega_\vphi).
\end{equation}
\end{defn}
Note that in general $\sigma^*\omega_\vphi=\sigma^*\omega_0+\sigma^*\sddb \vphi\neq \omega_0+\sddb \sigma^*\vphi= \omega_{\sigma^*\vphi}$. 

By appropriate regularization process, one can show that reduced coercivity over $\cE^1_g$ is equivalent to the coercivity over $\mcH^T$, i.e. the inequality \eqref{eq-redcoer} holds true for any $\vphi\in \mcH^T$. 
\begin{thm}
The following conditions are equivalent:
\begin{enumerate}
\item[(i)]
There exists a $g$-soliton metric on $X$.
\item[(ii)] 
$\bfD_g$ is reduced coercive over $\mcE^1_g$.
\item[(iii)]
$\bfM_g$ is reduced coercive $\mcE^1_g$.
\end{enumerate}
\end{thm}
We sketch the proof of
(i)$\Rightarrow $ (ii) by essentiall following Darvas-Rubinstein's argument (\cite{DR17}) which depends on the uniqueness result from Theorem \ref{thm-gpluripotential}. 
We can assume that $\vphi_0$ is a $g$-soliton metric. 
Suppose that $\bfD_g$ is not reduced coercive. Then for any $j>0$ there exists $\vphi_j\in \mcH^T$ such that 
$\bfD_g(\vphi_j)\le j^{-1} \inf_{\sigma\in \tilde{\bT}} \bfJ_{g, \tilde{\bT}}(\vphi_j)-j$. We can assume that $\vphi_j$ satisfies:
\begin{equation*}
\bfJ_g(\vphi_j)=\bfJ_{g,\tilde{\bT}}(\vphi_j)=\inf_{\sigma\in \tilde{\bT}} \bfJ_g(\sigma^*\omega_{\vphi_j}), \quad \sup \vphi_j=0. 
\end{equation*}
Then we have $0\le \bfD_g(\vphi_j)\le j^{-1} \bfJ_g(\vphi_j)-j$ which implies that
$
\bfJ_g(\vphi_j)\ge j^2\rightarrow +\infty. 
$
This implies $S_j:=-\bfE_g(\vphi_j)=\bfJ_g(\vphi_j)+O(1)\rightarrow+\infty$. Let $\{\vphi_j(s)\}_{s\in [0, S_j]}$ be the geodesic segment connecting $\vphi_0$ and $\vphi_j$. Then by the convexity of $\bfD_g$ along geodesic segment. We get:
\begin{eqnarray*}
\bfD_g(\vphi_j(s))\le \frac{S_j-s}{S_j}\bfD_g(\vphi_0)+\frac{s}{S_j}\bfD_g(\vphi_j)\le \frac{s}{S_j}j^{-1}.
\end{eqnarray*}
We know that $\vphi_j(s)$ converges weakly to $\vphi_\infty(s)$. It is known that $\bfD_g$ is lower semicontinuous with respect to the weak convergence, we know that $\bfD_g(\vphi_\infty(s))=0$. This implies $\vphi_\infty(s)\in (\cE^1)^{\tilde{T}}$ is a $g$-soliton metric. So $\omega_{\vphi_\infty}(s)=\sigma(s)^*\omega_{\vphi_0}$ for some $\sigma(s)\in \tilde{\bT}$ by Theorem \ref{thm-gpluripotential}.4. Moreover $\vphi_j(s)$ then converges to $\vphi_\infty(s)$ strongly and $\Phi_\infty=\{\vphi_\infty(s)\}$ becomes a geodesic ray which then satisfies $\bfJ'^\infty_{g,\bT}(\Phi)=1$ (see \cite[Proof of Proposition 6.2]{Li20}). But this contradicts the condition that $\vphi_\infty(s)$ are in the same orbit of $\tilde{\bT}$.   

\subsection{Stability of $(X, \bT)$}
\begin{defn}
\begin{enumerate}
\item
A $\bT$-equivariant test configuration consists of the data $(\mcX, \mcL, \zeta)$ that satisfies the following conditions:
\begin{enumerate}
\item[(i)] $\pi: \mcX\rightarrow \bC$ is a flat family of projective varieties, and $\mcL$ is a $\pi$-semiample $\bQ$-line bundle.
For each $t\in \bC\setminus \{0\}$, $(\mcX_t, \mcL_t)\cong (X, L)$. 
\item[(ii)] $\zeta$ is a holomorphic vector field that generates an effective $\bC^*$-action on $\mcX$ such that $\pi$ is $\bC^*$-equivariant satisfying $\pi_*(\zeta)=-t\partial_t$. Moreover $\zeta$ lifts to a linear action on $\mcL$. 
\item[(iii)] There are fibrewise $\bT$-action that commutes with the $\bC^*$-action. 
\item[(iv)] There is an $(\bC^*\times\bT)$-equivariant isomorphism $(\mcX, \mcL)\times_{\bC}\bC^*=(X, L)\times \bC^*$. 
\end{enumerate}

\item
A test configuration $(\mcX, \mcL)$ is special if $\mcX_0$ is a $\bQ$-Fano variety and $\mcL\sim_\bC -K_{\mcX/\bC}$. 
\item
For any $\bT$-equivariant test configuration $(\mcX, \mcL)$ and $\xi\in N_\bR$, we define the $\xi$-twist of $(\mcX, \mcL)$ as the data $(\mcX, \mcL, \zeta+\xi)$, which is also denoted simply by $(\mcX_\xi, \mcL_\xi)$. 
\end{enumerate}
\end{defn}
Two test configurations $(\mcX_i, \mcL_i), i=1,2$ are equivalent if there exists a model $\mcX'$ equipped with $\bC^*$-equivariant morphisms $p_i: \mcX' \rightarrow \mcX_i$ such that $p_1^*\mcL_1=p_2^*\mcL_2$. 

A test configuration is called dominant if there exists a birational morphism $\mcX\rightarrow X\times\bC=: X_\bC$.
We will identify equivalent test configurations. By resolution of singularities, any test configuration is equivalent to a dominant test configuration so that without loss of generality, we can always assume that a given test configuration is dominant.

It is convenient to view test configuration from two different points of view
\begin{enumerate}
\item
For $m\gg 1$, there exists a $(\bT\times\bC^*)$-equivariant birational morphism $f_m: \mcX\rightarrow \bP^{N_m-1}\times \bC$ such that 
$f_m^*H=m \mcL$ where $N_m=h^0(X, mL)$. 
The holomorphic vector field $\zeta$ is identified with the generator of a one parameter $\bC^*$-subgroup of $PGL(N_m, \bC)$ which can be diagonalized as ${\rm diag}\{\lambda^{(m)}_1, \dots, \lambda^{(m)}_{N_m}\}$. Choose a $T\times S^1$-invariant smooth Hermitian metric on $\mcL$. For our purpose, one will just choose $e^{-\vphi}$ to be the pull back of the $1/m$-th root of Fubini-Study metric on the hyperplane bundle on $\bP^{N_m-1}$.  

\item For any test configuration $(\mcX, \mcL)$, there exists an equivalent test configuration $(\mcX', \mcL')$ with dominant morphism $\rho: \mcX'\rightarrow X\times\bC$ which is obtained by blowing up a flag ideal on $X\times\bC$ and $\mcL'=\rho^*L_\bC+\sum_i b_i F_i$ where $F_i$ are irreducible components of $\mcX'_0$. 
\end{enumerate}
Assume the test configuration $(\mcX, \mcL)$ has the central fibre $(\mcX_0, \mcL_0)$. Define non-Archimedean functionals:
\begin{eqnarray}
\bfE^\NA(\mcX, \mcL)&=&\frac{1}{\bV}\int_{\mcX_0} \theta_\zeta(\vphi) (\omega_0+\sddb\vphi)^n
=\frac{1}{\bV}\lim_{m\rightarrow+\infty} \frac{n!}{m^n}\sum_i \frac{\lambda^{(m)}_i}{m}\\
&=&\frac{\bar{\mcL}^{\cdot n+1}}{(n+1)\bV}.\\
\bfE_g^\NA(\mcX, \mcL)&=&\frac{1}{\bV_g}\int_{\mcX_0} \theta_\zeta(\vphi) g_\vphi \omega_\vphi^n=\frac{1}{\bV_g}\lim_{m\rightarrow+\infty} \frac{n!}{m^n}
\sum_{\alpha,i} g(\frac{\alpha}{m})\frac{\lambda^{(m,\alpha)}_i}{m}.  \label{eq-ENAg}
\end{eqnarray}
Note that if $\mcX_0$ is a Fano manifold (or more generally a $\bQ$-Fano variety, i.e. if $\mcX$ is special), then $\bfE^\NA_g(\mcX, \mcL)=-\bV_g^{-1}\Fut_g(\zeta)$ on the central fibre $\mcX_0$ (see \eqref{eq-Futg2}). 

When $g=\binom{n}{\vec{d}}\prod_\kappa x_{\kappa}^{d_\kappa}$, we can use the fibration construction to express:
\begin{eqnarray*}
\bfE^\NA_g(\phi)=\frac{1}{\bV_g}(\bar{\mcL}^{[\vec{d}]})^{\cdot n+d+1}\cdot \frac{n!}{(n+d+1)!}
\end{eqnarray*}
For general $g$ and a sequence $g_j$ converging to $g$ uniformly, we have:
\begin{equation*}
\bfE^\NA_g(\mcX, \mcL)=\lim_{j\rightarrow+\infty}\bfE^\NA_{g_j}(\mcX, \mcL). 
\end{equation*}
Furthermore, one can define:
\begin{eqnarray*}
\Lam_g^\NA(\mcX, \mcL)&=&\Lam^\NA(\mcX, \mcL)=\lim_{m\rightarrow+\infty} \max_i\left\{\frac{\lambda^{(m)}_i}{m}\right\}\\
\bfJ_g^\NA(\mcX, \mcL)&=&\Lam^\NA_g(\mcX, \mcL)-\bfE^\NA_g(\mcX, \mcL)\\
&&\\
\lct(\mcX, K_{\mcX}+\mcL; \mcX_0)&=&\sup\{t; (\mcX, -(K_{\mcX}+\mcL)+t\mcX_0) \text{ is sub-log-canonical }\}\\
\bfL^\NA(\mcX, \mcL)&=&\lct(\mcX, -(K_{\mcX}+\mcL); \mcX_0)-1.
\end{eqnarray*}
Now assume that $(\mcX, \mcL)$ is a $\bT$-equivariant test configuration.
 
Let $\tilde{\bT}$ be a maximal torus of $\Aut(X, \bT)$ (see \eqref{eq-AutXT}). Set $\tilde{N}_\bZ={\rm Hom}_{\rm alg}(\bC^*, \tilde{\bT})$ and $\tilde{N}_\bQ=\tilde{N}_\bZ\otimes_\bZ\bQ$. The following definition generalizes the definitions introduced by Tian and Donaldson.
\begin{defn}
\begin{enumerate}
\item
$(X, \bT)$ is called reduced uniformly $g$-weighted Ding-stable (or just $g$-weighted stable) if there exists $\gamma>0$ such that for any $\tilde{\bT}$-equivariant test configuration $(\mcX, \mcL)$,
\begin{equation*}
\bfD_g^\NA(\mcX, \mcL)\ge \gamma\cdot \bfJ^\NA_{g, \tilde{\bT}}(\mcX, \mcL)
\end{equation*} 
where $\bfJ^\NA_{\tilde{\bT}}(\mcX, \mcL)=\inf_{\xi\in \tilde{N}_\bQ} \bfJ(\mcX_\xi, \mcL_\xi)$. 
\item
For any reductive subgroup $\bG$ of $\Aut(X, \bT)$, 
$(X, \bT)$ is called $\bG$-equivariantly $g$-weighted Ding-polystable if $\bfD_g^\NA\ge 0$ and $\bfD_g^\NA(\mcX, \mcL)=0$ if and only if $(\mcX, \mcL)$ is a product test configuration. 
\item
$(X, \bT)$ is called $g$-weighted Ding-polystable, if for any $\bT$-equivariant normal test configuration $(\mcX, \mcL)$, $\bfD_g^\NA(\mcX, \mcL)\ge 0$ and $\bfD_g^\NA(\mcX, \mcL)=0$ if and only if $(\mcX, \mcL)$ is a product test configuration.
\end{enumerate}
\end{defn}
For any reductive subgroup of $\Aut(X, \bT)$, one can define  $\bG$-uniform stability following \cite{His16b} (see also \cite{Li19}). For simplicity of exposition, here we only consider the case when $\bG=\tilde{\bT}$ and use the terminology of reduced uniform stability as in \cite{XZ20}. We refer to the survey \cite{Xu20} for developments in recent study of K-stability. 

\begin{thm}\label{thm-special}
A Fano manifold $(X, \bT)$ is reduced uniformly $g$-weighted stable if and only if it is reduced uniformly $g$-weighted stable for special test configurations.
\end{thm}
\begin{proof}[Sketch of proof]
This is proved by using the Minimal Model Program (MMP) and calculate the variation of $\bfD^\NA-\epsilon \bfJ^\NA_\bT$. We only give a sample of calculation. For any test configuration $(\mcX, \mcL)$, there is a process developed in \cite{LX14} that modify the $(\mcX, \mcL)$ to become a special test configuration $(\mcX^s, \mcL^s)$. 
We consider simple case when $(\mcX, \mcX_0=\sum_{i=1}^d E_i)$ has log canonical singularities and $\mcL$ relatively ample. 
Then we can run the MMP with rescaling by setting $\mcL_t:=\frac{K_{\mcX}+t \mcL}{t-1}$. Here we only show that over short time interval the difference $\bfD^\NA-\epsilon \bfJ^\NA$ decreases . The long time behavior and the effect of twisting require more arguments and calculation. However the following calculation is typical in the proof (see \cite{Fuj19a, HL20}). Assume that $-K_{\mcX}-\mcL=\sum_i e_i E_i$ with $e_1\ge e_2\ge \cdots \ge e_d$. 
\begin{eqnarray*}
&&-K_{\mcX}-\mcL_t=\frac{t}{t-1}(-K_{\mcX}-\mcL)=\frac{t}{t-1}\sum_i e_i E_i\\
&&\bfL^\NA(\mcX, \mcL_t)=-\frac{t}{t-1}e_1-1, \quad \frac{d}{dt}\bfL^\NA=\frac{1}{(t-1)^2}e_1.\\
&&\dot{\mcL}_t=\frac{d}{dt} \mcL_t=-\frac{K_{\mcX}+\mcL}{(t-1)^2}=\frac{1}{(t-1)^2} \sum_i e_i E_i \\
&&\frac{d}{dt}\bfE^\NA_g=\frac{1}{(t-1)^2}\sum_i e_i \int_{E_i} g_\vphi \omega_\vphi^n\\
&&\frac{d}{dt}\Lam^\NA_g=\frac{1}{(t-1)^2} \sum_i e_i \int_{\rho^*E_i}g_\psi (\sddb\psi)^n\\
&&\frac{d}{dt} (\bfD^\NA_g-\epsilon \bfJ^\NA_g)=\frac{d}{dt}\left(\bfL^\NA-(1-\epsilon)\bfE^\NA_g-\epsilon \Lam^\NA_g\right)\\
&&\hskip 1.5cm
=\frac{1}{(t-1)^2}\sum_i (e_1-e_i)\left((1-\epsilon)\int_{E_i}g_\vphi \omega_\vphi^n+\epsilon\int_{\rho^*E_i}g_\psi (\sddb\psi)^n \right)\ge 0.
\end{eqnarray*}
\end{proof}

In the algebraic study of K-stability, one needs to consider more general data than test configurations. One such generalization is filtrations of section rings of $-K_X$. Such a filtration gives rise to a sequence of test configurations. For Fano varieties, it turns out that it suffices to consider filtrations associated valuations. We will only use divisorial valuations in this note although more general quasi-monomial valuations do play an important role in the recent study. A divisor over $X$ is an ordinary irreducible divisor $E$ on a projective variety $X'$ that has a birational morphism to $X$. By resolution of singularities, we can assume that $X'$ is smooth. Then the order of vanishing along $E$ defines a functional $\ord_E$ on the field $\bC(X)$ of rational functions on $X$. Any functional on $\bC(X)$ of the form $c\cdot \ord_E$ with $c\in \bQ_{>0}$ is called a divisorial valuation. The set of divisorial valuations will be denoted by $X^{\rm div}_\bQ$.  

There are two invariants of a divisorial valuation which play important roles in the study of K-stability. The first one is the log discrepancy defined by the formula $A_E(c \cdot \ord_E)=c\cdot (\ord_X(K_{X'/X})+1)$. The second one is defined as (see \eqref{eq-volg}, \eqref{eq-Sgv0})
\begin{equation}\label{eq-Sgv}
S_g(v)=\frac{1}{\bV_g}\int_0^{+\infty}\vol_g(\mu^*L-t E)dt.
\end{equation}

We will not give detailed discussion of the use of valuations, but just give some hint of its importance. 
There is a compactification of $X^{\rm div}$ called the Berkovich space and denoted by $X^{\rm an}$. It is a non-Archimedean analytic space and is a compactification of $X^{\rm div}_\bQ$ which is in turn dense in $X^{\rm an}$. One also has a non-Archimedean line bundle $L^{\rm an}$. Now any test configuration defines a non-Archimedean metric on $L^{\rm an}$.
Equivalent test configurations define the same metric. For us it suffices to know that such non-Archimedean metric can be represented by a function on $X^{\rm div}_\bQ$ which is the relative non-Archimedean potential with respect to the trivial metric (defined by the trivial test configuration) on $L^{\rm an}$. More precisely, for any test configuration $(\mcX, \mcL)$ which dominates the trivial test configuration by $\rho: \mcX\rightarrow X\times\bC=:X_\bC$, we set:
\begin{equation}\label{eq-phiTC}
\phi(v)=\phi_{(\mcX, \mcL)}=G(v)(\mcL-\rho^*L_\bC).
\end{equation} 
One important consequence of this point of view is the following valuative formula for the $\bfL^\NA$ invariant:
\begin{equation}\label{eq-Lvaluative}
\bfL^\NA(\mcX, \mcL)=\inf_{v\in X^{\rm div}_\bQ} (A_X(v)+\phi(v)).
\end{equation}
This is a variation (or a global analogue) of the well-known valuative formula for log canonical thresholds.  
We can state a valuative criterion for the reduced uniform $g$-weighted stability.

\begin{thm}[\cite{Li19, HL20}]\label{thm-valuative}
$(X, \bT)$ is reduced uniformly $g$-weighted stable if and only if there exists $\delta>1$ such that for any $\tilde{\bT}$-invariant divisorial valuation $v$, there exists $\xi\in \tilde{N}_\bQ$ such that
\begin{equation}
A_X(v_\xi)-\delta\cdot S_g(v_\xi)\ge 0 
\end{equation}
where $v_\xi$ is defined in \eqref{eq-vxi}.
\end{thm}
\begin{proof}
Assume that $X$ is reduced uniformly $g$-weighted stable. Then $\Fut_g\equiv 0$ on $\tilde{N}_\bQ$. Choose any $v\in X^{\rm div}_\bQ$ and let $\cF_v$ be the associated filtration. Then
\begin{eqnarray*}
\inf_{\xi\in \tilde{N}_\bQ}\left[\bfD^\NA_g-\gamma \bfJ^\NA_g)((\cF_{v})_\xi)\right]=\inf_{\xi\in \tilde{N}_\bQ}\left[(\bfD^\NA_g-\gamma \bfJ^\NA_g)(\cF_{v_\xi})\right]\ge 0.
\end{eqnarray*}
Now we have $A_X(v_\xi)\ge \bfL^\NA(\cF_{v_\xi})$, $\bfE_g^\NA(\cF_{v_\xi})=S_g(v_\xi)$ and 
\begin{equation}
\bfJ_g^\NA(\cF_{v_\xi})\sim \bfJ^\NA(\cF_{v_\xi})\sim S(v_\xi)\sim S_g(v_\xi)
\end{equation} 
where $\sim$ means the ratio of two quantities are uniformly bounded (independent of $v$ and $\xi$). 
The second inequality is proved by K. Fujita. So there exists $C>0$ such that
\begin{eqnarray*} 
\inf_{\xi\in N_\bR}\left[A_X(v_\xi)-(1+C\gamma) S_g(v_\xi)\right]&\ge& \inf_{\xi\in \tilde{N}_\bQ}\left[(\bfD_g^\NA-\gamma \bfJ^\NA_g)((\cF_v)_\xi)\right]\ge 0.
\end{eqnarray*}

For the other direction, there are two approaches. The first approach is by using the above theorem \ref{thm-special}. For a special test configuration $(\mcX, -K_{\mcX})$, 
let $v=r(\ord_{\mcX_0})\in X^{\rm div}_\bQ$ be the restriction of $\ord_{\mcX_0}$ to $\bC(X)$. Then $\cF:=\cF_{(\mcX, -K_{\mcX})}=\cF_v(-A_X(v))$. Then
\begin{eqnarray*}
\bfD_g^\NA(\mcX, -K_{\mcX})&=&-\bfE_g^\NA(\cF)=-\bfE_g^\NA(\cF_v(-A_X(v)))\\
&=&A_X(v)-\bfE_g^\NA(\cF_v)=A_X(v)-S_g(v)\\
&=&A_X(v_\xi)-S_g(v_\xi)+\Fut_g(\xi). 
\end{eqnarray*}
The theorem then follows easily. The second approach uses the inequality \eqref{eq-phiSgeE} and is also used in the proof of Theorem \ref{thm-YTDlog}.

\end{proof}

We now have the following fundamental result:
\begin{thm}[\cite{LXZ21, BLXZ}]\label{thm-BLXZ}
$(X, \bT)$ is reduced $g$-weighted stable iff $(X, \bT)$ is $\tilde{\bT}$-equivariantly $g$-weighted Ding-polystable, iff $(X, \bT)$ is $g$-weighted Ding-polystable. 
\end{thm}
The rough idea to prove this is to look for the minimizer of the functional $\delta_\bT(v)=\sup_{\xi\in \tilde{N}_\bR}\frac{A_X(v_\xi)}{S_g(v_\xi)}$. Using deep MMP techniques and recently developed boundedness results, one can show that the minimizer $v_*$ is quasi-monomial (see \cite{BLX19}). The question is then reduced to a finite generation problem for the quasi-monomial valuation $v_*$. This has been resolved in \cite{BLXZ} based the important work \cite{LXZ21}.

\subsection{Yau-Tian-Donaldson conjecture}

A psh ray $\Phi=\{\vphi(s)\}_{[0, +\infty)}$ is a geodesic ray if for any $s_1, s_2\in [0, +\infty)$, $\Phi|_{[s_1,s_2]}$ is a geodesic segment between $\vphi(s_1)$ and $\vphi(s_2)$. It is known that both $\sup(\vphi(s)-\vphi_0)$ and $\bfE_g(\vphi(s))$ is affine with respect to $s$. Any geodesic ray defines a non-Archimedean metric on $L^\NA$ which is represented by a function on $X^{\rm div}_\bQ$ defined as:
\begin{equation}
\Phi_\NA(v)=-G(v)(\Phi), \quad \text{ for any } v\in X^{\rm div}_\bQ
\end{equation}
where $G(v)$ is the Gauss extension of $v$ and $G(v)(\Phi)$ is the generic Lelong number with respect to $G(v)$ (see the paragraph after Theorem \ref{thm-multiplier} for more details). 

Because $\mcL$ is relatively ample, there always exists locally bounded psh Hermitian metrics on $\mcL$ which we call the psh ray.
\begin{thm}[\cite{PS07}]
Fix a reference metric $\vphi_0\in \mcH(L)$.
For any test configuration $(\mcX, \mcL)$, there exists a unique geodesic ray emanating from $\vphi_0$.
\end{thm}

The following result connects the Archimedean and non-Archimedean functionals. 

\begin{thm}
For any test configuration $(\mcX, \mcL)$, let $\Phi=\{\vphi(s)\}$ be any locally bounded psh ray induced by a locally bounded psh Hermitian metric on $\mathcal{L}$. Then we have the following slope formula:
for any $\bfF\in \{\bfE_g, \bfI_g, \bfJ_g, \Lam_g, \bfL\}$, 
\begin{equation}
\bfF'^\infty(\Phi)= \lim_{s\rightarrow+\infty}\frac{\bfF(\vphi(s))}{s}=\bfF^\NA(\mcX, \mcL). 
\end{equation}
\end{thm}
Applying this theorem to both sides of the inequality \eqref{eq-Dingineq}, we get its non-Archimedean version. 
\begin{cor}
Let $(\mcX, \mcL)$ be any test configuration. 
For any $t\in (0,1]$, we have the inequality:
\begin{equation}\label{eq-JNAineq}
\bfJ^\NA_g(\mcX, t \mcL+(1-t)L_\bC)\le t^{-C_3} \bfJ^\NA_g(\mcX, \mcL).
\end{equation} 
\end{cor}

The following is a version of the Yau-Tian-Donaldson conjecture. 
\begin{thm}[\cite{HL20}, see also \cite{His19}]\label{thm-YTDsm}
$(X, \bT)$ admits a $g$-soliton if and only if $(X, \bT)$ is reduced uniformly $g$-weighted stable. 
\end{thm}
\begin{proof}[Sketch of proof]
We just need to show that $\bfM_g$ is $\tilde{\bT}$-coercive. Assuming not, then for any fixed $\epsilon>0$ there exists a sequence of $\vphi_j$ satisfying:
\begin{equation*}
\sup \vphi_j=0, \quad \bfJ_g(\vphi_j)=\inf_{\sigma\in \tilde{\bT}}\bfJ_g(\sigma^*\omega_{\vphi_j}), \quad
\bfM_g(\vphi_j)\le \epsilon \bfJ_g(\vphi_j)-j. 
\end{equation*}
Because $\bfM_g=\bfH_g-(\bfI_g-\bfJ_g)\ge -C_2 \bfJ_g$ (since $\bfH_g(\vphi)\ge 0$), we get:
\begin{equation*}
\bfJ_g(\vphi_j)=-\bfE_g(\vphi_j)+O(1)\rightarrow +\infty.  
\end{equation*}
Set $S_j=-\bfE_g(\vphi_j)\rightarrow +\infty$ and let $\{\vphi_j(s)\}_{s\in [0, S_j]}$ be geodesic segment connecting $\vphi_0$ to $\vphi_j$. By using the convexity of Mabuchi functional and the compactness results in Theorem \ref{thm-E1compact}, we can extract strongly convergent sequence and construct a geodesic ray $\Phi=\{\vphi(s)\}_{s\in [0, +\infty)}$ satisfying (see \cite{DR17, BBJ18, Li20})
\begin{equation*}
\inf_{\xi\in N_\bR}\bfJ'^\infty_g(\Phi_\xi)=1, \quad \bfM'^\infty_g(\Phi)\le \epsilon
\end{equation*}
where $\Phi_\xi=\{\vphi_\xi(s)\}_{s\in [0, +\infty)}$ with $h_0 e^{-\vphi_\xi(s)}=\exp(s \xi)^*(h_0e^{-\vphi(s)})$. 
This implies
\begin{equation}\label{eq-Ddest}
\bfL'^\infty(\Phi) \le \bfD'^\infty_g(\Phi)+\bfE'^\infty_g(\Phi)\le \bfM'^\infty_g(\Phi)-1\le -1+\epsilon.
\end{equation} 

The idea due to Berman-Bouksom-Jonsson is now to approximate the geodesic ray by a sequence of test configurations by blowing up 
multiplier ideal sheaves.
Let $\cJ(m\Phi)$ be the multiplier ideal sheaf. Let $\mu_m: \mcX_m\rightarrow \mcX$ be the blow up of $\cJ(m\Phi_m)$ with exceptional divisor $E_m$. Set $\mcL_m=\mu_m^* L_\bC-\frac{1}{m+m_0}E_m$. Then by the global generation property of multiplier ideals, one can prove that $\mcL_m$ is globally generated. So $(\mcX, \mcL_m)$ is a test configuration of $(X, L)$. Moreover
\begin{equation}
\lim_{m\rightarrow+\infty}\phi_m(v)=-G(v)(\Phi)=\Phi_\NA(v), \quad \text{ for any } v\in X^{\rm div}_\bQ.
\end{equation} 

Next we want to approximate the slopes by non-Archimedean quantities of test configurations. 
First by Demailly's result (Theorem \ref{thm-Demailly}), $\Phi_m\ge \Phi$, where $\Phi_m$ is the geodesic rays associated to $(\mcX_m, \mcL_m)$. This implies $(\Phi_m)_\xi\ge \Phi_\xi$ for any $\xi\in \tilde{N}_\bR$ where $\Phi_\xi=\{\sigma_\xi(s)^*\vphi(s)\}$, and
\begin{equation}
\bfE^\NA_g((\phi_m)_\xi)=\bfE'^\infty_g((\Phi_m)_\xi) \ge \bfE'^\infty_g(\Phi_\xi), \quad \Lam^\NA_g((\phi_m)_\xi)=\Lam'^\infty_g((\Phi_m)_\xi)\ge \Lam'^\infty_g(\Phi_\xi)
\end{equation}
where $\phi_m=\phi_{(\mcX_m, \mcL_m)}$ represents (the non-Archimedean metric associated to) the test configuration $(\mcX_m, \mcL_m)$. 
 To compare the $\bfL$ quantities, we use the following formula similar to \eqref{eq-Lvaluative} (see \eqref{eq-LPhival}) which is proved by Berman-Boucksom-Jonsson using the valuative description of multiplier ideals and Theorem \ref{thm-multiplier} (whose proof uses the strong openness conjecture proved by Guan-Zhou).
\begin{equation}
\bfL'^\infty(\Phi)=\inf_{v\in X^{\rm div}_\bQ} (A_X(v)+\Phi_\NA(v)).
\end{equation}
Moreover by using the valuative description of multiplier ideals it is easy to prove:
\begin{equation}
\lim_{m\rightarrow+\infty}\bfL^\NA(\mcX_m, \mcL_m)=\bfL'^\infty(\Phi).
\end{equation}
By reduced uniform stability, we know that there exists $\xi_m\in \tilde{N}_\bR$ such that 
\begin{equation*}
0\le (-\bfE_g^\NA+\bfL^\NA-\gamma (\Lam^\NA_g-\bfE_g^\NA))((\phi_m)_\xi).
\end{equation*} 
So we can estimate as follows:
\begin{eqnarray*}
\bfL^\NA(\phi_m)&\ge& \gamma \Lam^\NA_g((\phi_m)_{\xi_m})+(1-\gamma)\bfE^\NA_g((\phi_m)_{\xi_m})\\
&\ge& \gamma \Lam'^\infty_g(\Phi_{\xi_m})+(1-\gamma) \bfE'^\infty_g(\Phi_{\xi_m})\\
&=& \gamma \bfJ'^\infty_g(\Phi_{\xi_m})+ \bfE'^\infty_g(\Phi_{\xi_m})\\
&=& \gamma \cdot 1+\bfE'^\infty_g(\Phi)+\Fut_g(\xi_m)\\
&=& \gamma-1. 
\end{eqnarray*}
By letting $m\rightarrow+\infty$, we get $\bfL'^\infty(\Phi)\ge \gamma-1$. This contradicts \eqref{eq-Ddest} when $\epsilon< \gamma$. 
\end{proof}
As an application of this theorem, one can recover the existence results in the toric case as first proved in \cite{WZ04} (see also \cite{BB17}). 
Moreover combining Theorem \ref{thm-YTDsm} and Theorem \ref{thm-BLXZ}, we indeed get the polystable version of Yau-Tian-Donaldson conjecture.

\begin{rem}
It is natural to ask whether one can use the original proof in \cite{CDS15, Tia15} to prove the YTD conjecture. This requires to prove compactness and partial $C^0$-estimates for the Ricci form $Ric(\omega)+\sddb f$. In the case of ordinary K\"{a}hler-Ricci soliton, this has been studied first in \cite{WZ19, JWZ17} and later in \cite{DaS16}.   However there are technical difficulties (pointed out by Feng Wang). For example, because one wants to control the Riemann geometry by using the full Bakry-Emery Ricci tensor $Ric+\nabla^2 f$, in general it is non-trivial to deal with the extra term $Re(\bar{\nabla}\bar{\nabla} f)$ (which vanishes for ordinary K\"{a}hler-Ricci solitons). 
\end{rem}

\subsection{Log Fano case}

A log pair $(X, D)$ consists a projective normal variety and a Weil divisor such that $K_X+D$ is $\bQ$-Cartier. We say that $(X, D)$ is log Fano if the $\bQ$-Cartier divisor $-(K_X+D)$ is ample and $(X, D)$ has klt singularities. Given any pair $(X, D)$, we can choose a log resolution of singularities $\mu: X'\rightarrow X$ such that $\mu$ is an isomorphism over the simple normal crossing locus of $(X, D)$ and ${\rm supp}(\mu^{-1}((X, D)^{\rm sing})\cap \mu_*^{-1}D)$ has simple normal crossings, where $(X, D)^{\rm sing}$ is the non-simple-normal-crossing locus. We then have an identity:
\begin{equation}\label{eq-disciden}
K_{X'}=\mu^*(K_X+D)+\sum_i a_i E_i.
\end{equation}
Define $A_{(X, D)}(E_i)=a_i+1$. $(X, D)$ is klt if $A_{(X, D)}(E_i)>0$. It is known that
this holds true if and only if $A_{(X, D)}(v)>0$ for any $v\in X^{\rm div}_\bQ$. 

Examples of (local) klt singularities include orbifold singularities and log Fano cones. 
\begin{exmp}
(1) If $X$ is smooth projective, and $D$ is an effective divisors. Then $(X, D)$ has klt singularities if and only if $\mcJ(D)=\mcO_{X}$ where $\mcJ(D)$ is the multiplier ideal of $D$. If $D$ has simple normal crossings, then this is equivalent to that the coefficients of $D$ lie in $(0,1)$. \\
(2)
$(X, D)$ is an orbifold if for any $p\in X$ there exists a neighborhood $U$ of $p$ such that $U\cong \bC^n/\Gamma$ where $\Gamma$ is a finite group acting on $\bC^n$ linearly, and $D=\sum_k (1-d_k^{-1})D_k$ where $D_k$ is the locus of points with stabilizer $\bZ_{d_k}$. Any orbifold pair $(X, D)$ has klt singularities.
\\
(2)
Let $(S, \Delta)$ be a log Fano pair and assume $L=-\gamma^{-1} (K_S+\Delta)$ for some $\gamma\in \bQ_{>0}$ is a line bundle. Set $X={\rm Spec}(\bigoplus_{m} H^0(S, mL))$ and $D={\rm Spec} \left(\bigoplus_m H^0(\Delta, mL)\right)$. Then (local) pair $(X, D)$ has klt singularities. 
\end{exmp}

The pluripotential analysis can be extended to the log Fano case. When $g=1$, this was studied in detail in \cite{BBEGZ}. The Yau-Tian-Donaldson conjecture was established in \cite{HL20}.

\begin{thm}\label{thm-YTDlog}
$(X, D, \bT)$ admits a $g$-soliton metric if and only if $(X, D, \bT)$ is reduced uniformly $g$-weighted stable. 
\end{thm}

The main technical difficulty to carry over the proof in the smooth case to the singular case is that the multiplier ideal approximation in the singular case is not so well-behaved. In particular, Demailly's result from Theorem \ref{thm-Demailly} is not true in the singular case. So the key idea to overcome this is to use a perturbation method. This was introduced and developed in \cite{LTW17, LTW19, Li19}. A rough sketch of the proof of Theorem \ref{thm-YTDlog} is given as follows.

\vskip 1mm
{\bf (1):} Assume that $\bfM_g$ is not $\bG$-coercive. There exists a destabilizing geodesic ray $\Phi=\{\vphi(s)\}_{s\in [0, +\infty)}$ satisfying:
\begin{equation}
\bfD'^\infty_g(\Phi)\le \bfM'^\infty_g(\Phi)\le 0, \quad \inf_{\xi\in N_\bR}\bfJ'^\infty_g(\Phi_\xi)\ge 1, \quad \sup_X \vphi(s)=0. 
\end{equation}

{\bf (2): }
Because the resolution of singularities is a composition of blowups along smooth centers, 
there exists $E_\theta=\sum_k \theta_k E_k$ with $\theta_k\ge 0$ such that $H:=\mu^*(-K_X-D)-E_\theta$ is semiample. 
Set $B_\epsilon=\sum_k (-a_k+\frac{\epsilon}{1+\epsilon}\theta_k E_k)$. Then we can rewrite the 
identity \eqref{eq-disciden} as:
\begin{eqnarray*}
-K_{X'}-B_\epsilon&=&-K_{X'}+\sum_k a_k E_k-\frac{\epsilon}{1+\epsilon}E_\theta=\mu^*(-K_X-D)-\frac{\epsilon}{1+\epsilon}E_\theta\\
&=&\frac{1}{1+\epsilon}(\mu^*(-K_X-D)+\epsilon H)=:L_\epsilon.
\end{eqnarray*}
For fixed $0<\epsilon\ll 1\in \bQ$, consider the subgeodesic ray $\Phi_\epsilon=\frac{\Phi+\epsilon \psi_H}{1+\epsilon}$ over $L_\epsilon$ where $\psi_H$ is a fixed smooth psh Hermitian metric on $P$. For $m\gg 1$ sufficiently divisible, $p_1^*L_\epsilon\otimes \cJ(X', m\Phi_\epsilon)$ is globally generated. 
Construct test configurations $(\mcX'_m, \mcB_{\epsilon,m}, \mcL_{\epsilon,m})$ of $(X', L_\epsilon)$ by blowing up the multiplier ideal $\mcJ(X', m {\Phi}_\epsilon)$. Set 
\begin{equation}
\mcL_{\epsilon, m}=\pi'^*_m p'^*_1 L_\epsilon-\frac{1}{m}E_{\epsilon, m}.
\end{equation}
Consider the associated non-Archimedean quantities:
\begin{eqnarray*}
\bfL^\NA_{(X', B_\epsilon)}(\phi_\epsilon)&=&\inf_{v\in X^{\rm div}_\bQ}(A_{(X', B_\epsilon)}(v)-G(v)(\Phi_\epsilon))\\
\bfL^\NA_{(X', B_\epsilon)}(\phi_{m, \epsilon})&=&\inf_{v\in X^{\rm div}_\bQ}(A_{(X', B_\epsilon)}(v)+\phi_{m,\epsilon}(v)).
\end{eqnarray*}
One can show that: for any $\xi\in N_\bR$ and $\bfF\in \{\bfE, \Lam\}$, 
\begin{eqnarray*}
\bfF^\NA_g(\phi_{m,\epsilon,\xi})\ge \bfF'^\infty_g(\Phi_{\epsilon, \xi}), \quad \lim_{\epsilon \rightarrow+\infty}\bfF'^\infty_g(\Phi_{\epsilon,\xi})=\bfF'^\infty_g(\Phi_\xi).
\end{eqnarray*}
We also have the convergence:
\begin{eqnarray}\label{eq-convLNA}
\lim_{m\rightarrow+\infty} \bfL^\NA_{(X', B_\epsilon)}(\phi_{m,\epsilon})=\bfL^\NA_{(X', B_\epsilon)}(\phi_\epsilon), \quad 
\lim_{m\rightarrow+\infty} \bfL^\NA_{(X', B_\epsilon)}(\phi_\epsilon)=\bfL'^\infty(\Phi). 
\end{eqnarray}

{\bf (3):} 
Another key estimate we need is 
\begin{lem}
Let $v\in (X^{\rm div}_\bQ)^\bT$ for a $\bT$-invariant divisorial valuation. 
Assume that $A_{(X, D)}(v)\ge \delta \cdot S_g(v)$ for some $\kappa>0$. There exists $\delta'>1$ independent of $\epsilon$ and $v$ such that $A_{(X', B_\epsilon)}(v)\ge \delta' \cdot S_{g, L_\epsilon}(v)$ for any $\epsilon$ sufficiently small. 
\end{lem}
\begin{proof}
By possibly doing further blowing-ups and rescaling invariance of the statement, we can assume that $v=\ord_F$ where $F$ is an ordinary $\bT$-invariant divisor on $X'$. 
We set $f_{\epsilon}(v)=\frac{A_{(X', B_\epsilon)}(v)}{S_{g,L_\epsilon}(v)}$ and $f_0(v)=\frac{A_{(X,D)}(v)}{S_g(v)}$. We estimate:
\begin{eqnarray*}
\frac{f_{\epsilon}(v)}{f_0(v)}&=&\frac{A_{(X',B_\epsilon)}(v)}{A_{(X,D)}(v)}\frac{\bV_g}{\bV_{g,L_\epsilon}}\frac{\int_0^{+\infty}\vol_g(\mu^*L-tF)dt}{\int_0^{+\infty}\vol_{g}(L_\epsilon-tF) dt}.
\end{eqnarray*}
The first factor can be estimated as follows:
\begin{eqnarray*}
\frac{A_{(X', B_\epsilon)}(v)}{A_{(X,D)}(v)}&=&\frac{A_{X'}(v)-v(B_\epsilon)}{A_{X'}(v)-v(B_0)}=\frac{A_{X'}(v)-v(B_0)-\frac{\epsilon}{1+\epsilon}v(E_\theta)}{A_{X'}(v)-v(B_0)}\\
&=&1-\frac{\epsilon}{1+\epsilon}\left(\frac{A_{X'}(v)-v(B_0)}{v(E_\theta)}\right)^{-1}\ge 1-\frac{\epsilon}{1+\epsilon}\left(\lct(X', B_0; E_\theta)\right)^{-1}.
\end{eqnarray*}
The right-hand-side  does not depend on $v$ and approaches 1 as $\epsilon\rightarrow 0$. 
The second ratio
$\frac{\bV_g}{\bV_{g,L_\epsilon}}$, which does not depend on $v$, approaches $1$ as $\epsilon\rightarrow 0$. This can be seen using the analytic expression of integration over the underlying space in \eqref{eq-bVg}.
We estimate the third ratio by simply estimating the integrand:
\begin{eqnarray*}
\vol_g(L_\epsilon-tF)=\vol_g(\mu^*L-\frac{\epsilon}{1+\epsilon}E_\theta-tF)\le \vol_g(\mu^*L-tF)
\end{eqnarray*}
because $E_\theta$ is effective. So the third ratio is always greater than $1$. The statement follows easily. 
\end{proof}

{\bf (4):} 
For any $k\in \bZ_{>0}$, there exists $v_k$ such that:
\begin{equation}
\bfL'^\infty(\Phi)\le A_X(v_k)+\Phi_\NA(v_k)<\bfL'^\infty(\Phi)+k^{-1}. 
\end{equation}
By the valuative criterion in Theorem \ref{thm-valuative} which also holds in the log Fano case, 
there exists $\xi_k\in \tilde{N}_\bR$ such that $A_X(v_{k,\xi_k})\ge \delta S_g(v_{k,\xi_k})$. 

Moreover, for a fixed $k$ and for any $\beta>0$, there exist $\epsilon_0, m_0$ such that for any $\epsilon<\epsilon_0$ and $m\ge m_0$:
\begin{equation}\label{eq-LNAuniform}
\max\left\{\left|A_{(X', B_\epsilon)}(v_k)-A_{X}(v_k)\right|, \left|\phi_{\epsilon,m}(v_k)-\Phi_\NA(v_k)\right|, \left|\bfL^\NA(\phi_{m,\epsilon})-\bfL'^\infty(\Phi)\right|\right\}<\beta.
\end{equation}
With the above estimates, we can perturb the earlier argument to complete the contradiction: 
\begin{eqnarray*}
&&\bfL^\NA(\phi_{\epsilon,m})+3\beta+k^{-1} \ge \bfL'^\infty(\Phi)+2\beta+k^{-1}\\
&\ge&A_X(v_k)+\Phi_\NA(v_k)+2\beta\ge A_{(X', B_\epsilon)}(v_k)+\phi_{\epsilon,m}(v_k) \\ 
&=&  A_{(X', B_\epsilon)}(v_k)+\phi_{\epsilon,m}(v_k)= A_{(X', B_\epsilon)}(v_{k,\xi_k})+\phi_{\epsilon, m, -\xi_k}(v_{k,\xi_k})\quad \quad (\text{ by } \eqref{eq-Aphixi} )\\
&\ge& \delta' S_{g,L_\epsilon}(v_{k,\xi_k})+\phi_{\epsilon, m,-\xi_k}(v_{k,\xi_k})\ge \delta'\cdot \bfE^\NA_{g,L_\epsilon}(\delta'^{-1}\phi_{m,-\xi_k})\quad (\text{by } \eqref{eq-phiSgeE})\\
&\ge& (-\delta'\cdot \bfJ^\NA_g(\delta'^{-1}\phi_{\epsilon, m, -\xi_k})+\bfJ^\NA_g(\phi_{v_k,\xi_k}))+\delta'\cdot \bfE^\NA_{g,L_\epsilon}(\phi_{\epsilon, m, -\xi_k})\\
&\ge&(1-\delta'^{-1/C_2})\bfJ_{g,L_\epsilon}^\NA(\phi_{\epsilon, m,-\xi_k})+\delta' \cdot \bfE^\NA_{g,L_\epsilon}(\phi_{\epsilon, m,-\xi_k})\\
&=&(1-\delta'^{-1/C_2}) \Lam_{g,L_\epsilon}^\NA(\phi_{\epsilon, m,-\xi_k})+\delta'^{-1/C_2} \bfE^\NA_{g,L_\epsilon}(\phi_{\epsilon, m, -\xi_k})\\
&\ge&(1-\delta'^{-1/C_2}) \Lam'^\infty_{g, L_\epsilon}(\Phi_{\epsilon, -\xi_k})+\delta'^{-1/C_2} \bfE'^\infty_{g,L_\epsilon}(\Phi_{\epsilon, -\xi_k})\\
&=&(1-\delta'^{-1/C_2}) \bfJ'^\infty_{g, L_\epsilon}(\Phi_{\epsilon, -\xi_k})+ \bfE'^\infty_{g, L_\epsilon}(\Phi_{\epsilon, -\xi_k}).
\end{eqnarray*}
Letting $m\rightarrow+\infty$ and then $\epsilon\rightarrow 0$, one gets:
\begin{eqnarray*}
\bfL'^\infty(\Phi)+O(k^{-1})&\ge& (1-\delta'^{-1/C_2})\bfJ'^\infty_g(\Phi_{-\xi_k})+\bfE'^\infty_g(\Phi_{-\xi_k})\\
&\ge& (1-\delta'^{-1/C_2}) \bfJ'^\infty_g(\Phi_{-\xi_k})+\bfE'^\infty_g(\Phi)\\
&\ge& -\delta'^{-1/C_2}.
\end{eqnarray*}
Letting $k\rightarrow+\infty$, we get contradiction to \eqref{eq-Ddest} as long as $-1+\epsilon<-\delta'^{-1/C_2}$ i.e. $\epsilon<1-\delta'^{-1/C_2}$.

\section{Application of $g$-soliton equations to K\"{a}hler Ricci-flat cone metrics}

In this section, we first revisit the recent works of Apostolov et al. in \cite{AC21, AJL21, ACL20}. They discovered a connection between a particular $g$-soliton
equation and Sasaki-Einstein metrics. This connection goes roughly as follows. Each Sasaki manifold has a CR and contact structure, and an associated Tanaka-Webster connections which coincides with the Levi-Civita connection. 
The change of Reeb of vector field in the Reeb cone changes the associated Tanaka-Webster connection. The transformation formula for the Ricci curvature of the Tanaka-Webster connection, which was well-known in the CR literature (for example from \cite[2]{JL88}), implies that Sasaki-Einstein metrics can be described as a $g$-soliton in the tensorial form on a quasi-regular quotient. Here we will understand their work via a different point of view by considering (transversal) Monge-Amp\`{e}re equation that is equivalent to the Ricci-flat K\"{a}hler cone equation on affine varieties. We will see that this approach is in some way simpler and shows immediately that, more generally, $g$-soliton equations for different Reeb vector fields are related to each other. Finally we discuss how to combine this transformation and the Yau-Tian-Donaldson conjecture for $g$-solitons on general log Fano pairs to derive the Yau-Tian-Donaldson conjecture for polarized Fano cones.  

\subsection{K\"{a}hler Ricci-flat cone metric}
We will first review the Ricci-flat K\"{a}hler cone metric following the work of Martelli-Sparks-Yau in \cite{MSY08}.
Assume that $Y$ is an affine variety with an isolated singularity $o\in Y$. Assume that there exists an effective torus action by $\hat{\bT}\cong (\bC^*)^{r+1}$ and $o\in Y$ is the only fixed point. Assume that $Y={\rm Spec}_\bC(R)$. 
Set $\hat{M}_\bZ={\rm Hom}_{\rm alg}(\hat{\bT}, \bC^*)$ and $\hat{N}_\bZ={\rm Hom}_{\rm alg}(\bC^*, \hat{\bT})$. 
The $\hat{\bT}$-action corresponds to a weight decomposition:
\begin{equation*}
R=\bigoplus_{\hat{\alpha}\in \hat{M}_\bZ} R_{\hat{\alpha}}. 
\end{equation*}
$\hat{\bN}_\bZ$ is isomorphic to $\bZ^{r+1}$. 
Set $\hat{N}_\bQ=\hat{N}_\bZ\otimes \bQ\cong \bQ^{r+1}$, $\hat{N}_\bR=\hat{N}_\bZ\otimes \bR\cong \bR^{r+1}$. $\hat{N}_\bR$ can be identified with the Lie algebra of the compact torus $\hat{T}:=(S^1)^{r+1} \subset \hat{\bT}$. For each $\hat{\xi}\in \hat{N}_\bR$, there is an associated real holomorphic vector field (still denoted by $\hat{\xi}$) and a holomorphic $(1,0)$-vector field $v_{\hat{\xi}}=\frac{1}{2}(-J\hat{\xi}-\sqrt{-1}\hat{\xi})$ where $J$ is the complex structure on $Y^*:=Y\setminus \{o\}$. We define the Reeb cone as follows:
\begin{equation}
\hat{N}_\bR^+=\{\hat{\xi}\in \hat{N}_\bR; \la \hat{\alpha}, \hat{\xi}\ra >0 \text{ for any } \hat{\alpha}\neq 0 \text{ with } R_{\hat{\alpha}} \neq 0 \}.
\end{equation}
Any $\hat{\xi}\in \hat{N}_\bQ^+=\hat{N}_\bR^+\cap \hat{N}_\bQ$ is called quasi-regular. Otherwise it is called irregular. 
For any quasi-regular $\hat{\xi}\in \hat{N}_\bQ^+$, $v_{\hat{\xi}}$ generates a subgroup $\la v_{\hat{\xi}}\ra \cong \bC^* \subset \hat{\bT}$. The quotient of $Y/\la v_{\hat{\xi}}\ra$ is an orbifold $(X, D)$ equipped with an orbifold line bundle $L$ such that $R=\bigoplus_m H^0(X, \lfloor mL \rfloor)$.
\begin{defn}
We say that a $\hat{T}$-invariant function $r: X\rightarrow \bR_{>0}$ is a radius function if 
\begin{enumerate}
\item
$\hat{\omega}=\sddb r^2$ is a K\"{a}hler cone metric on $Y$ with radius function $r$. In other words, if we set $G=\frac{1}{2}\hat{\omega}(\cdot, J\cdot)$ then $G$ is a Riemannian metric on $Y^*=Y\setminus \{o\}$ that is isometric to $dr^2+r^2 g_S$ where $S=\{r=1\}$ and $g_S=G|_S$. 
\item
There exists $\hat{\xi}\in \hat{N}_\bR^+$ such that $J(r\partial_r)=\hat{\xi}$. We say that $r$ is a radius function for $\hat{\xi}$. 
\end{enumerate}
\end{defn}
Define a distribution $\mathfrak{D}$ on $Y^*=Y\setminus \{o\}$ by $\mathfrak{D}=\bR\{ r\partial_r, J (r\partial_r)\}^{\perp}$. By using the identity $\mathfrak{L}_{r\partial_r}G=2r G$ and $\mathfrak{L}_{\hat{\xi}}G=0$, one easily verifies that $\mathfrak{D}$ satisfies:
\begin{equation}\label{eq-Lieinv}
\mathfrak{L}_{r\partial_r}\mathfrak{D}\subseteq \mathfrak{D}, \quad \mathfrak{L}_{\hat{\xi}}\mathfrak{D}\subseteq \mathfrak{D}.
\end{equation}
Define a 1-form on $Y^*$ by:
\begin{equation}\label{eq-etar}
\eta:=\eta_r=-Jd \log r=-\frac{1}{2}r^{-2} J d r^2
\end{equation}
Then $\eta$ is uniquely determined by the condition:
\begin{equation}\label{eq-etaprop}
\eta(r\partial_r)=0, \quad \eta(\hat{\xi})=1, \quad \eta|_{\mathfrak{D}}=0. 
\end{equation} 
We have the identity:
\begin{equation}
\hat{\omega}=\sddb r^2=-\frac{1}{2} dJ dr^2=
d (r^2 \eta)=2r dr\wedge \eta+r^2 d\eta.
\end{equation} 
From this formula, one can also recover the associated Riemann metric on $Y^*$ by using the data $\{\hat{\xi}, \eta, r\}$:
\begin{eqnarray}
&&\left.G\right|_{\mathfrak{D}}=\frac{1}{2}r^2 d\eta (\cdot, J\cdot), \quad G(\hat{\xi}, \mathfrak{D})=0=G(J\hat{\xi}, \mathfrak{D}),\\
&& G(\hat{\xi}, \hat{\xi})=r^2=G(J\hat{\xi}, J\hat{\xi}), \quad G(\hat{\xi}, J\hat{\xi})=0.
\end{eqnarray}

For any $\hat{\zeta}\in \hat{N}_\bR$, there is a decomposition $\hat{\zeta}=\eta(\hat{\zeta}) \hat{\xi}+\zeta^h$ with $\zeta^h\in \mathfrak{D}$. 
Because $[r\partial_r, \hat{\zeta}]=0=[\hat{\xi}, \hat{\zeta}]$, by using \eqref{eq-Lieinv} we easily get 
\begin{equation}\label{eq-etainv}
r\partial_r(\eta(\hat{\zeta}))=0=\hat{\xi}(\eta(\hat{\zeta})), \quad [\hat{\xi}, \zeta^h]=0.
\end{equation}
Moreover, we have the identity 
\begin{equation}\label{eq-Hamcone}
d (r^2\eta(\hat{\zeta}))=d \iota_{\hat{\zeta}}(r^2 \eta)=\mathfrak{L}_{\hat{\zeta}} (r^2\eta)-\iota_{\hat{\zeta}} (d(r^2 \eta))=-\iota_{\hat{\zeta}} \hat{\omega}. 
\end{equation}
Note that \eqref{eq-Hamcone} is equivalent to the condition that $f=r^2 \eta(\hat{\zeta})$ satisfies:
\begin{equation}
\hat{\xi}(f)=0 \text{ and } \quad
d f|_{\mathfrak{D}}=-\iota_{\hat{\zeta}^h}\hat{\omega}|_{\mathfrak{D}}. 
\end{equation}
By \eqref{eq-Hamcone} we have a moment map for the $\hat{T}$-action with respect to $\hat{\omega}$.
\begin{eqnarray*}
\mu=\mu_r: Y^*&\longrightarrow& \hat{N}_\bR^*
\end{eqnarray*}
satisfying 
\begin{equation}\label{eq-cmoment}
\la \mu, \hat{\zeta}\ra = r^2 \eta(\hat{\zeta})=-\frac{1}{2} J dr^2 (\hat{\zeta}).
\end{equation}
It is known that the image of $\mu$ is a convex rational polyhedral cone $\mathcal{C}^*\subset M_\bR$. Moreover, the dual cone 
\begin{equation*}
\mathcal{C}=\{ \hat{\zeta}\in N_\bR; \la \hat{y}, \hat{\zeta}\ra >0 \text{ for all } \hat{y}\in \mathcal{C}^* \}
\end{equation*} 
is the same as the Reeb cone (see \cite{CS18, ACL20}). In other words, $\hat{\zeta}\in N^+_\bR$ if and only if $\eta(\hat{\zeta})>0$.

Set $\DHM_T=\mu_*(\hat{\omega}^{n+1})$ and $\DHM^{\hat{\xi},S}_T=(\mu|_S)_*((d\eta)^n\wedge \eta)$. Then we have the identity:
\begin{equation}
\DHM_T=(n+1)\ell_{\hat{\xi}}^{n} d \ell_{\hat{\xi}} \wedge \DHM^{\hat{\xi},S}_T.
\end{equation} 
Indeed, choosing any continuous function $f(\hat{y})$ of compact support, we get:
\begin{eqnarray*}
\int_{\mathcal{C}^*}f(\hat{y})\mu_*(2 (n+1)r^{2n+1}dr\wedge (d\eta)^n\wedge \eta)&=&\int_{Y} f(\mu(q)) 2(n+1) r^{2n+1} dr\wedge (d\eta)^n\wedge \eta\\
&=&(n+1)\int_Y f(\mu(q)) d \la \mu(q), \hat{\xi}\ra^{n+1} \wedge (d\eta)^n\wedge \eta\\
&=&(n+1)\int_{\mathcal{C}^*} f(\hat{y}) \ell_{\hat{\xi}}^n d\ell_{\hat{\xi}} \wedge \DHM^{\hat{\xi}, S}_T. 
\end{eqnarray*}
The manifold $S=\{r=1\}$ has an induced Sasaki structure $(\eta, \hat{\xi}, \mathfrak{D})$. We will sometimes emphasize this Sasaki structure by writing $S^{\hat{\xi}}$. 
Because $\eta(\hat{\xi})=1$, 
its image $\mu(S)$ is given by: 
\begin{equation*}
P_{\hat{\xi}}=\left\{\hat{y}\in \mathcal{C}^*; \la \hat{y}, \hat{\xi}\ra=1\right\}
\end{equation*}
which is a cross section of the cone $\mathcal{C}^*$ with normal vector $\hat{\xi}$. 
We then have the identity:
\begin{eqnarray}
\vol(S)&:=&\vol(S^{\hat{\xi}})=\int_{S} (d\eta)^n\wedge \eta=\int_{P_{\hat{\xi}}} \DHM^{\hat{\xi},S}_T\nonumber \\
&=&\frac{1}{(n+1)!}\int_Y e^{-r^2} (\sddb r^2)^{n+1}=\frac{1}{n!}\int_{\mathcal{C}^*}e^{-\ell_{\hat{\xi}}} \ell_{\hat{\xi}}^n d\ell_{\hat{\xi}}\wedge \DHM^{\hat{\xi},S}_T. \label{eq-volS}
\end{eqnarray}
From now on, we also assume that $Y$ is $\bQ$-Gorenstein. $\hat{\bT}$ naturally acts on $m K_Y$ for $m$ sufficiently divisible. 
Assume that $s\in |mK_Y|$ is a $\hat{\bT}$-equivariant nowhere vanishing section. It defines a volume form on $X$:
\begin{equation}
dV_Y=(\sqrt{-1}^{m(n+1)^2} s\wedge \bar{s})^{1/m}.
\end{equation}
\begin{defn}
We say that a radius function $r$ is a radius function of a Ricci-flat K\"{a}hler cone metric on $X$ if it satisfies an equation:
\begin{equation}\label{eq-Ricciflat}
(\sddb r^2)^{n+1}=dV_Y.
\end{equation}
More generally, if $g$ is a positive function on $P_{\hat{\xi}}$, we can consider the general $g$-soliton type equation:
\begin{equation}\label{eq-csoliton}
g(\eta_r) (\sddb r^2)^{n+1}=dV_Y. 
\end{equation} 
\end{defn}
The fact that the radius funtion of a Ricci-flat K\"{a}hler cone metric indeed satisfies \eqref{eq-Ricciflat} follows from the maximal principle, because the logrithmic of ratio of both sides is a bounded pluriharmonic since it is invariant under $r\partial_r$.  

Note that in this case we have the identity
$\mathfrak{L}_{r\partial_r} dV_Y=2(n+1) dV_Y$
which is equivalent to the identity (note that $dV_Y$ is $T$-invariant)
$\mathfrak{L}_{v_{\xi}}s=m(n+1)s.$ 
So it is natural to introduce the
following set of normalized Reeb vectors:
\begin{equation}\label{eq-normRcone}
\bar{N}_\bR^+=\{\xi\in \hat{N}_\bR^+; \mathfrak{L}_{v_{\xi}} s=m(n+1) s\}.
\end{equation} 
By \cite{Li18}, this is equivalent to $A_{Y}(\wt_{\hat{\xi}})=n+1$ where $\wt_{\hat{\xi}}$ is the valuation associated to $\hat{\xi}$. 

We can write the equation \eqref{eq-Ricciflat} as an equation on the Sasaki manifold $S=\{r=1\}$. 
By a direct calculation, we get:
\begin{eqnarray*}
\sddb r^2&=&-\frac{1}{2}dJd e^{\log r^2}=-\frac{1}{2}d\left( e^{\log r^2}Jd \log r^2\right)=d(e^{\log r^2} \eta)\\
&=&  r^2 d\eta+d r^2 \wedge \eta=r^2 d\eta+2 r dr\wedge \eta.
\end{eqnarray*}
So we get the volume form:
\begin{equation}\label{eq-conevolform}
\hat{\omega}^{n+1}=(\sddb r^2)^{n+1}=2 (n+1) r^{2n+1} (d\eta)^n\wedge \eta.
\end{equation}
On the other hand, if we set $dV_S=dV^{\hat{\xi}}_S=2^{-1}\iota_{r\partial_r}dV_Y$, then  
\begin{eqnarray*}
dV_Y=
2 r^{2n} dr \wedge dV^{\hat{\xi}}_S.
\end{eqnarray*}
Then the equation \eqref{eq-Ricciflat} and \eqref{eq-csoliton} are equivalent respectively to the equations on $S=\{r=1\}$:
\begin{eqnarray}
(d\eta)^n\wedge \eta&=&dV_S^{\hat{\xi}} \label{eq-RicflatS}\\
g(\eta) (d\eta)^n\wedge\eta&=&dV_S^{\hat{\xi}}.  \label{eq-csolitonS}
\end{eqnarray}
The more general equation \eqref{eq-csolitonS} can be viewed as a transversal $g$-soliton equation. 
When $\hat{\chi}$ is quasi-regular, we will calculate more explicitly in section \ref{sec-orbMA} to see that the above equation is equivalent to a $g$-soliton equation on the quotient orbifold.


\subsection{Deformation of Reeb vector fields}\label{sec-defReeb}
Fix a reference $\hat{\chi}\in \hat{N}_\bR^+$ and choose a radius function $r_0$ with respect to $\hat{\chi}$. Set 
$\eta_0=-J d\log r_0$, $S_0=\{r_0=1\}$. 

For any $\hat{\xi}\in \hat{N}_\bR^+$, there exists a unique radius function $r$ such that $J r\partial_r=\hat{\xi}$ and $S=\{r=1\}$. 
We will also denote this radius function by $r^{\hat{\xi}}_0$ and call it the deformation of $r_0$ with respect to $\hat{\xi}$.  Such deformation of radius function was studied in (\cite{HS16}, \cite{ACL20}) and gives an equivalent description of the type I deformation of Sasaki structure as discussed in \cite[section 3]{BGM06} (see also \cite{AC21}). 
This radius function $r=r^{\hat{\xi}}_0$ can be  defined implicitly in the following way. 
For any $q\in Y^*$, let $\gamma_q(t)$ be the integral curve of $J\hat{\xi}$ with initial condition $q$. In other words, $\frac{d}{dt}\gamma_q(t)=J\hat{\xi}$ and $\gamma_q(0)=q$. Then $r_0$ changes according to:
 \begin{equation}\label{eq-drdt}
 \frac{d}{dt} \log r_0=d \log r_0 \cdot J\hat{\xi}=-\eta_0(\hat{\xi})<0.
 \end{equation}
 So there exists a unique $t_*=t_*(q)$ such that $r_0(\gamma_q(t_*(q)))=1$. Because $\hat{J}\hat{\xi}=-r\partial_r$, so we get
 \begin{equation}\label{eq-logrqt}
 \log r(\gamma_q(t))-\log r(q)=-t. 
 \end{equation}
We then define $r(q)=r^{\hat{\xi}}_0(q)=e^{t_*(q)}$. 
Using the fact that $(q, t)\mapsto r_0(\gamma_q(t))$ is a smooth function and $r_0$ is strictly increasing with respect to $t$, by implicit function theorem we know that $r$ is a smooth function of $q\in Y^*$. 
 

Recall that $\hat{\xi}\in \hat{N}_\bR^+$ is of the form $\hat{\xi}=f \hat{\chi}+\xi^h$ with $f=\eta_0(\hat{\xi})>0$ and $\xi^h\in \mathfrak{D}$. 
By \eqref{eq-etaprop}, we have $\eta=\eta_r=f^{-1}\eta_0$, which implies that $(d\eta)^n\wedge \eta=f^{-n-1} (d\eta_0)^n\wedge \eta_0$. 
The Riemannian metric $G=\frac{1}{2}\sddb r^2(\cdot , J\cdot)$ associated to $\eta$ is then given by:
\begin{equation}
\left.G\right|_{\mathfrak{D}}=\frac{1}{2} f^{-1} d\eta_0(\cdot, J\cdot)=f^{-1} \left.G_0\right|_{\mathfrak{D}}, \quad G(\hat{\xi}, \mathfrak{D})=0, \quad G(\hat{\xi}, \hat{\xi})=1.
\end{equation}
In particular, $r^2$ is strictly plurisubharmonic on $Y^*$. 

Moreover
\begin{equation}
r \partial_r=-J(\hat{\xi})=-f J(\hat{\chi})+J(\xi^h)=f\cdot r_0\partial_{r_0}-J(\xi^h).
\end{equation}
We get $\left.\iota_{r\partial_r}dV_Y\right|_S=\left.f \iota_{r_0\partial_{r_0}}dV_Y\right|_S=f dV_S^{\hat{\chi}}$. 
So we see that $\eta=\eta_r$ satisfies \eqref{eq-RicflatS} if and only if $\eta_0=\eta_{r_0}$ satisfies the equation:
\begin{equation}\label{eq-RFeta0}
\eta_0(\hat{\xi})^{-n-2}(d\eta_0)^n\wedge \eta_0=dV_S^{\hat{\chi}}.
\end{equation}
More generally, $\eta$ satisfies \eqref{eq-csolitonS} if and only if $\eta_0$ satisfies the equation:
\begin{equation}\label{eq-csolitonSchi}
\eta_0(\hat{\xi})^{-n-2}g\left(\eta_0(\hat{\xi})^{-1}\eta_0\right) (d\eta_0)^n\wedge \eta_0=dV_S^{\hat{\chi}},
\end{equation}
Equivalently if we let $g_0$ be the positive function defined on $P_{\hat{\chi}}$ by the formula:
\begin{equation}
g_0(\hat{y})=\la \hat{y}, \hat{\xi}\ra^{-n-2} g\left(\frac{\hat{y}}{\la \hat{y}, \hat{\xi}\ra} \right),
\end{equation}
then \eqref{eq-csolitonSchi} is the same as:
\begin{equation}\label{eq-gtransMA}
g_0\left(\eta_0\right) (d\eta_0)^n\wedge \eta_0=dV^{\hat{\chi}}_S.
\end{equation}
In other words, for any two Reeb vector fields, $\hat{\xi}, \hat{\chi}$, we can always transform $g$-soliton equation for $\hat{\xi}$ into the $g_0$-soliton equation for $\hat{\chi}$. 
\begin{rem}
As emphasized earlier, such transformation of Ricci-flat K\"{a}hler cone equation to $g$-soliton equation was originally discovered in \cite{AC21, AJL21} by using transformation formula for curvature tensors of Tanaka-Webster connections. Here we are offering a different PDE point of view. 
\end{rem}

For any other $\hat{\xi}\in \hat{N}^+_\bR$, there is a projection $\ell_{\hat{\xi}}^{-1}: P_{\hat{\chi}}\rightarrow P_{\hat{\xi}}$ given by $\hat{y}\mapsto \frac{\hat{y}}{\la \hat{y}, \hat{\xi} \ra}$. The following proposition means that the Duistermaat-Heckman measure on $\mathcal{C}^*$ does not depend on $\hat{\xi}\in \mathcal{C}^*$. 
\begin{prop}
We have the identities:
\begin{equation}\label{eq-crosmeas}
\DHM^{\hat{\xi},S}_{T}=(\ell_{\hat{\xi}}^{-1})_*\left(\ell_{\hat{\xi}}^{-n-1}\DHM^{\hat{\chi}, S}_{T}\right).
\end{equation}
Moreover, $\DHM^{\hat{\xi}}_T=\DHM^{\hat{\chi}}_T$ over $\mathcal{C}^*$. 
\end{prop}

\begin{proof}
Over $S=\{r_0=1\}=\{r=1\}$, we have:
\begin{eqnarray*}
\eta=\eta_0(\hat{\xi})^{-1}\eta_0, \quad \left.(d\eta)^n\wedge \eta\right|_S=\eta_0(\hat{\xi})^{-n-1} \left.(d\eta_0)^n\wedge \eta_0\right|_S.
\end{eqnarray*}
For any $q\in S$, $\la \mu(q), \hat{\zeta}\ra= \eta(\hat{\zeta})= \eta_0(\hat{\xi})^{-1}\eta_0(\hat{\zeta})=\eta_0(\hat{\xi})^{-1} \la \mu_0(q), \hat{\zeta}\ra$. So 
$\mu|_S=\eta_0(\hat{\xi})^{-1} \mu_0|_S=\left(\frac{1}{\ell_{\hat{\xi}}(\hat{y})}\hat{y}\right)\circ \mu_0|_S$.
We can then verify \eqref{eq-crosmeas} by choosing any test continuous function $v(y)$ over $P_{\hat{\xi}}$, and calculating:
\begin{eqnarray}
\int_{P_{\hat{\xi}}} v(y) \mu_*((d\eta)^n\wedge \eta)&=&\int_{S} v(\mu(q)) (d\eta)^n\wedge \eta\nonumber \\
&=&\int_{S} v(\eta_0(\hat{\xi})^{-1}\mu_0(q))  \eta_0(\hat{\xi})^{-n-1} (d\eta_0)^n\wedge \eta_0 \nonumber \\
&=&\int_{P_{\hat{\chi}}}v\left( \ell_{\hat{\xi}}^{-1}\hat{y}\right) \ell_{\hat{\xi}}^{-n-1} (\mu_0)_*((d\eta_0)^n\wedge \eta_0). \label{eq-crossrel}
\end{eqnarray}
We verify the last statement by choosing any test continuous function $v(\hat{y})$ with compact support over $\mathcal{C}^*$, and calculating as follows:
\begin{eqnarray*}
\int_{\mathcal{C}^*} v(\hat{y})\DHM^{\hat{\xi}}_T&=&\int_{\mathcal{C}^*} v(\hat{y}) \ell_{\hat{\xi}}^n d \ell_{\hat{\xi}}\wedge \DHM^{\hat{\xi},S}_T=
\int_{P_{\hat{\xi}}} \left(\int_0^{+\infty} v(ty) t^n dt\right) \DHM^{\hat{\xi},S}_T\\
&=&\int_{P_{\hat{\chi}}}\left(\int_0^{+\infty} v(t \ell_{\hat{\xi}}^{-1}y)t^n dt\right) \ell_{\hat{\xi}}^{-n-1} \DHM^{\hat{\chi},S}_T\\
&=&\int_{P_{\hat{\chi}}}\left(\int_0^{+\infty} v(t \ell_{\hat{\xi}}^{-1}y) (\ell_{\hat{\xi}}^{-1}t)^n d (\ell_{\hat{\xi}}^{-1}t))\right)\DHM^{\hat{\chi},S}_T\\
&=&\int_{P_{\hat{\chi}}}\left(\int_0^{+\infty} v(tx) t^n dt \right)\DHM^{\hat{\chi},S}_T\\
&=&\int_{\mathcal{C}^*} v(\hat{x}) \DHM^{\hat{\chi}}_T. 
\end{eqnarray*}
\end{proof}

\subsection{Transveral K\"{a}hler deformation}

 Fix a reference radius function $r_0$ with respect to $\hat{\chi}$ as above. If $r$ is another radius function for $\hat{\chi}$. Then 
$r=r_0 e^{\vphi/2}$ for a $\hat{T}$-invariant function $\vphi$ that also satisfies $r\partial_r(\vphi)=0$.  Set

$$\mathcal{R}^{\hat{\chi}}(r_0)=\left\{\vphi; r_\vphi:=r_0 e^{\vphi/2} \text{ is a radius function for } \hat{\chi}\right\}.$$ 

\begin{lem}
$(\mu_{\vphi})_*(\sddb r_\vphi^2)^{n+1}$ does not depend on $\vphi\in \mathcal{R}^{\hat{\chi}}(r_0)$. 
\end{lem}
\begin{proof}
Choose any test function $f$ that decays exponentially with respect to $\ell_{\hat{\chi}}=\la \cdot, \hat{\chi}\ra$: $|f|\le C_1 e^{-C_2 \ell_{\hat{\chi}}}$. Set
$r=r_t=r_0 e^{t\vphi}$. Denote by $\mu=\mu_t$ be the moment map associated to $\sddb r^2$. Choose a basis $\hat{\xi}_1,\dots, \hat{\xi}_{r+1}$ of $N_\bR$ so that $\mu=(\mu(\hat{\xi}_1), \dots, \mu(\hat{\xi}_{r+1})=(\hat{\theta}_\kappa, \dots, \hat{\theta}_\kappa)$ satisfies
$\hat{\theta}_\kappa=-\frac{1}{2}J d (r^2)(\hat{\xi}_\kappa)=-\frac{1}{2}J\hat{\xi}_\kappa(r^2)$ (see \eqref{eq-cmoment}). Set $u=\frac{d}{dt}r^2=r^2 \vphi$. Then
 $\frac{d}{dt}\hat{\theta}_\kappa=-\frac{1}{2}J\hat{\xi}_{\kappa}(u)$ and 
 $$\frac{d}{dt}f(\mu)=\sum_\kappa f_{\kappa}(-\frac{1}{2}J\hat{\xi}_\kappa)(u)=-\frac{1}{2}J\zeta(u)$$ 
 where $\zeta=-f_{\kappa} \hat{\xi}_\kappa$ .  
On the other hand, $df=d (f(\mu))=\sum_\kappa f_{\kappa} d \hat{\theta}_\kappa= \sum_\kappa f_{\kappa} \iota_{\hat{\xi}_\kappa} (\sddb r^2)=\iota_{\zeta}(\sddb r^2)$. So we get:
\begin{eqnarray*}
&&\frac{d}{dt} \int_{\mathcal{C^*}}f(\hat{y}) \mu_*(\sddb r^2)^{n+1}=\frac{d}{dt}\int_{Z} f(\mu) (\sddb r^2)^{n+1}\\
&=&\int_Z -\frac{1}{2}J\zeta(u) (\sddb r^2)^{n+1}+f(\mu)\sddb u \wedge (n+1)(\sddb r^2)^{n}\\
&=&\int_Z -\frac{1}{2}J\zeta(u) (\sddb r^2)^{n+1}+\frac{1}{2}d f(\mu)\wedge Jd u\wedge (n+1)(\sddb r^2)^n\\
&=&\frac{1}{2}\int_Z -J\zeta(u) (\sddb r^2)^{n+1}+\iota_{\zeta} Jdu \wedge (\sddb r^2)^{n+1}=0.
\end{eqnarray*}
Because $f$ is arbitrary test function (with exponential decay), the statement follows easily. 

\end{proof}
For $\hat{\chi}$ (resp. $\hat{\xi}$) in $\hat{N}^+_\bR$, denote by $\mcR^{\hat{\chi}}$ (resp. $\mcR^{\hat{\xi}}$) the set of smooth radius function for $\hat{\chi}$ (resp. $\hat{\xi}$). We thus get a map:
\begin{eqnarray}\label{eq-raddef}
\mathfrak{d}: \mcR^{\hat{\chi}} &\rightarrow& \mcR^{\hat{\xi}}\\
r_0 e^{\vphi/2} &\mapsto & (r_0e^{\vphi/2})^{\hat{\xi}}= r^{\hat{\xi}}_0 e^{\vphi^{\hat{\xi}}/2}.
\end{eqnarray}
Set $r_\vphi=r_0 e^{\vphi/2}$ and $S_\vphi=\{r_\vphi=1\}$. 
Fix any point $q\in S_\vphi$. Assume that $\gamma_q(t)$ is the integral curve of $J\hat{\xi}$ with initial condition $q$ as above. Then again there exists a unique $t_*=t_*(q)\in \bR$ such that $p=q(t_*(q))\in S_0=S$. In other words, $r_0(p)=1=r_0^{\hat{\xi}}(p)$. Then we have the identity $r_\vphi^{\hat{\xi}}(q)=r^{\hat{\xi}}_0(p) e^{\frac{\vphi^{\hat{\xi}}(q)}{2}}=e^{\frac{\vphi^{\hat{\xi}}(q)}{2}}=e^{t_*(q)}$. Consider the function $f(q, t)=r_0(\gamma_q(t))$. Then by \eqref{eq-logrqt}, $f(q, -\frac{\vphi^{\hat{\xi}}(q)}{2})=1$ or equivalently $\log f(q, -\frac{\vphi^{\hat{\xi}}(q)}{2})=0$.
Now let $\vphi(t)$ be a curve of potentials satisfying $\vphi(0)=0$. Set $q(t)=\exp_{p}(-\frac{\vphi}{2} r_0\partial_{r_0})$ so that $r_\vphi(q(t))=1$. Then 
\begin{equation}\label{eq-logf0}
\log f(q(t), -\frac{\vphi(t)^{\hat{\xi}}(q(t))}{2})=0.
\end{equation} 
Note that $q(0)=\exp_p(0)=p$, $\dot{q}(0)=-\frac{\dot{\vphi}(0)}{2} r_0\partial_{r_0}$, $\vphi(0)^{\hat{\xi}}=0^{\hat{\xi}}\equiv 0$. 
Taking derivative of \eqref{eq-logf0} with respect to $t$ at $t=0$, we get:
\begin{eqnarray*}
0=d \log r_0(-\frac{\dot{\vphi}}{2} r_0\partial_{r_0})-d \log r_0(\frac{\dot{\vphi}^{\hat{\xi}}}{2} J\hat{\xi})=-\frac{\dot{\vphi}}{2}+\eta_0(\hat{\xi})\frac{\dot{\vphi}^{\hat{\xi}}}{2}
\end{eqnarray*}
which gives us the useful identity discovered in \cite[Lemma 2.17]{ACL20}:
\begin{equation}
\dot{\vphi}^{\hat{\xi}}=\frac{\dot{\vphi}}{\eta_0(\hat{\xi})}. 
\end{equation}

This is useful by comparing the functionals in the variational approach. The comparison for (transversal) Mabuchi functional was proved in \cite[Lemma 4.4]{ACL20}.
Here we show a corresponding result for the cone version of Ding functional. 
 More precisely, we know that the solution to \eqref{eq-csolitonS} is the critical point of the following $\bfD^{\hat{\xi}}_g$-functional over 
$\mathcal{R}^{\hat{\xi}}(r)$ (see \cite{CS19, LWX21}):
\begin{equation}
\bfD^{\hat{\xi}}_g(\vphi)=-\bfE^{\hat{\xi}}_g(\vphi)+\bfL^{\hat{\xi}}(\vphi).
\end{equation}
where
\begin{eqnarray*}
\bfE^{\hat{\xi}}_g(\vphi)&=&\frac{1}{\hat{\bV}^{\hat{\xi}}_g(n+1)!}\int_0^1 dt \int_Y \vphi\cdot g(\eta_{t\vphi}) e^{-r_{t\vphi}^2} (\sddb r_{t\vphi}^2)^{n+1},\\
\hat{\bV}^{\hat{\xi}}_g&=&\frac{1}{(n+1)!}\int_Y g(\eta) e^{-r^2} (\sddb r^2)^{n+1}=\int_S g(\eta) (d\eta)^n\wedge \eta,\\
\bfL^{\hat{\xi}}(\vphi)&=&-\frac{1}{n+1}\log\left( \int_Y e^{-r_\vphi^2} dV_Y\right). 
\end{eqnarray*}
\begin{prop}\label{prop-cDing}
Set $\tilde{\xi}=\hat{\xi}-\hat{\chi}$. 
Assume that the following vanishing holds true:
\begin{equation}
\Fut^{\hat{\xi}}_g(\tilde{\xi}):=\int_Y \eta(\tilde{\xi}) g(\eta) e^{-r^2}(\sddb r^2)^{n+1}=0.
\end{equation} 
Then for any $\vphi\in \mathcal{R}^{\hat{\chi}}(r_0)$, we have:
\begin{equation}\label{eq-defEL}
\bfE^{\hat{\xi}}_g(\vphi^{\hat{\xi}})=\bfE^{\hat{\chi}}_{g_0}(\vphi), \quad
\bfL^{\hat{\xi}}(\vphi^{\hat{\xi}})=\bfL^{\hat{\chi}}(\vphi).
\end{equation}
\end{prop}

\begin{proof}
For $\bfE_g$, similar to \cite{ACL20}, we use its infinitesimal variational formula to calculate:
\begin{eqnarray*}
\bV^{\hat{\xi}}_{g}\cdot \left.\frac{d}{dt}\bfE_g^{\hat{\xi}}(\vphi^{\hat{\xi}})\right|_{t=0}&=&\int_Y \dot{\vphi}^{\hat{\xi}} g(\eta) e^{-r^2}(\sddb r^2)^{n+1}=(n+1)!\int_{S} \dot{\vphi}^{\hat{\xi}}g(\eta) (d\eta)^n\wedge \eta\\
&=&(n+1)!\int_S \frac{\dot{\vphi}}{\eta_0(\hat{\xi})}g\left(\frac{\eta_0}{\eta_0(\hat{\xi)}}\right) \eta_0(\hat{\xi})^{-n-1} (d\eta_0)^n\wedge \eta_0\\
&=&\bV^{\hat{\chi}}_{g_0}\cdot \left.\frac{d}{dt}\bfE_{g_0}^{\hat{\chi}}(\vphi)\right|_{t=0}.
\end{eqnarray*}
So we just need to verify that $\bV^{\hat{\xi}}_g=\bV^{\hat{\chi}}_{g_0}$ 
to get the identity for $\bfE_g$. This is where the vanishing condition is used. Indeed, because $\eta(\tilde{\xi})=\eta(\hat{\xi}-\hat{\chi})=1-\frac{1}{\eta_0(\hat{\xi})}$, the vanishing condition implies:
\begin{eqnarray*}
\bV^{\hat{\xi}}_g&=&\bV^{\hat{\xi}}_g-\int_S \eta(\tilde{\xi}) g(\eta) (d\eta)^n\wedge \eta=\int_S \frac{1}{\eta_0(\hat{\xi})}g\left(\frac{\eta_0}{\la , \hat{\xi}\ra}\right) \eta_0(\hat{\xi})^{-n-1} (d\eta_0)^n\wedge \eta_0\\
&=&\int_S g_0(\eta_0) (d\eta_0)^n\wedge \eta_0=\bV^{\hat{\chi}}_{g_0}.
\end{eqnarray*}

 For $\bfL$, we have (with $dV^{\hat{\xi}}_S=2^{-1}\iota_{\partial_r}dV_Y$):
\begin{eqnarray*}
\int_Y e^{-r^2} dV_Y=n! \int_S dV^{\hat{\xi}}_S=n! \int_S \eta_0(\hat{\xi}) dV^{\hat{\chi}}_S.
\end{eqnarray*}
We claim that the last integral is equal to $\int_S dV^{\hat{\chi}}_S$.
To see this, recall that $\hat{\xi}$ is normalized by the condition $\mathfrak{L}_{-J\hat{\xi}}dV_Y=2(n+1)dV_Y$. 
On the other hand, we know that $dV_Y=r_0^{2n+1}dr_0 \wedge dV^{\hat{\chi}}_S=r^{2(n+1)}d\log r_0\wedge dV_S$. So by using product formula for Lie derivatives and Cartan's formula, we get: 
\begin{eqnarray*}
2(n+1)dV_Y&=&2(n+1)r_0^{2n+1} (-J\hat{\xi}(r_0)) d\log r_0\wedge dV_S+r_0^{2(n+1)} d (-J\hat{\xi}(\log r_0))\wedge dV_S\\
&&\hskip 3cm +r_0^{2(n+1)}d\log r_0\wedge \mathfrak{L}_{-J\hat{\xi}}dV_S. 
\end{eqnarray*}
Contracting both sides with $r_0\partial_{r_0}$ and by using the $r_0$-independence of $\eta_0(\hat{\xi})=-J\hat{\xi}(\log r_0)$ we get the identity on $S=\{r_0=1\}$:
\begin{equation}
2(n+1)dV_S=2(n+1)\eta_0(\hat{\xi}) dV_S+\mathfrak{L}_{-J\xi^h}dV_S 
\end{equation}
where $\xi^h=\hat{\xi}-\eta_0(\hat{\xi})\hat{\chi}\in \mathfrak{D}$ so that $J\xi^h\in \mathfrak{D}$ is tangent to $S$.
Integrating both sides on $S$ proves the claim. 
\end{proof}

\subsection{Reduction to $g$-soliton equation on orbifolds}\label{sec-orbMA}

We assume that $\hat{\chi}\in N^+_\bQ$ is quasi-regular and let $r$ be a radius function for $\hat{\chi}$. \footnote{Since in this subsection, we do not deform the Reeb vector field. All data (radius functions, contact forms) in this subsection are for the fixed Reeb vector field $\hat{\chi}$. So the notation in this section is different with the previous section. } Let $(X, D)=Y/\la v_{\hat{\chi}}\ra$ and $L$ be as the induced orbifold line bundle with the canonical projection $\pi: L\rightarrow X$. Let $L^*$ denote the dual orbifold line bundle of $L$. Then there is analytic contraction morphism $L^*\rightarrow Y$ from the total space of $L^*$ such that $L^*\setminus X\cong Y\setminus \{o\}$. 
 The function $r$ induces a $T$-invariant orbifold Hermitian metric on $L^*$. Conversely any $T$-invariant orbifold smooth Hermitian metric with positive Chern curvature on $L^*$ induces smooth radius function on $Y^*$. 

In the following calculation, we assume that $L$ is an ordinary line bundle. For orbifold line bundles, we just need to calculate the same calculation on uniformization charts. Note that there is a well-understood condition for the smoothness of the total space of orbifold line bundle in terms of local data of $(X, D)$ (see \cite{Kol04} and \cite[Theorem 4.7.5]{BG08}). Fix a smooth orbifold smooth Hermitian metric $h_0$ on $L$ which is induced by a smooth radius function $r_0$ for $\hat{\chi}$. 
Then the orbifold Hermitian metric induced by $r$ on $L^*$ can be written as $h_0e^{\vphi}$ which satisfies $\omega=\omega_\vphi>0$. 
Assume that $s$ is the local trivializing holomorphic section. 
Set $a=|s|^2_{h_0} e^{\vphi}$ which satisfies $\sddb\log a=\omega=\omega_0+\sddb\vphi$. Let $w$ be a linear coordinate along the fibre of $L^*$. We consider $h:=a|w|^2$ as a function on $L^*$. Then there exists $\beta>0$ such that $r^2=h^\beta$ and $r\partial_r=2\beta^{-1} Re(w\partial_w)$.  
It is straightforward to calculate that:
\begin{eqnarray*}
\hat{\omega}:=\sddb r^2=\beta^2 h^\beta \sqrt{-1} \partial \log h\wedge \bar{\partial}\log h+\beta h^\beta \omega. 
\end{eqnarray*}

So we get:
\begin{eqnarray*}
\hat{\omega}^{n+1}=(\sddb r^2)^{n+1}
&=&(n+1)\beta^{n+2}r^{2(n+1)} \sqrt{-1}\partial \log h\wedge \bar{\partial}\log h\wedge \omega^n.
\end{eqnarray*}
Note that $\partial \log h=a^{-1}\partial a+w^{-1}dw$ and the wedge product on the right-hand-side can be calculated:
\begin{eqnarray*}
\sqrt{-1} \partial \log h\wedge \bar{\partial} \log h\wedge \omega^n&=&\frac{\sqrt{-1}dw\wedge d\bar{w}}{|w|^2}\wedge \omega^n=\frac{2 d|w|\wedge d \arg(w)}{|w|}\wedge \omega^n\\
&=&\beta^{-1}\frac{2 d (a^{1/2}|w|)^\beta \wedge d (\arg(w))}{(a^{1/2}|w|)^\beta}\wedge \omega^n=\beta^{-1} \frac{2 dr\wedge d\psi}{r}\wedge \omega^n
\end{eqnarray*}
where we set $\psi=\arg(w)$. 
So we get the identity (see \eqref{eq-conevolform}):
\begin{equation*}
(\sddb r^2)^{n+1}=2(n+1)\beta^{n+1}r^{2n+1}dr\wedge d\psi\wedge \omega^n=2(n+1)r^{2n+1}(d\eta)^n\wedge \eta
\end{equation*}
where the contact form $\eta$ is given by:
\begin{eqnarray}\label{eq-etaorb}
\eta&=&-Jd \log r=-\frac{\beta}{2}Jd \log h=\beta\frac{\sqrt{-1}}{2}(\bar{\partial}-\partial)\log h \nonumber \\
&=&\beta\frac{\sqrt{-1}}{2}(a^{-1}(\bar{\partial}a-\partial a)+\bar{w}^{-1}d\bar{w}-w^{-1}dw)=\beta\left(d\psi-\frac{1}{2}Jd \log a\right).
\end{eqnarray}
So we get $d\eta=\beta \sddb\log h=\beta \omega$ and the identity:
\begin{eqnarray}\label{eq-volS2X}
(d\eta)^n\wedge \eta&=&\beta^n\omega^n\wedge \eta=\beta^{n+1} \omega^n\wedge d\psi.
\end{eqnarray}
In particular, 
\begin{equation}\label{eq-volSomega}
\int_S (d\eta)^n\wedge \eta=\beta^{n+1}2\pi \int_X \omega^n=(n+1)^{-n-1}\cdot 2\pi \gamma \int_X (\gamma \omega)^n.
\end{equation}

From now on, we assume furthermore that $Y$ is $\bQ$-Gorenstein and there exists a $\bT$-equivariant nowhere vanishing section $s\in |mK_Y|$. In this case we say that $Y$ is a Fano cone. Let $(X, D)=Y/\la v_{\hat{\chi}}\ra$ and $L$ be the same as before. 
By \cite[40-42]{Kol04}, we know that in this case, there is an identity $-(K_X+D)=\gamma L$ for some $\gamma>0$. 
Choose $m\gg 1$ sufficiently divisible such that $-m(K_X+D)$ is Cartier. 
Assume that $w$ is a linear variable along the fibre of $L^*$. 
Then formally we can write down a nowhere vanishing section $s\in |mK_Z|$ as $s=dz^{\otimes m} \wedge (dw)^{m\gamma}$ so that
\begin{eqnarray*}
dV_Y=(\sqrt{-1}^{m(n+1)^2}s\wedge \bar{s})^{1/m}=(\sqrt{-1})^{n^2+1} dz \wedge d\bar{z} \wedge dw^\gamma \wedge d\bar{w}^\gamma.
\end{eqnarray*} 
Moreover $\mathfrak{L}_{w\partial_w}s= (m\gamma) s$. Assume now that $r$ is a radius function satisfying $\mathfrak{L}_{r\partial_r} dV_Y=(n+1)dV_Y$. Then since $r\partial_r=\beta^{-1} (w\partial_w+\bar{w}\partial_{\bar{w}})$, 
\begin{equation}
(n+1)dV_Y=\mathfrak{L}_{r\partial_r} dV_Y=\mathfrak{L}_{\beta^{-1} (w\partial_w+\bar{w}\partial_{\bar{w}})}dV_Y=\beta^{-1}\gamma dV_Y
\end{equation} 
which implies $\beta=\frac{\gamma}{n+1}$. Moreover we have:
\begin{eqnarray*}
\iota_{r\partial_r}dV_Y&=&\beta^{-1}\iota_{w\partial_w+\bar{w}\partial_{\bar{w}}}(\sqrt{-1})^{n^2+1}dz\wedge d\bar{z}\wedge dw^\gamma \wedge d\bar{w}^\gamma\\
&=& \beta^{-1} \gamma^2 (\sqrt{-1})^{n^2+1} dz\wedge d\bar{z} |w|^{2(\gamma-1)} (wd\bar{w}-\bar{w} dw)\\
&=&\beta^{-1}\gamma^2 (\sqrt{-1})^{n^2} dz\wedge d\bar{z} 
\wedge 2 (a|w|^2)^{\gamma}a^{-\gamma} d\psi \\
&=&2 \beta^{-1}\gamma^2 r^{2(n+1)}\wedge d\psi\wedge \Omega_{\gamma \vphi}.
\end{eqnarray*}
Here $\Omega_{\gamma\vphi}=\Omega_0 e^{-\gamma\vphi}$ where $\Omega_0$ is the volume form induced by $h_0$. 
$$
$$
If $v$ is any holomorphic vector field on $X$, then the unique horizontal lifting of $v$ to $L$ is given by:
\begin{eqnarray*}
v^h&:=&\sum_i v^i \partial_i -v(\log a) w\partial_w,
\end{eqnarray*}
which satisfies $\eta(v^h)=0$. 
Set 
\begin{equation}
\theta_v=-\frac{\mathfrak{L}_v e^{-\gamma\vphi}}{e^{-\gamma\vphi}}.
\end{equation}
which satisfies $\iota_v (\gamma \omega)=\sqrt{-1}\bar{\partial} \theta_v$. The canonical holomorphic lifting of $v$ is then given by:
\begin{equation}
\tilde{v}=v^h+\theta_v \gamma^{-1}w\partial_w=\sum_i v^i\partial_i-v(\log a)w\partial_w+\theta_v \gamma^{-1} w\partial_w
\end{equation}
By \eqref{eq-etaorb},  $\eta(w\partial_w)=-\beta\frac{\sqrt{-1}}{2}$. Since $\beta\gamma^{-1}=\frac{1}{n+1}$,, 
we get $\eta(\tilde{v})=-\frac{\sqrt{-1}}{2(n+1)}\theta_v$ and is equivalent to the identity $\eta(\tilde{\xi})=\frac{\theta_v}{n+1}$ where $v=v_{\xi}$. 
If $\hat{\xi}=\hat{\chi}+\tilde{\xi}\in \hat{N}_\bR$, then 
\begin{equation}\label{eq-normetaxi}
\eta(\hat{\xi})=1+\frac{\theta_{v_\xi}}{n+1}.
\end{equation} 
In particular $\hat{\xi}\in \hat{N}^+_\bR$ if and only if $1+\frac{\theta_{v_\xi}}{n+1}>0$. 
So we get 
\begin{prop}
With the above notation, the equation \eqref{eq-gtransMA} is equivalent to the following Monge-Amp\`{e}re equation on $X$:
\begin{equation}
(n+1+\theta_{v_\xi}(\vphi))^{-n-2}g\left((1+\frac{\theta_{v_\xi}}{n+1})^{-1} \eta \right) (\omega_0+\sddb\vphi)^n=\Omega_{\gamma\vphi}
\end{equation}
where $\eta=(\eta(\hat{\chi}), \eta(\tilde{\xi}_{\kappa})_{\kappa=1,\dots, r})=(1, \frac{\theta_{\kappa}}{n+1})$ (with respect to a basis $\{\hat{\chi}, \tilde{\xi}_\kappa; \kappa=1,\dots, r\}$). \\
In particular, corresponding to Ricci-flat K\"{a}hler cone metrics (with $g=1$), the equation \eqref{eq-RFeta0} is equivalent to the following equation:
\begin{equation}\label{eq-RForb}
(n+1+\theta_{v_\xi}(\vphi))^{-n-2}(\omega_0+\sddb \vphi)^n=\Omega_0 e^{-\gamma\vphi}.
\end{equation}
\end{prop}
We summarize our process for transforming Ricci-flat K\"{a}hler cone equation into the last equation into the following diagram:
\begin{equation}\label{eq-transdiag}
\xymatrix@R=0.8pc @C=1pc{
\eqref{eq-Ricciflat} \text{ on $Y$}  \ar@{<-->}[dd] \ar@{<->}[rr] &   & \eqref{eq-RicflatS} \text{ on $S^{\hat{\xi}}$} \ar@{<->}[dd]   \\
 & & \\
\eqref{eq-RForb} \ar@{<->}[rr] \text{ on $(X, D)$} &  & \eqref{eq-RFeta0} \text{ on $S^{\hat{\chi}}$} 
}
\end{equation}




\subsection{Stability of Fano cones vs. $g$-weighted stability of Fano orbifolds}
We explain in this subsection that the stability of Fano cones is equivalent to the $g$-weighted stability of any quasi-regular quotients. In terms of K-stability of Fano cones, which was defined by Collins-Sz\'ekelyhidi, this connection was already pointed out in \cite{ACL20, AJL21} when the test configurations has smooth ambient spaces.
Here we will use the Ding-stability study in \cite{LWX21} to deal with the general case. We refer to \cite{LWX21} for terminology used in the following discussion.  

Assume that $(Y, \hat{\xi})$ is a polarized Fano cone with an effective action by $\hat{\bT}$. A test configuration $(\mcY, \hat{\xi}; \zeta)$ of $Y$ consists of a $\bT$-equivariant flat family of affine varieties $\pi: \mcY\rightarrow \bC$ with an extra $\bC^*$-action generated by $\zeta$ which commutes with the $\bT$-action such that $\pi$ becomes $\bC^*\cong \la v_{\zeta}\ra$-equivariant. For simplicity of notation, we just denote the test configuration by $\mcY$. We always assume that $\mcY$ is normal and $\bQ$-Gorenstein. 
$\mcY$ is called special if $(\mcY, \mcY_0)$ has plt singularities. 

Let $\hat{\chi}$ be a quasi-regular Reeb vector and let $(X, D)$ be the quotient $Y/\la v_{\hat{\chi}}\ra$. Then $\mcY/\la v_{\hat{\chi}}\ra=(\mcX, \mcD, \gamma \mcL=-(K_{\mcX}+\mcD))$ is a (an anti-canonical) test configuration of $(X, D, -(K_X+D))$ where $\mcL$ is an induced orbifold bundle considered as a $\bQ$-Weil divisor, and $\zeta$ projects to be a holomorphic vector field on $\mcX$ generating a $\bC^*$-action. Conversely any test configuration $(\mcX, \mcD, \mcL)$ induces a test configuration of $Y$, which is obtained directly by using the extended Rees algebra of the filtration $\cF=\cF_{(\mcX, \mcL)}$ associated to the test configuration (see \eqref{eq-filTC}):
\begin{equation}
\mcY={\rm Spec}_{\bC[t]}\left( \bigoplus_{m\in \bZ}\bigoplus_{\lambda\in \bZ}t^{-\lambda}\cF^{\lambda}R_m\right).
\end{equation} 
From this description, we see that $\mcY$ can be obtained by taking fiberwise cones over $X$ with respect to the $\bQ$-Weil divisor $\mcL$ (see \cite{Kol04}). In this correspondence, special test configurations of $Y$ correspond to special test configurations of $(X, D)$. 

Let $\hat{\bT}'$ be the torus generated by $\hat{\bT}$ and $\zeta$. Then in general $\mcY_0$ admits $\hat{\bT}'$-action which corresponds to a decomposition $R'=\bigoplus_{\alpha\in \bZ^{r+1}}R'_{\alpha}$ where $R'$ is now the coordinate ring of $\mcY_0$. Its volume is defined as:
\begin{equation}
\vol(\hat{\xi})=\lim_{\lambda\rightarrow+\infty} \frac{\dim_\bC \bigoplus_{\alpha, \la \alpha, \hat{\xi}\ra < \lambda} R'_\alpha }{\lambda^{n+1}/(n+1)!}.
\end{equation}
Assume that $\hat{\xi}$ and $\zeta$ are normalized by the condition 
\begin{equation}\label{eq-normzeta}
n+1=A_{\mcY_0}(\hat{\xi}):=\frac{1}{m}\frac{\mathfrak{L}_{v_{\hat{\xi}}}s'}{s'}, \quad 0=\frac{\mathfrak{L}_\zeta s'}{s'}
\end{equation}
for a $\hat{\bT}'$-equivariant nowhere vanishing section $s'\in |mK_{\mcY_0}|$ (see \cite[2.2]{LWX21} for this notation). 

\begin{thm}
With the above notation, for any $\hat{\xi}=\hat{\chi}+\tilde{\xi} \in \hat{N}_\bR^+$, we have the following volume formula:
\begin{eqnarray}\label{eq-volint}
\vol(\hat{\xi})=\frac{\gamma}{(2\pi)^n}\int_{\mcX_0} (n+1+\theta_\xi)^{-n-1}(\gamma\omega)^n.
\end{eqnarray}
\end{thm}
\begin{proof}
Assume that $-(K_{\mcX_0}+\mcD_0)=\gamma \mcL_0$ and $\beta=\frac{\gamma}{n+1}$. Then $\hat{\chi}=\beta^{-1}w\partial_w$ and $\wt_{\hat{\chi}}=\beta^{-1} \wt_{w\partial_w}$. So if $\hat{\xi}=\hat{\chi}+\tilde{\xi}$, then the filtration associated to the weight valuation is given by:
\begin{equation}
\cF^\lambda_{\wt_{\hat{\xi}}}R_m=\cF^\lambda R_m=\bigoplus_\alpha \left\{R_{m,\alpha}; \la \alpha, \xi\ra+m \beta^{-1} \ge \lambda\right\}.
\end{equation}
This is the shift of the filtration induced by the (real) holomorphic vector field $\xi$ over $\mcX_0$:
\begin{equation}
\tilde{\cF}^\lambda R_m=\bigoplus \{R_{m,\alpha}; \la \alpha, \xi\ra \ge \lambda\}.
\end{equation}
By \cite[4.1]{Li17} (see also \cite{LX18}), we have the volume formula:
\begin{eqnarray*}
\vol(\hat{\xi})&=& \int_\bR \frac{-d \vol(\cF^{(\lambda)})}{t^{n+1}}=\int_\bR \frac{-d \vol(\tilde{\cF}^{(t-\beta^{-1})})}{t^{n+1}}\\
&=&\int_{\bR} \frac{-d\vol(\tilde{\cF}^{(t)})}{(\beta^{-1}+t)^{n+1}}=\frac{1}{(2\pi)^n} \int_{\mcX_0} ((n+1){\gamma}^{-1}+\gamma^{-1}\theta_\xi)^{-n-1} \omega^n\\
&=&\frac{1}{(2\pi)^n} \gamma \int_{\mcX_0} (n+1+\theta_\xi)^{-n-1} (\gamma\omega)^n.
\end{eqnarray*}
Here we used the fact that $(2\pi)^{n}\left(-d\vol(\tilde{\cF}^{(t)})\right)=(\gamma^{-1}\theta_{\xi})_*\omega^n$ (see Theorem \ref{thm-numconv}). 
Note that although the formula in \cite[4.1]{Li19} deals with the case of irreducible normal varieties. One can apply this formula for normalization of each irreducible components as explained in \cite[Lemma 2.11]{LWX21}. 
\end{proof}
We can also derive this formula by using the contact forms associated to smooth radius functions as follows.
Let $r=r^{\hat{\xi}}$ be a smooth radius function for $\hat{\xi}$. Then we have the identity (see \eqref{eq-volS}):
\begin{eqnarray}
\vol(\hat{\xi})&=&\frac{1}{(n+1)!(2\pi)^{n+1}}\int_{\mcY_0}e^{-r^2} (\sddb r^2)^{n+1}\nonumber \\
&=&\frac{1}{(2\pi)^{n+1}}\int_{\{r=1\}}(d\eta)^n\wedge \eta=\frac{1}{(2\pi)^{n+1}}\vol(S^{\hat{\xi}}). \label{eq-volxiS}
\end{eqnarray}
Indeed, it is easy to verify this identity when $\hat{\xi}$ is quasi-regular (see \eqref{eq-volSomega}). It is true for any $\hat{\xi}\in \hat{N}^+_\bR$ because both sides depend continuously on $\hat{\xi}$.  Now we can derive \eqref{eq-volint} as follows:
\begin{eqnarray*}
\vol(S)&=&\int_{\{r=1\}} (d\eta)^n\wedge \eta=\int_{\{r_0=1\}} \eta_0(\hat{\xi})^{-n-1} (d\eta_0)^n\wedge \eta_0\\
&=&2\pi \int_{\mcX_0} (1+\frac{1}{n+1} \theta_v)^{-n-1} \beta^{n+1} \omega^n=2\pi \gamma \int_{\mcX_0} (n+1+\theta_v)^{-n-1} (\gamma\omega)^n.
\end{eqnarray*}

Taking derivative in \eqref{eq-volint} with respect to $\hat{\xi}$ and using the normalization in \eqref{eq-normzeta} we get the identity (see \cite[Definition A.1]{LWX21} and also \cite{Wu21}):
\begin{eqnarray*}
-\bfE^\NA(\mcY)&=&\frac{A_{\mcY_0}(\hat{\xi})}{n+1}\cdot \frac{D_{\zeta}\vol(\hat{\xi})}{\vol(\hat{\xi})}=\frac{A_{\mcY_0}(\hat{\xi})}{\vol(\hat{\xi})}\frac{\gamma}{(2\pi)^n} \int_{\mcX_0} \frac{-\theta_\zeta}{(n+1+\theta_v)^{n+2}}(\gamma\omega)^n\\
&=&-C(g, \hat{\xi})\cdot \bfE^\NA_g(\mcX, \mcD, \gamma \mcL)
\end{eqnarray*}
for $g=(n+1+\la x, \xi\ra)^{-n-2}$ (see \eqref{eq-ENAg}) and $C(g, \hat{\xi})=\frac{\gamma\cdot A_{\mcY_0}(\hat{\xi})\bV_g}{(2\pi)^n\vol(\hat{\xi})}$. 
Now we add the assumption that $D_{\tilde{\xi}}\vol(\hat{\xi})=0$ which is necessary for $(Y, \xi)$ to be semistable by \cite{MSY08, CS18, LX18}. Then we get:
\begin{equation}
\int_{\mcX_0}\frac{\theta_{\xi}}{(n+1+\theta_\xi)^{n+2}} (\gamma \omega)^n=0
\end{equation}
which implies 
\begin{equation*}
(n+1) \bV_g=\int_{\mcX_0}\frac{n+1}{(n+1+\theta_\xi)^{n+2}}(\gamma \omega)^n=\int_{\mcX_0} \frac{1}{(n+1+\theta_\xi)^{n+1}}(\gamma \omega)^n=\frac{(2\pi)^n}{\gamma}\vol(\hat{\xi}).
\end{equation*}
So we get $C(g, \hat{\xi})=\frac{(2\pi)^{n+1}A_{\mcY_0}(\hat{\xi})}{n+1}=1$ and hence: 
\begin{equation}\label{eq-ENAdef}
\bfE^\NA(\mcY)=\bfE^\NA_g(\mcX, \mcD, \mcL). 
\end{equation}
On the other hand, by construction we easily have:
\begin{equation}\label{eq-LNAdef}
\bfL^\NA(\mcY):=\lct(\mcY; \mcY_0)-1=\lct(\mcX, D, -(K_{\mcX}+\mcD+\gamma \mcL))-1=\bfL^\NA(\mcX, \mcD, \mcL).
\end{equation} 
This can be verified analytically by using \cite{Berm15} and \cite[Proposition A.14]{LWX21}. In other words, the negative of both sides are equal to the Lelong number of the following function define on $\bC$:
\begin{equation*}
f(t)=-\log \left(\int_{\mcY_t}e^{-r_t^2} dV_{\mcY_t}\right)=-\log \int_{\mcX_t} \Omega_{\gamma\vphi_t}+{\rm constant}
\end{equation*}  
where $\{\vphi_t; t\in \bC\}$ is a psh Hermitian metrics on $L_\bC$ that extends to be a locally bounded psh Hermitian metric on $\mcL$, 
and $r_t^2=(h_0 e^{\vphi_t})^{\gamma/(n+1)}$. A more direct way to see \eqref{eq-LNAdef} is as follows. Let $\mathfrak{s}$ be a local generator of $K_{\mcY}$ (if $mK_{\mcY}$ is Cartier with a local generator $\mathfrak{s}'$, then we can formally choose $\mathfrak{s}=\mathfrak{s}'^{1/m}$).  Let $\mu_\mcY: \tilde{\mcY}\rightarrow \mcY$ be the weighted blowup associated to the valuation $\wt_{\hat{\chi}}$. Then $\mu_\mcY^* \mathfrak{s}=s\wedge dw^\gamma$ where $s$ is a local generator of $K_{\mcX}+\mcD$ and $w$ is a linear coordinate along the fibre of $\mcL$. Now both sides of \eqref{eq-LNAdef} are then given by:
\begin{eqnarray*}
&&\sup\left\{\alpha; |t|^{-2(\alpha+1)} (\mathfrak{s}\wedge \bar{\mathfrak{s}})^{1/m} \text{ is locally integrable }\right\}\\
&&\hskip 1cm =\sup\left\{\alpha; |t|^{-2(\alpha+1)} (s\wedge \bar{s})^{1/m}\text{ is locally integrable for all local generator } s\right\}. 
\end{eqnarray*}

Note that \eqref{eq-ENAdef}-\eqref{eq-LNAdef} are nothing but non-Archimedean version of \eqref{eq-defEL}. Combining them gives us the identity (see \cite[Definition A.1]{LWX21}):
\begin{eqnarray*}
\bfD^\NA(\mcY)&=&-\bfE^\NA(\mcY)+\bfL^\NA(\mcY)\\
&=&-\bfE^\NA_g(\mcX, \mcD, \mcL)+\bfL^\NA(\mcX, \mcD, \mcL)=\bfD^\NA_g(\mcX, \mcD, \mcL). 
\end{eqnarray*}
In this way, we get that the (poly)stability of the affine cone $Y$ (for special test configurations in the sense of Collins-Sz\'{e}kelyhidi, or for more general $\bQ$-Gorenstein test configurations as defined in \cite{LWX21}) is equivalent to $g$-Ding-(poly)stability of the quasi-regular quotient $(X, D)$. So the Yau-Tian-Donaldson for the $g$-soliton equation on $(X, D)$ implies the Yau-Tian-Donaldson conjecture for the Ricci-flat K\"{a}hler cone metric on $Y$ by using the following relations:
\begin{equation*}\label{eq-3morph}
\xymatrix@R=0.8pc @C=0.5pc{
\text{stability for $Y$ }  \ar@{<-->}[dd] \ar@{<->}[rr] &   & \text{$g$-weighted stability for $(X, D)$} \ar@{<->}[dd]   \\
 & & \\
\text{Ricci-flat K\"{a}hler cone metric on $Y$} \ar@{<->}[rr] &  & \text{$g$-soliton on $(X, D)$} 
}
\end{equation*}
This is clear when $Y$ has an isolated singularity in which case $(X, D)$ is an orbifold. Note that we have the fact that any weak $g$-soliton metric on $(X, D)$ is orbifold smooth which then gives rise to smooth Ricci-flat K\"{a}hler cone metric by the previous discussion.
This application of results from \cite{HL20} was already pointed out in \cite{AJL21}. 

For a general Fano cone $Y$ with not necessarily isolated singularity, we can still use the above process to deduce the existence of a weak $g$-soliton potential on $(X, D)$ from the (poly)stability of $Y$, which then induces a weak Ricci-flat K\"{a}hler cone potential on $Y$. By a weak Ricci-flat K\"{a}hler cone potential, we mean a locally bounded function on $Y$ that is a smooth radius function over $Y^{\rm reg}$ and satisfies the Ricci-flat K\"{a}hler cone equation in the pluripotential sense. 
In other words we indeed get the following theorem, which generalizes the result of Collins-Sz\'{e}kelyhidi in \cite{CS19} (see also \cite{Berm20} for singular toric case) . Since all ingredients in the proof have been ready, we only sketch its proof, which is a technical refinement of the isolated singularity case.
\begin{thm}
For any polarized Fano cone $(Y, \hat{\xi})$, there exists a (weak) Ricci-flat K\"{a}hler cone potential if and only if $(Y, \hat{\xi})$ is Ding-polystable, if and only if any quasiregular quotient is $g$-weighted polystable (where $g$ is equal to $(n+1+\la \cdot, \xi\ra)^{-n-2}$ as above). Moreover it suffices to test these stability condition over special test configurations. 
\end{thm}
\begin{proof}[Sketch of proof]
It suffices to prove the direction of stability to existence since the other direction has been proved in \cite{CS19, LWX21} for stability of $Y$ and for $(X, D)$ in \cite{HL20}. We have explained that Ding-polystability of $Y$ is equivalent to the $g$-weighted Ding-polystability of $(X, D)$, which by Theorem \ref{thm-BLXZ} is equivalent to reduced uniformly $g$-weighted stability and can be tested over special test configurations by \cite{HL20, LX14}. By Theorem \ref{thm-YTDlog}, we get a weak $g$-soliton potential.  Moreover by combining the orbifold resolution process in \cite{LT19} with the proof of \cite[Proposition 32]{HL20}, we know that the weak $g$-soliton potential is globally bounded (see \cite[Corollary 31]{HL20}) and is smooth on the orbifold locus $X^{\rm orb}$ of $(X, D)$. By the orbifold locus, we mean the quotient of $Y^{\rm reg}$ by the $\bC^*$-action generated by $v_{\hat{\chi}}$. In particular, the coefficients of $D$ around any point $p\in X^{\rm orb}$ is of the form $1-d^{-1}\in [\frac{1}{2}, 1)$ for some $d\ge 2$. As mentioned above, it has been well-understood how to characterize the smoothness of $Y^*$ along $\pi^{-1}(p)$ for any $p\in X^{\rm orb}$ using the local data of $(X, D)$ (see \cite{Kol04} and \cite[Theorem 4.7.5]{BG08}). 
As a consequence, over $Y^{\rm reg}$ we have a smooth radius function $r_0$ induced by the orbifold smooth $g$-soliton potential. In the following paragraph, we will argue that the deformation $r=r_0^{\hat{\xi}}$ of $r_0$ over $Y^{\rm reg}$ with respect to $\hat{\xi}$ is still well-defined. The assumption that $r_0$ satisfies the $g$-soliton equation implies that $r$ satisfies the Ricci-flat K\"{a}hler cone equation over $Y^{\rm reg}$ by the same reason as in the isolated singularity case (see the diagram \eqref{eq-transdiag}). Moreover we will argue that $r$ can be extended to become a locally bounded psh function that solves the Ricci-flat K\"{a}hler cone equation globally over $Y$. 

Fix an equivariant embedding of $Y$ into $\bC^N$ and let $\tilde{r}_0$ (resp. $\tilde{r}$) be a smooth radius function of $\bC^N$ with respect to $\hat{\chi}$ (resp. $\hat{\xi}$). 
For any $a< b\in \bR_{>0}$, define the subset $\mathcal{U}:=\mathcal{U}_{a,b}=\{q\in Y^{\rm reg}; a<\tilde{r}_0(q)<b\}$ and its closure $\bar{\mathcal{U}}$ which is an open subset of $Y$. The $g$-soliton relative potential is then equal to the logarithmic of $r_0^2/\tilde{r}_0^2$ and is known to be globally bounded. So $r_0$ is uniformly bounded over $\bar{\mathcal{U}}$. 
Recall that the deformation $r(q)=r_0^{\hat{\xi}}(q)$, for any $q\in \mathcal{U}$, is obtained by first solving the integral curve $\gamma_q(t)$ of $J\hat{\xi}$ with the initial condition $q$ and then solving $r_0(\gamma_q(t_*))=1$ to get $r(q)=e^{t_*(q)}$ (see \ref{sec-defReeb}). 
Because $Y^{\rm reg}$ is invariant under the $(\bC^*)^r$ action, any such integral curve $\gamma_q(t)$ stays in $Y^{\rm reg}$ for all time. 
We claim that the derivative $dr_0/dt$ along an integral curve $\gamma_q(t)$, which is equal to $-\eta_0(\gamma_q(t))(\hat{\xi})$ by \eqref{eq-drdt}, is negative and uniformly bounded away from $0$. This implies that $\gamma_q(t)$ will hit the smooth subset $\{r_0=1\}\cap Y^{\rm reg}$ in uniformly bounded time (which can be either positive or negative), which in turn implies that $r$ is a well-defined smooth function on $\mathcal{U}$ and is indeed uniformly bounded over $\mathcal{U}$. 
\footnote{The uniformly boundedness of $r$ over $\mathcal{U}$ can also be derived using the uniformly boundedness of $r_0$ using the fact that $\{r=1\}\cap Y^{\rm reg}=\{r_0=1\}\cap Y^{\rm reg}$ and then using the rescaling property of $r$ with respect to $J\hat{\xi}=-r\partial_r$. }
$r$ is then a radius function for $\hat{\xi}$ and satisfies the Ricci-flat K\"{a}hler cone equation over $\mathcal{U}$ by the same reason as in the isolated singularity case. 
Now because $\bar{\mathcal{U}}\setminus \mathcal{U}$ is an analytic subset of complex co-dimension at least 2, the uniformly bounded strictly psh function $r^2$ over $\mathcal{U}$ extends to be a bounded plurisubharmonic function over $\bar{\mathcal{U}}$ that satisfies the Ricci-flat K\"{a}hler cone equation globally on $\bar{\mathcal{U}}$. Moreover because of the rescaling property of $r$, we can write $r=\tilde{r} e^{\vphi^{\hat{\xi}}/2}$ for a uniformly bounded function $\vphi^{\hat{\xi}}$ on $Y^*$. Now because $a, b$ are arbitrary and also $r(q)$ converges to $0$ as $q$ converges to the vertex, we conclude that $r$ satisfies the Ricci-flat K\"{a}hler cone equation globally over $Y$.

Finally the uniform negativity of $-\eta_0(\gamma_q(t))(\hat{\xi})$ claimed above
comes from the fact that $\hat{\xi}$ is in the interior of the Reeb cone and the image of $\eta_0(Y^{\rm reg})$ is contained in the moment polytope of $X$. This latter containment property follows from the fact that the $g$-soliton potential can be globally approximated by smooth quasi-psh potentials which converge smoothly over the orbifold locus of $(X, D)$ (again by using \cite{LT19, HL20} together). 
\end{proof}
\begin{rem}
We are informed by Chenyang Xu that there is a work in progress by Kai Huang which aims to prove a valuative criterion for the stability of Fano cones without using the above correspondence with $g$-weighted stability. 
\end{rem}

\appendix

\section{Chern curvature of associated holomorphic line bundle}\label{App-equiv}

We calculate the Chern curvature of the holomorphic line bundles over a fibre bundle that arises as the associated bundle to a holomorphic principal $\bC^*$- bundle. See \cite{Don05} for a discussion involving only the connections. 

Let $P\rightarrow B$ be (the total space of) a holomorphic line bundle over a complex manifold $B$. We denote the complement of the zero section of $P$ by $P^*$, which is a holomorphic principal $\bC^*$-bundle. Assume $P$ is equipped with a smooth Hermitian metrics $h_P$. Let $\bar{P}\rightarrow B$ the associated unit circle bundle which is a principal $S^1$-bundle over $B$. Let $L\rightarrow X$ be holomorphic line bundle with a holomorphic $\bC^*$-action and an $S^1$-invariant Hermitian metric $h$. Consider the associated space $(Y, F):=(P^*\times (X, L))/\bC^*$ where $\bC^*$-acts on $P^*\times (X, L)$ by the action
\begin{equation}
t\circ (y, (x, s))=(y\cdot t^{-1}, t\circ x, t\circ s).
\end{equation}
Then $Y$ is a holomorphic fibre bundle over $B$ and $F\rightarrow Y$ is a holomorphic line bundle equipped with a Hermitian metric induced by the metric $h_F$. We will calculate the Chern curvature form of this metric as follows. Fix a point $b_*\in U$. 
First we choose a local nowhere vanishing holomorphic section $\sigma$ (resp. $s$) of $P$ (resp. $L$) over a small open set $U\subset B$ (resp. over an open subset $V\subset X$).  
The holomorphic trivialization over $U\times V\subset B\times X$ can be chosen to be $s'=\sigma\cdot s$. Set $t(b):=\log |\sigma(b)|_{h_P}$ and $r(b)=|\sigma(b)|_{h_P}$. The Hermitian norm of the induced metric is equal to 
\begin{equation}
|s'|_{h_F}^2=\left||\sigma|_{h_P} \circ s\right|_{h_L}^2=|\exp(t(b))\circ s(x)|_{h_L(e^{t}\circ x)}^2.
\end{equation}
Choose local holomorphic coordinates $\{b_i\}$ over $U$, and local holomorphic coordinates $\{x_\alpha\}$ over $V$. 
We have a local holomorphic parametrization $\sigma: U\times V\rightarrow Y$. 
Set
\begin{equation}
u(t, x)=\log |\exp(t)\circ s(x)|_{h_L}^2, \quad f(b, x):=u(t(b), x)=\log |s'|_{h_F}^2.
\end{equation} 
We need to calculate the following form:
\begin{eqnarray*}
\ddb f&=& f_{i\bar{j}}db_i\wedge d\bar{b}_j+f_{i\bar{\beta}}db_i\wedge d\bar{x}_\beta+f_{\alpha\bar{j}}dx_\alpha \wedge d\bar{b}_j+f_{\alpha\bar{\beta}}dx_\alpha\wedge d\bar{x}_\beta. 
\end{eqnarray*} 
Note that:
\begin{eqnarray*}
f_{i}=u_t t_i, \quad f_{i\bar{j}}=u_{tt} t_i t_{\bar{j}}+u_t t_{i\bar{j}}, \quad f_{i\bar{\beta}}=u_{t\bar{\beta}}t_i, \quad f_{\alpha\bar{\beta}}=u_{\alpha\bar{\beta}}.
\end{eqnarray*}
Fix any point $b_*\in U$. We first assume that $\sigma$ satisfies $t_i:=(\partial_{b_i}t)(b_*)=0$. This corresponds to the assumption that $\sigma$ is horizontal to the first order at $b_*$ and hence $\sigma_*(\partial_{b_i})$ is horizontal at point $\sigma(b_*, x)$. 
Then at points $\sigma(b_*, x)$, we get 
\begin{eqnarray*}
\ddb f(b_*)&=&u_t(b_*) t_{i\bar{j}}db_i\wedge d\bar{b}_j+u_{\alpha\bar{\beta}} dx_\alpha\wedge d\bar{x}_\beta\\
&=&\frac{1}{\sqrt{-1}}\left[(r(b_*)^* \theta) \omega_P+r(b_*)^* \omega_L\right]. 
\end{eqnarray*}
In particular, if $r(b_*)=1$ then at point $b_*$, we have
\begin{equation}\label{eq-equicurv}
\sddb \log |s'|^2_{h_F}= \theta \omega_P+\omega_L
\end{equation} 
with respect to the splitting $TY=TY^h\oplus TY^v$ into horizontal and vertical parts. 

For general $\sigma$, choose $\lambda\in \mcO_{B}^*(U)$ with $\lambda(b_*)=1$ such that $\tilde{\sigma}=\sigma\cdot \lambda^{-1}$
\begin{equation}
\tilde{t}_{i}(b_*):=(\partial_{b_i}\tilde{t})(b_*)=0
\end{equation} 
where $\tilde{t}=\log |\tilde{\sigma}|_{h_P}=-\log |\lambda|+\log |\sigma|_{h_P}$, i.e. $t(b)=\tilde{t}(b)+\log |\lambda(b)|$.    
Then at point $b_*$, we have:
\begin{eqnarray*}
\lambda_{b_i}= \partial_{b_i}\log \lambda=\partial_{b_i}\log |\lambda|^2= 2 \partial_{b_i}\log |\sigma|_{h_P}= 2 t_i. 
\end{eqnarray*}
So at $(b_*, x)$, 
\begin{eqnarray*}
\partial_{b_i}=\sigma_*(\partial_{b_i})=\left(\tilde{\sigma}\cdot \lambda\right)_*(\partial_{b_i})=\tilde{\sigma}_*(\partial_{b_i})+\tilde{\sigma}_*( \lambda_{b_i} v)=\partial_{b_i}^{h}+2 t_i v
\end{eqnarray*}
where $v=\sum_\alpha\theta^\alpha\partial_{\alpha}$ is the generator of the $\bC^*$-action. So we get the identity:
\begin{eqnarray*}
f_{i\bar{j}}&=& (r(b)^*\theta) (\omega_P)_{i\bar{j}}+ 4 t_i t_{\bar{j}}|v|^2_{r(b)^*\omega_L}, \\
f_{i\bar{\beta}}&=& 2 t_i (r(b)^*\omega_L)(v, \partial_{\bar{x}_\beta}); \quad
f_{\alpha\bar{\beta}}=(r(b)^*\omega_L)_{\alpha\bar{\beta}}.
\end{eqnarray*}

\section{Moment map associated to torus actions}

Set $\hat{\bT}=(\bC^*)^{r+1}$ and $T=(S^1)^{r+1}\subset \hat{\bT}$. 
Assume that $\hat{\bT}$ acts on $\bC^N$ linearly with weights $w_1, \dots, w_{N}\in \bZ^{r+1}$. 
Let $Y={\rm Spec}(\hat{R})$ be a $\hat{\bT}$-invariant affine variety in $\bC^N$. The $\hat{\bT}$-action corresponds to a weight decomposition $\hat{R}=\bigoplus_\alpha \hat{R}_\alpha$. Choose a Hermitian inner product on $\bC^N$ that is $\hat{T}$-invariant and set $\hat{\omega}=\sqrt{-1}\sum_i dZ_i\wedge d\bar{Z}_i$. Then the moment map is given by 
\begin{equation}
\la\mu(Z), \hat{\xi}\ra=\sum_{i} \la w_i, \hat{\xi} \ra |Z_i|^2.
\end{equation}

\begin{thm}[see \cite{BV11, Sja98}]
Assume that $Y$ is irreducible. Then the image of the moment map $\mu(Y)=:\hat{P}$ is a polyhedral cone that is the closure of the set 
$$\{\alpha; \hat{R}_\alpha\neq 0\}.$$ There is a one-to-one correspondence between $\hat{T}$-orbits and the facets of $\hat{P}$. 

In general if $Y=\sum_k Y_k$ the irreducible decomposition of $Y$, then $\mu(Y)$ is the union of polyhedral cones $\mu(Y_k)$. 
\end{thm}
Let $\hat{\chi}\in \hat{N}_\bR^+$ be a quasi-regular element and let $(X, D)=Y\sslash\la v_{\hat{\chi}}\ra$ as the GIT quotient. Then $Y$ becomes an orbifold cone over $(X, D)$. 
In other words, there exists an orbifold line bundle $\hat{L}$ which can be considered as a $\bQ$-divisor over $X$ such that
\begin{equation}
\hat{R}=\bigoplus_m H^0(X, \lfloor mL \rfloor).
\end{equation}

Choose $M\gg 1$ such that $L:=M\hat{L}$ is a genuine line bundle. We can also assume that $L$ is very ample and set $\bP^{N-1}=\bP(H^0(X, L))$. 
Let $\bT\cong (\bC^*)^r$. 
There is then an effective $\bT$-action on $\bP^N$ and $X$ is an $\bT$-invariant subvariety.  
$R=\hat{R}^{(M)}$ is the $M$-th Veronese subalgebra of $\hat{R}$ and is generated in degree 1.

Set $R_m=H^0(X, mL)$ and $R=\bigoplus_m R_m$. 
The $\bT$-action gives rise to a weight decomposition:
\begin{equation}
R_m=\bigoplus_\alpha R_{m,\alpha}, \quad
R(X, L)=\bigoplus_{m,\alpha} R_{m,\alpha}
\end{equation}
where $R_{m,\alpha}=\{s\in R_m; t\circ s=t^\alpha s\}$. 
Assume that $\{s_i\}$ is a basis compatible with the weight decomposition. Then we can choose a $T$-invariant Fubini-Study metric:
\begin{equation}
\omega=\sddb \log \sum_i |s_i|^2.
\end{equation}
The moment map is then given by:
\begin{equation}
\bm(\xi)=\frac{\sum_i \la \alpha_i, \xi\ra |s_i|^2}{\sum_i |s_i|^2}.
\end{equation}
The following is the projective version of the convexity theorem of  Atiyah-Guillemin-Sternberg. 
\begin{thm}[see \cite{BV11, Sja98}]
Assume that $X$ is irreducible. 
\begin{enumerate}
\item
The image of the moment map is a polyhedral polytope $P$, called the moment polytope. Moreover there is a one-to-one correspondence between the $\bT$-orbit and the facets of $P$. The vertex of $P$ consists of the images of $\bT$-fixed points on $X$. 
\item
Let $\{\alpha^{(m)}_i\}$ denote the set of weights such that $R_{m,\alpha^{(m)}_{\alpha_i}}\neq 0$. Then the moment polytope
$P$ is the closure of the set $\mathfrak{W}:=\left\{\frac{\alpha^{(m)}_i}{m}; m\in \bN\right\}$. 

\end{enumerate}
\end{thm}

Define the measure $\bm_*(\omega^n)$ and the weight measure:
\begin{equation}
\nu_m=\frac{n!}{m^n}\sum_i \dim R_{m,\alpha}\cdot \delta_{\frac{\alpha}{m}}.
\end{equation}
We also need:
\begin{thm}\label{thm-numconv}
As $m\rightarrow+\infty$, $\nu_m$ converges weakly to $(2\pi)^{-n}\bm_*(\omega^n)$. 
\end{thm}
By Stone-Weierstrass theorem, it suffices to prove that for any polynomial $\prod_{\kappa} (x_\kappa+c)^{d_\kappa}$ over the moment polytope $P$
with $P+c(1,\dots, 1)\in \bR_{>0}^r$ there is a convergence:
\begin{equation}
\lim_{m\rightarrow+\infty} \frac{n!}{m^n} \sum_\alpha \prod_{\kappa}\left(\la\frac{\alpha}{m}, \xi_\kappa\ra+c\right)^{d_\kappa} \dim R_{m,\alpha}=(2\pi)^{-n}\int_X  \prod_{\kappa}\theta_\kappa^{d_\kappa} \omega^n.
\end{equation}
This can be proved by using the equivariant Riemann-Roch theorem in the same way as used in \cite{Don05} (see also \cite{BHJ17}). A different argument based on spectral analysis of Toeplitz operator was given in \cite[Proposition 4.1]{BW14}. 

Choose a basis $\{\hat{e}_0, \hat{e}_1, \dots, \hat{e}_{r-1}\}$ for $\hat{N}_\bZ$ such that $\hat{e}_0\in \bR_{+} \hat{\chi}$ and let $\{\hat{e}_0^*, \hat{e}_1^*, \dots, \hat{e}_r\}$ be the corresponding dual basis. Then $N_\bZ=\hat{N}_\bZ/\bZ \hat{e}_0$ is spanned by $\{\bar{e}_1,\dots, \bar{e}_r\}$ where $\bar{e}_i$ is the image of $\hat{e}_i$ under the canonical projection. 
We have an isomorphism:
\begin{eqnarray*}
\hat{N}_\bZ
&\stackrel{\cong}{\longrightarrow} & \bZ\oplus N_\bZ\\
k\cdot  \hat{e}_0+ \sum_i a_i \hat{e}_i
&\mapsto & (k, \sum_i a_i \bar{e}_i) .
\end{eqnarray*}
We have the dual map
\begin{eqnarray*}
\bZ\oplus M_\bZ 
& \stackrel{\cong}{\longrightarrow} & \hat{M}_\bZ\\
(m, \alpha)&\mapsto & m\cdot \hat{e}_0^*+\sum_i \alpha_i(\bar{e}_i) \hat{e}_i^*. 
\end{eqnarray*}

$\hat{e}_i^*$ is located in the hyperplane $\ell_{\hat{e}_0}=0$. The moment cone $\hat{P}$ is the cone over $M \cdot \hat{e}^*_0+P$.

\section{Filtrations}

Assume that $L$ is an ample line bundle or an orbifold line bundle over $X$. Set $R_m=H^0(X, mL)$ and $N_m=\dim_\bC H^0(X, mL)=\frac{L^{\cdot n}}{n!}m^{n}+O(m^{n-1})$.  
\begin{defn}\label{defn-gdfiltr}
A filtration $\mcF R_\bullet$ of the graded $\bC$-algebra $R=\bigoplus_{m=0}^{+\infty}R_m$ consists of a family of subspaces $\{\mcF^x R_m\}_x$ of $R_m$ for each $m\ge 0$ satisfying:
\begin{itemize}
\item (decreasing) $\mcF^x R_m\subseteq \mcF^{x'}R_m$, if $x\ge x'$;
\item (left-continuous) $\mcF^xR_m=\bigcap_{x'<x}\mcF^{x'}R_m$; 
\item (multiplicative) $\mcF^x R_m\cdot \mcF^{x'} R_{m'}\subseteq \mcF^{x+x'}R_{m+m'}$, for any $x, x'\in \bR$ and $m, m'\in \bZ_{\ge 0}$;
\item (linearly bounded) There exist $e_-, e_+\in \bZ$ such that $\mcF^{m e_-} R_m=R_m$ and $\mcF^{m e_+} R_m=0$ for all $m\in \bZ_{\ge 0}$. 
\end{itemize}
\end{defn}

Given any filtration $\{\mcF^{x}R_m\}_{x\in\bR}$ and $m\in \bZ_{\ge 0}$, the successive minima on $R_m$ is the decreasing sequence 
\[
\lambda^{(m)}_{\max}=\lambda^{(m)}_1\ge \cdots \ge \lambda^{(m)}_{N_m}=\lambda^{(m)}_{\min}
\]
defined by:
\[
\lambda^{(m)}_j=\max\left\{\lambda \in \bR; \dim_{\bC} \mcF^{\lambda} R_m \ge j \right\}.
\]
Denote $\cF^{(\lambda)}:=\cF^{(\lambda)}R=\bigoplus_{k=0}^{+\infty} \cF^{m\lambda}R_m$ and define 
\begin{equation}
\vol\left(\cF^{(\lambda)}\right)=\vol\left(\cF^{(\lambda)}R\right):=\limsup_{k\rightarrow+\infty}\frac{\dim_{\bC}\cF^{m\lambda}H^0(X, mL)}{m^{n}/n!}.
\end{equation}
\begin{prop}[{\cite{BC11}, \cite[Corollary 5.4]{BHJ17}}]\label{BHJvol}
\begin{enumerate}[(1)]
\item 
The Radon measure 
\[
\frac{n!}{m^n} \sum_{j}\delta_{m^{-1}\lambda^{(m)}_j}=-\frac{d}{d\lambda}\left(\frac{n!}{m^n} {\rm dim}_{\bC} \cF^{m\lambda}H^0(X, mL)\right) d\lambda
\]
converges weakly as $m\rightarrow+\infty$ to the Radon measure:
\[
\DHM(\mcF):=-d\; \vol\left(\cF^{(\lambda)}\right)=-\frac{d}{d\lambda} \vol\left(\cF^{(\lambda)}\right) d\lambda.
\]
\item The support of the measure $\DHM(\mcF)$  is given by ${\rm supp}(\DHM(\mcF))=[\lambda_{\min}, \lambda_{\max}]$ with 
\begin{align}
&
\displaystyle \lambda_{\min}:= \lambda_{\min}(\mcF):=\inf\left\{t\in \bR; \vol\left(\cF^{(\lambda)}\right)< L^{\cdot n} \right\}; \label{lambdamin} \\
&
\lambda_{\max}:=\lambda_{\max}(\mcF):=\lim_{m\rightarrow+\infty}\frac{\lambda_{\max}^{(m)}}{m}=\sup_{m\ge 1}\frac{\lambda_{\max}^{(m)}}{m}.\label{lambdamax}
\end{align}
\end{enumerate}
\begin{rem}\label{rem-supadd}
The limit in the \eqref{lambdamax} exists because $\{\lambda^{(m)}_{\max}\}_{m\in \bZ_{>0}}$ is superadditive in the sense that $\lambda^{(m+m')}_{\max}\ge \lambda^{(m)}_{\max}+\lambda^{(m')}_{\max}$, by the multiplicative property of filtrations in Definition \ref{defn-gdfiltr}. 
\end{rem}
\end{prop}
\begin{exmp}

\begin{enumerate}
\item
Any divisorial valuation also defines a filtration.
\begin{equation}\label{eq-filval}
\cF_v^\lambda R_m=\{s\in R_m; v(s)\ge \lambda\}
\end{equation}
where $v(s):=v(f)$ if $s=f \cdot e$ with $f\in \mcO^*_X$ and $e$ is a nowhere vanishing local holomorphic section of $L$. 
\item
Any normal test configuration $(\mcX, \mcL)$ defines a filtration: 
\begin{equation}\label{eq-filTC}
\cF^\lambda_{(\mcX, \mcL)}R_m:=\{s\in R_m; t^{-\lceil \lambda\rceil} \bar{s} \in H^0(\mcX, m\mcL)\}
\end{equation}
where the meromorphic section $\bar{s}$ is the pull back of $s$ under the equivariant isomorphism $\mcX\setminus \mcX_0\cong X\times \bC^*$. 
There is a more concrete description. Assume $\mcX_0=\sum_i b_i E_i$ and choose an equivalent dominant test configuration $(\mcX', \mcL')$ with the birational morphism $\rho: \mcX'\rightarrow X\times\bC$ that satisfies $\mcL'=\rho^*L_\bC+D$. Then we get:
\begin{equation}\label{eq-filTC}
\cF^\lambda_{(\mcX, \mcL)}R_m=\left\{s\in R_m; r(\ord_{E_i})(s)+m \cdot \ord_{E_i}(D)\ge \lceil \lambda \rceil b_i\right\}.
\end{equation}
where $r(\ord_{E_i})$ is the restriction of the valuation $\ord_{E_i}$ to the function field $\bC(X)$. 

An important case is when $(\mcX, -K_{\mcX})$ is a test configuration of $(X, -K_X=L)$. In other words, we assume that $\mcX$ is normal and $-K_{\mcX}$ is semiample. We then have the identity:
\begin{equation}
\ord_{E_i}(D)=-a_{X\times\bC}(E_i)=-(a_{X}(r(E_i))+b_i)=-b_i A(b_i^{-1}r(\ord_{E_i})).
\end{equation}  
We can then set $v_{E_i}=b_i^{-1}r(\ord_{E_i})$ to get the identity:
\begin{equation}
\cF^\lambda_{(\mcX, \mcL)}R_m=\left\{s\in R_m; v_{E_i}(s)-A(v_{E_i})m\ge \lceil \lambda \rceil  \right\}.
\end{equation}

\end{enumerate}
\end{exmp}

Any filtration $\mcF R_\bullet$
can be ``approximated" by a sequence of test configurations defined as follows.
Choose $e_-$ and $e_+$ as in the definition \ref{defn-gdfiltr}. For convenience, we can choose $e_+=\lceil \lambda_{\rm max}(\mcF R)\rceil \in \bZ$. 
Set $e=e_+-e_-$ and define (fractional) ideals:
\begin{eqnarray}
I_{(m,x)}&:=&I^\mcF_{(m,x)}:={\rm Image}\left(\mcF^x R_m\otimes \mcO_Z(m  L)\rightarrow \mcO_Z\right); \label{eq-Imx}\\
\tilde{\mcI}_m&:=&\tilde{\mcI}^{\mcF}_m:= I^{\mcF}_{(m, m e_+)}t^{-m e_+}+I^{\mcF}_{(m,me_+-1)}t^{1-m e_+}+\cdots\nonumber \\
&&\hskip 4cm \cdots+ I^{\mcF}_{(m, me_-+1)}t^{-m e_--1}+\mcO_Z\cdot t^{-me_-}; \label{eq-tcIm}\\
\mcI_m&:=&\mcI_m^{\mcF(e_+)}=\tilde{\mcI}^\mcF_m\cdot t^{m e_+}=I^\mcF_{(m, m e_+)}+I^{\mcF}_{(m, m e_+-1)} t^1+\cdots\nonumber\\
&&\hskip 4cm \cdots+I^{\mcF}_{(m, m e_-+1)} t^{me-1}+(t^{me})\subseteq \mcO_{Z_\bC}. \label{eq-cIm}
\end{eqnarray}
Let $\mu_m: \check{\mcX}_m\rightarrow X\times\bC$ be the normalized blowup of $\mcI_m$ with $\mu_m^*\mcI_m=\mcO_{\mcX_m}(-E_m)$. Set $\mcL_m=\mu_m^*L_\bC-\frac{1}{m}E_m$. For any $v\in X^{\rm div}_\bQ$, set:
\begin{eqnarray*}
&&\phi_{\cF}(v)=\lim_{m\rightarrow+\infty} -\frac{1}{m} G(v)(\tilde{\mcI}_m)=\lim_{m\rightarrow+\infty} \phi_{(\check{\mcX}_m, \check{\mcL}_m)}(v)\\
\end{eqnarray*}
Assume $\bT\cong (\bC^*)^r$ acts on $X$ effectively. Let $M_\bZ={\rm Hom}_{\rm alg}(\bT, \bC^*)$ be the weight lattice. Assume that $\cF$ is a $\bT$-equivariant test configuration. Then for any $m\in \bN$ and $\lambda \in \bR$, there is a weight decomposition 
\begin{equation}
R_m=H^0(X, mL)=\bigoplus_{\alpha\in M_\bZ} R_{m,\alpha}, \quad 
\cF^\lambda R_m=\bigoplus_{\alpha\in M_\bZ} (\cF^\lambda R_m)_\alpha
\end{equation}
where $\cF^\lambda R_m=\{s\in R_m; t\cdot s=t^\alpha s\}$. 
Let $g$ be a positive smooth function on the moment polytope $P$ as before. 
For each $t\in \bR$, define a volume function:
\begin{equation}\label{eq-volg}
\vol_g(\cF^{(\lambda)})=\lim_{m\rightarrow+\infty} \frac{n!}{m^n} \sum_\alpha g(\frac{\alpha}{m}) \dim_\bC (\cF^{tm}R_m)_\alpha.
\end{equation}
Let $\{\lambda^{(m,\alpha)}_j\}$ be the successive minima of $R_{m,\alpha}$:
\begin{equation}
\lambda^{(m,\alpha)}_j=\max\{\lambda; \dim \cF^{\lambda} R_{m,\alpha}\ge j\}.
\end{equation}

Consider the measure:
\begin{equation}
\nu^g_m=\sum_{\alpha,i}g(\frac{\alpha}{m})\delta_{\frac{\lambda^{(m,\alpha)}_j}{m}}.
\end{equation}
As $m\rightarrow+\infty$, $\nu^g_m$ converges to the measure $-d\vol_g(\cF^{(\lambda)})$. Define:
\begin{eqnarray*}
\bfE^\NA_g(\cF)&=&\frac{1}{\bV_g}\frac{n!}{m^n}\lim_{m\rightarrow+\infty} \sum_{\alpha,i} g(\frac{\alpha}{m})\frac{\lambda^{(m,\alpha)}_j}{m}=\frac{1}{\bV_g}\int_\bR \lambda (-d\vol_g(\cF^{(\lambda)}))\\
&=&\lambda^g_{\min}(\cF)+\frac{1}{\bV_g}\int_{\lambda^g_{\min}(\cF)}\vol_g(\cF^{(\lambda)})d \lambda
\end{eqnarray*}
where $\lambda^g_{\min}(\cF)=\inf \left\{\lambda; \vol_g(\cF^{(\lambda)})<\bV_g\right\}$. 
Note that this is monotone with respect to $\cF$:
\begin{equation}
\cF_1\le \cF_2\quad \Longrightarrow\quad \bfE^\NA_g(\cF_1)\le \bfE^\NA_g(\cF_2). 
\end{equation}
Here by $\cF_1\le \cF_2$ we mean that $\cF_1^\lambda R_m\le \cF_2^\lambda R_m$ for any $\lambda\in \bR$ and $m\in \bN$. 

\begin{exmp}
For any divisorial valuation $v\in X^{\rm div}_\bQ$ with associated filtration $\cF_v$, we have
\begin{equation}\label{eq-Sgv0}
\bfE^\NA_g(\cF_v)=\frac{1}{\bV_g}\int_0^{+\infty}\vol_g(\cF^{(\lambda)}_v)d\lambda=:S_g(v). 
\end{equation}
\end{exmp}

\begin{lem}\label{lem-phiSgeE}
For any test configuration $(\mcX, \mcL)$ with associated non-Archimedean metric $\phi=\phi_{(\mcX, \mcL)}$ and for any divisorial valuation $v\in X^{\rm div}_\bQ$, we have the inequality:
\begin{equation}\label{eq-phiSgeE}
\phi(v)+S_g(v)\ge \bfE_g^\NA(\phi). 
\end{equation}
\end{lem}
\begin{proof}
Let $\cF=\cF_{(\mcX, \mcL)}$ be the filtration associated to $(\mcX, \mcL)$. 
Set $a=\phi(v)$. Then we get 
\begin{equation}
\phi(v)+S_g(v)=a+\bfE^\NA_g(\cF_v)=\bfE^\NA_g(\cF(a))
\end{equation} 
where $\cF(a)^\lambda R_m=\cF^{\lambda-a}R_m$. 
By the monotonicity of $\bfE^\NA_g$, we just need to show that $\mcF^\lambda R_m\subseteq \mcF^{\lambda-am}_v R_m$. For any $s\in \cF^\lambda R_m$, $t^{-\lceil\lambda\rceil}\bar{s}$ extends across $\mcX_0$. Without loss of generality, we can assume that there is a dominating morphism $\rho: \mcX\rightarrow X\times\bC$ with $\mcL=\rho^*L_\bC+D$. Then we get the wanted inequality easily:
\begin{eqnarray*}
v(s)=G(v)(\bar{s})=G(v)(t^{-\lceil \lambda\rceil}\bar{s})+\lceil\lambda\rceil\ge G(v)(D)+\lceil \lambda \rceil=-m a+\lceil \lambda \rceil\ge \lambda-ma.
\end{eqnarray*}
\end{proof}

Moreover one can define non-Archimedean functionals for filtrations:
\begin{eqnarray*}
&&\Lam^\NA_g(\mcF)=\lambda_{\max}(\cF) \\
&&\bfJ^\NA_g(\mcF)=\Lam^\NA_g(\mcF)-\bfE^\NA_g(\cF) \\
&&\bfL^\NA(\mcF)=\lim_{m\rightarrow+\infty}\bfL^\NA(\check{\mcX}_m, \check{\mcL}_m) \\
&&\bfD^\NA_g(\mcF)=-\bfE^\NA_g(\cF)+\bfL^\NA(\cF).
\end{eqnarray*}

\begin{thm}
Let $\bfF\in \{\bfE, \Lam, \bfJ, \bfL, \bfD\}$. 
\begin{enumerate}
\item For a test configuration $(\mcX, \mcL)$ with associated filtration in \eqref{eq-filTC},  $\bfF^\NA_g(\cF_{(\mcX, \mcL)})=\bfF^\NA_g(\mcX, \mcL)$.  
\item
In general, for any filtration, we have:
\begin{equation}
 \lim_{m\rightarrow+\infty}\bfF^\NA_g(\check{\mcX}_m, \check{\mcL}_m)=\bfF^\NA_g(\cF).
\end{equation} 
\end{enumerate}
\end{thm}
We define the twist of filtrations:
\begin{equation}\label{eq-filtwist}
(\cF_\xi^\lambda R_m)_\alpha=\cF^{\lambda-\la \alpha, \xi\ra} R_{m,\alpha}.
\end{equation}
This generalizes the twist of test configurations because we have $\cF^\lambda_{(\mcX_\xi, \mcL_\xi)}R_m=(\cF^{\lambda}_{(\mcX, \mcL)})_\xi R_m$. 
If $v\in X^{\rm div}_\bQ$ be a $T$-invariant valuation, then for any $\xi\in N_\bR$, there is a twist $v_\xi$: if $f=\sum_\alpha f_\alpha\in \mcO_X$ satisfying $f_\alpha \in \mcO_{X}$ and $t\circ f_\alpha=t^\alpha f_\alpha$, then 
\begin{equation}\label{eq-vxi}
v_\xi(f)=\min_{\alpha} \{v(f_\alpha)+\la \alpha, \xi\ra\}. 
\end{equation}

which defines a non-Archimedean metric $\phi=\phi_{(\mcX, \mcL)}$. For any $\xi\in N_\bQ$, we have a twist $(\mcX_\xi, \mcL_\xi)$ which also defines a non-Archimedean metric which we denote by $\phi_\xi$. 
Moreover for any $v\in X^{\rm div}_\bQ$ and $\xi\in N_\bQ$, there exists a twist $v_\xi$. We have a useful identity:
\begin{thm}[\cite{Li19}]
For any filtration $\cF$ and any $\xi\in N_\bQ$, we have: 
\begin{equation}\label{eq-ENAxi}
\bfE^\NA_g(\cF_\xi)=\bfE^\NA_g(\cF)+\Fut_g(\xi). 
\end{equation}
For any $\phi=\phi_{(\mcX, \mcL)}$ for any test configuration $(\mcX, \mcL)$ and any $v\in X^{\rm div}_\bQ$ we have the identity:
\begin{equation}\label{eq-Aphixi}
A_X(v_{-\xi})+\phi_{\xi} (v_{-\xi})=A_X(v)+\phi(v)
\end{equation} 
where $\phi_{\xi}=\phi_{(\mcX_{\xi},\mcL_{\xi})}$. 
As a consequence, we get $\bfL^\NA(\cF)=\bfL^\NA(\cF_\xi)$. 
\end{thm}
\begin{thm}[\cite{Li19}]
For any $v\in X^{\rm div}_\bQ$ and any $\xi\in N_\bQ$, we have the identity:
\begin{equation}
\cF_{v_\xi}=(\cF_v)_\xi (b_\xi)
\end{equation}
where $b_\xi=A_X(v_\xi)-A_X(v)$. Moreover we have the identity:
\begin{equation}
A_X(v_\xi)-S_g(v_\xi)=A_X(v)-S_g(v)+\Fut_g(\xi).
\end{equation}
\end{thm}


\section{Multiplier ideal sheaf}

Let $\vphi$ be a psh function on a domain $U\subset \bC^n$. Then define the $m$-th multiplier ideal sheaf:
\begin{equation}
\cJ(m\vphi)(U)=\{f\in \mcO(U); \int_U |f|^2 e^{-m\vphi}dV_{\bC^n}<+\infty\}. 
\end{equation}
Take an orthonormal basis $\{f^{(m)}_i; i\in \bN\}$ of $\cJ(m\vphi)$ and set:
\begin{equation}
\vphi_m=\frac{1}{m}\log \sum_{i} |f^{(m)}_i|^2.
\end{equation}
Then by using Ohsawa-Takegoshi extension theorem, Demailly proved 
\begin{thm}\label{thm-Demailly}
There exists constants $C_1>0$ such that for any $m\in \bN$ and any $z\in U$
 \begin{equation}\label{eq-vphile}
 \vphi(z)-\frac{C_1}{m}\le \vphi_m(z).
 \end{equation} 
 Moreover $\vphi_m$ converges to $\vphi$ pointwise and in $L^1_{\rm loc}$ topology on $U$ as $m\rightarrow+\infty$. 
\end{thm}
More generally, on a projective manifold $X$ for any psh Hermitian metric $h_0 e^{-\vphi}$ on a line bundle $L$, one can define multiplier ideal sheaf. 
Boucksom-Favre-Jonsson gave a valuative description of multiplier ideals, which can be strengthened by using Demailly's strong openness conjecture proved by Guan-Zhou:
\begin{thm}\label{thm-multiplier}
$f\in \cJ(m\vphi)$ if and only if there exists $\epsilon>0$ such that $A_X(v)+v(f)-(1+\epsilon) v(\vphi)>0$ for any $v\in X^{\rm div}_\bQ$. 
\end{thm}
Here $v(\vphi)$ is called the generic Lelong number of $\vphi$ with respect to $v$: if $v=\ord_E$ for a smooth divisor $E$ on $Y\rightarrow X$. Then
$v(\vphi)=\sup\{C>0; \vphi\le C \log |z_1|^2+O(1)\}$ where $z_1$ is any local coordinate satisfying $E=\{z_1=0\}$. 

Now let $\Phi$ be a psh ray which is a psh Hermitian metric on $p_1^*(-K_X)$. Let $v\in X^{\rm div}_\bQ$, $G(v)$ be the Gauss extension which is the unique $\bC^*$-invariant valuation on $X\times\bC$ satisfying $G(v)(t)=1$ and $G(v)|_{\bC(X)}=v$. 
By using this valuative description, Berman-Boucksom-Jonsson derives a valuative formula for $\bfL'^\infty(\Phi)$:
\begin{equation}\label{eq-LPhival}
\bfL'^\infty(\Phi)=\inf_{v\in X^{\rm div}_\bQ} (A_X(v)-G(v)(\Phi)). 
\end{equation}
Multiplier ideals satisfy the important a (Nadel) vanishing and a global generation property. The version we need is the following 
\begin{thm}
Let $p_1: X\times\bB_1\rightarrow X$ be the projection. 
Assume that $p_1^*L$ is equipped with a singular psh Hermtian metric $h=p_1^*h_0 e^{-\Phi}$ such that the Lelong number of $\Phi$ is zero outside $X\times \{0\}$. Let $\cJ(m\Phi)$ be the multiplier ideal sheaf of $h$. Then there exists $m_0>1$ such that for any $m\in \bN$, $\cO_{X\times \bC}((m+m_0)p_1^*L)\otimes \cJ(m\Phi)$ is globally generated.  
\end{thm}
This can be proved by using Siu's proof of global generation in \cite{Siu98}. See also \cite[Corollary 11.2.13]{Laz04} for a corresponding result in the algebraic case.
With this result, Berman-Boucksom-Jonsson constructed test configurations to approximate geodesic rays: set
$\mu_m: \mcX_m\rightarrow X\times\bC$ the blowup of $\mcJ(m\Phi)$ and $\mcL_m=p_1^*L_\bC-\frac{1}{m+m_0}E_m$ where $\mu_m^*\mcJ(m\Phi)=\mcO_{\mcX_m}(-E_m)$. We have the inequality:
\begin{equation}
w(\cJ(m\Phi))\le w(m\Phi)\le w(\cJ(m\Phi))+A(v).
\end{equation}
The first inequality follows from \eqref{eq-vphile} and the second inequality follows from Theorem \ref{thm-multiplier}.
This easily implies, with $\phi_m=\phi_{(\mcX_m, \mcL_m)}$
\begin{equation}
\lim_{m\rightarrow+\infty}\phi_m(v)=-G(v)(\Phi)=:\Phi_\NA(v)\quad \text{ for any } v\in X^{\rm div}_\bQ
\end{equation} 
and
\begin{equation}
\lim_{m\rightarrow+\infty} \bfL^\NA(\phi_m)=\inf_{v\in X^{\rm div}_\bQ}(A_X(v)-G(v)(\Phi))=\bfL'^\infty(\Phi).
\end{equation}
Finally we sketch the proof of the second convergence in \eqref{eq-convLNA}. Recall that:
\begin{eqnarray}\label{eq-infvh}
f_\epsilon:=\bfL^\NA_{(X', B_\epsilon)}(\phi_{\epsilon})&=&\inf_{v\in X^{\rm div}_\bQ}h_{\epsilon}(v) 
\end{eqnarray}
where
\begin{eqnarray*}
h_\epsilon(v)&=&A_{(X', B_\epsilon)}(v)-G(v)(\Phi_\epsilon)\\
&=&A_{X'}(v)-v(B_0)-\frac{\epsilon}{1+\epsilon}v(E_\theta)-\frac{1}{1+\epsilon} G(v)(\Phi).
\end{eqnarray*}
We need to show that $\lim_{\epsilon\rightarrow 0}f_\epsilon=f_0$. Note that $h_\epsilon\rightarrow h_0$ as $\epsilon\rightarrow 0$. 
Because $\mcJ(\Phi)$ is co-supported in $X\times\{0\}$, there exists $k\ge 1$ such that $t^k\in \mcJ(\Phi)$. So by Theorem \ref{thm-multiplier}, there exists $\tau>0$ such that $A_{X'\times\bC}(G(v))-(1+\tau)G(v)(\Phi)+k>0$ which implies $G(v)(\Phi)\le \frac{A(v)+k}{1+\tau}$. So we get
Then we can estimate: for any $v\in X^{\rm div}_\bQ$, 
\begin{eqnarray*}
h_\epsilon(v)&=&A_{X'}(v)-v(B_0)-G(v)(\Phi)+\frac{\epsilon}{1+\epsilon}G(v)(\Phi)\\
&\le& h_0(v)+\frac{\epsilon}{1+\epsilon} \frac{A(v)+k}{1+\tau}.
\end{eqnarray*}
Now the key is to show that the infimum in \eqref{eq-infvh} can be taken over the set of $v\in X^{\rm div}_\bQ$ with $A_{X'}(v)$ uniformly bounded (with the bound independent of $\epsilon$). This can be proved by using Theorem \ref{thm-multiplier} again. The lower bound of $h_\epsilon$ by $h_0$ can be proved similarly.

\vskip 3mm

\noindent
Department of Mathematics, Rutgers University, Piscataway, NJ 08854-8019.

\noindent
{\it E-mail address:} chi.li@rutgers.edu

\vskip 2mm

\end{document}